\documentclass[10pt, twoside]{amsart}
\title{Comodule Algebras and 2-Cocycles over the (Braided) Drinfeld Double}
\date{\today}
\author{Robert Laugwitz}
\address{Department of Mathematics, Rutgers University,
Hill Center, 110 Frelinghuysen Road,
Piscataway, NJ 08854-8019}
\email{robert.laugwitz@rutgers.edu}
\urladdr{https://www.math.rutgers.edu/~rul2/}

\usepackage{mathabx}
\usepackage{import}
\usepackage{amsmath}
\usepackage{amsfonts}
\usepackage{amsthm}
\usepackage{amssymb}
\usepackage[alphabetic, initials]{amsrefs}
\usepackage[english]{babel}
\usepackage{url}
\usepackage{fancyhdr}
\usepackage{graphicx}
\usepackage[
colorlinks=true,
linkcolor=black, 
anchorcolor=black,
citecolor=black,
urlcolor=black, 
]{hyperref}
\usepackage{geometry}

\newcommand{\leftexp}[2]{{\vphantom{#2}}^{#1}{#2}}
\newcommand{\leftexpsub}[3]{{\vphantom{#3}}^{#1}_{#2}{#3}}
\newcommand{\lYD}[1]{\leftexpsub{#1}{#1}{\mathbf{YD}}}
\newcommand{\rYD}[1]{\mathbf{YD}^{#1}_{#1}}

\newcommand{\op}[1]{\operatorname{#1}}
\newcommand{\oop}{\operatorname{op}}
\newcommand{\cop}{\operatorname{cop}}
\newcommand{\ov}[1]{\overline{#1}}
\newcommand{\un}[1]{\underline{#1}}
\newcommand{\lmod}[1]{#1\text{-}\mathbf{Mod}}
\newcommand{\rmod}[1]{\mathbf{Mod}\text{-}#1}
\newcommand{\lcomod}[1]{#1\text{-}\mathbf{CoMod}}
\newcommand{\rcomod}[1]{\mathbf{CoMod}\text{-}#1}

\newcommand{\coev}{\operatorname{coev}}

\newcommand{\Drin}{\operatorname{Drin}}
\newcommand{\ev}{\operatorname{ev}}

\newcommand{\Heis}{\operatorname{Heis}}

\newcommand{\Hom}{\operatorname{Hom}}
\newcommand{\ide}{\operatorname{Id}}

\newcommand{\Nat}{\operatorname{Nat}}

\newcommand{\triv}{\operatorname{triv}}

\newcommand{\Vect}{\mathbf{Vect}_\Bbbk}

\providecommand{\fr}[1]{\mathfrak{#1}}

\providecommand{\op}[1]{\operatorname{#1}}

\newcommand{\mC}{\mathbb{C}}

\newcommand{\mZ}{\mathbb{Z}}

\newcommand{\mF}{\mathbb{F}}

\newcommand{\cC}{\mathcal{C}}

\newcommand{\cB}{\mathcal{B}}

\newcommand{\cO}{\mathcal{O}}

\newcommand{\cT}{\mathcal{T}}

\newcommand{\cM}{\mathcal{M}}

\newcommand{\cZ}{\mathcal{Z}}

\newcommand{\rF}{\mathrm{F}}

\newtheoremstyle{mystyle}
  {0.5cm}                   
  {0.5cm}                   
  {\normalfont}           
  {}                      
  {\itfont\bfseries}  
  {:}                     
  {0.3cm}              
  {\thmname{#1}}

\newtheoremstyle{defstyle}
  {0.5cm}                   
  {0.5cm}                   
  {\normalfont}           
  {}     
  {\normalfont\bfseries}  
  {:}                     
  {0.3cm}              
  {\thmname{#1}\thmnumber{ #2}\thmnote{ (#3)}}

\numberwithin{equation}{section}
\newtheorem{theorem}{Theorem}[section]

\newtheorem{proposition}[theorem]{Proposition}
\newtheorem{corollary}[theorem]{Corollary}
\newtheorem{lemma}[theorem]{Lemma}

\newtheorem{theorem*}{Theorem}

\theoremstyle{definition}
\newtheorem{definition}[theorem]{Definition}

\theoremstyle{remark}
\newtheorem{example}[theorem]{Example}
\newtheorem{remark}[theorem]{Remark}

\renewcommand{\sectionmark}[1]		
	{
	\markboth{\small\it \thesection{} #1}{}
	}

\subjclass[2010]{Primary 16T05; Secondary 17B37, 20G42, 18D10}
\keywords{Braided Hopf algebras, comodule algebras, 2-cocycles, Drinfeld double, quantum groups, double bosonization}

\begin{document}

\begin{abstract}
We show that for dually paired bialgebras, every comodule algebra over one of the paired bialgebras gives a comodule algebra over their Drinfeld double via a crossed product construction. These constructions generalize to working with bialgebra objects in a braided monoidal category of modules over a quasitriangular Hopf algebra. Hence two ways to provide comodule algebras over the braided Drinfeld double (the double bosonization) are provided. Furthermore, a map of second Hopf algebra cohomology spaces is constructed. It takes a pair of 2-cocycles over dually paired Hopf algebras and produces a 2-cocycle over their Drinfeld double. This construction also has an analogue for braided Drinfeld doubles.
\end{abstract}
\maketitle



\section{Introduction}

\subsection{Motivation and Background}

The \emph{Drinfeld double} $\Drin(H)$ of a Hopf algebra $H$ from \cite{Dri} appears as part of the structures used in algebraic approaches to constructing different types of $3$-dimensional topological quantum field theories (TQFTs). For example, Dijkgraaf--Witten TQFTs involve the Drinfeld double $\Drin(G)$ of a group algebra $\Bbbk G$ \cites{DW,DPR}. 
Further, Reshe\-ti\-khin--Turaev type TQFTs, defined for a modular tensor category, are of particular interest in the case of the quantum groups $U_q(\fr{g})$ \cite{RT}. These quantum groups can be obtained from a \emph{braided} version of the Drinfeld double (the \emph{double bosonization}) of its nilpotent part \cite{Maj3}.

Related to the philosophy of \emph{categorification}, which aims to construct $4$-dimensional TQFTs by lifting the algebraic structures used in the construction of $3$-dimensional theories to categories \cite{CF}, modules over monoidal categories $\cM$ have been studied. A detailed account of this theory is given in \cite{EGNO}*{Chapter~7}. A left module over $\cM$ is a category $\cC$ admitting a $\Bbbk$-linear bifunctor 
$\triangleright \colon \cM\times \cC\longrightarrow\cC,$
such that the module axioms hold only up to natural isomorphisms which satisfy coherences. 

The category of modules over the Drinfeld double of a bialgebra $B$ is a braided monoidal category,  which is equivalent to the \emph{center} $\cZ(\lmod{B})$ of the category $\lmod{B}$ of modules over $B$. A theorem of \cites{Ost, EGNO} gives, for a finite tensor-category $\cM$, a categorical Morita equivalence
$\lmod{\cM\boxtimes \cM^{\oop}} \stackrel{\sim}{\longrightarrow} \lmod{\cZ(\cM)}.$
On the left hand side $\cM$-bimodules appear, while the right hand side consists of categorical modules over the monoidal center of $\cM$.

More generally, bialgebras (or Hopf algebras) $B$ can be defined in a braided monoidal category $\cB$ and are sometimes called \emph{braided bialgebras} (respectively, \emph{braided Hopf algebras}), see e.g. \cites{Maj1,AS}. Modules over such a braided bialgebra again form a monoidal category, denoted by $\lmod{B}(\cB)$. 
A central motivation for this paper is to contribute to understanding categorical modules not over the whole category $\cZ(\cM)$, for $\cM=\lmod{B}(\cB)$, but over the a subcategory, called the \emph{relative} monoidal center $\cZ_{\cB}(\cM)\subseteq \cZ(\cM)$ defined on the categorical level in \cites{Lau,Lau2}. For the framework of this paper, it is sufficient to view $\cZ_{\cB}(\cM)$ as the equivalent category $\lYD{B}(\cB)$ of \emph{Yetter--Drinfeld modules} (also called \emph{crossed modules}) over $B$ within $\cB$ of \cites{Bes, BD}.


Working with braided Hopf algebras has become a productive point of view in quantum algebra (see e.g. \cites{Maj8,Maj6, Lub, AS} and other papers). A prominent example is given by the positive part $U_q(\fr{n}^+)$ of the quantum group $U_q(\fr{g})$ associated to a Kac--Moody algebra $\fr{g}$ \cites{Dri,Jim,Lus}. The algebra $U_q(\fr{n}^+)$ is not a Hopf algebra over $\mC(q)$, but a Hopf algebra object in the braided monoidal category of comodules over a lattice, with a special, non-symmetric, braiding obtained from the parameter $q$ and the Cartan datum defining $\fr{g}$ \cites{Maj3,AS}.

The relative monoidal center of $U_q(\fr{n}^+)$-modules is equivalent to the braided monoidal category of highest weight modules over $U_q(\fr{g})$ \cite{Lau2}. To study modules over the quantum group $U_q(\fr{g})$, it is often reasonable to assume certain restrictions on the modules. For example, one can study the category consisting of \emph{highest weight modules}, which are semisimple with integral weights, such that $U_q(\fr{g}^+)$ acts locally finitely. In addition, finite generation as $U_q(\fr{g}^+)$-modules can be assumed, giving an analogue of the category $\cO$ \cite{BGG} for quantum groups \cite{AM}. However, the class of finitely generated modules is not closed under tensor products.
This point of view that the relative center is a natural category to study holds for a larger class of braided Hopf algebras of infinite dimension \cite{Lau}*{Section 3.7}. This relative version of the center turns out to be related to the \emph{braided Drinfeld double}, called \emph{double bosonization} in \cite{Maj2}. 

Another structure of interest in this paper are bialgebra $2$-cocycles. These are a part of a generalization of Sweedler's bialgebra cohomology \cite{Swe} with coefficients in $\Bbbk$ to non-cocommutative bialgebras \cite{Maj4}.
Twists by bialgebra $2$-cocycles are a fundamental tool in the deformation theory leading to quantum groups, and the study of the larger class of pointed Hopf algebras. The idea goes back to \cite{Dri}, where a Hopf algebra can be given a twisted coproduct by conjugation with a $2$-cycle, and \cite{DT2} for the twist of the algebra structure by conjugation with a $2$-cocycle, see also \cites{AS, Mas, AS2}. In this paper, we give an explicit map inducing bialgebra $2$-cocycles on the (braided) Drinfeld double from the datum of bialgebra $2$-cocycles on $B$ and its dual. Such $2$-cocycles have been classified for the Drinfeld double and similarly constructed bicross product algebras by Schauenburg \cites{Sch2,Sch3}. Other results on inducing so-called \emph{lazy} $2$-cocycles on the Drinfeld double are found in \cites{BC,CP}.

\subsection{Summary}

Let $\cB$ be a braided monoidal $\Bbbk$-linear category, and $A$ a bialgebra in $\cB$. Further let $A$ be a $B$-comodule algebra in $\cB$ (see Section \ref{ydtensoractions}). The first main result of the paper (Theorem \ref{mainthm}) is that there exists a left categorical action  
$$\triangleright \colon \lYD{B}\times \lmod{A}(\lcomod{B})\longrightarrow \lmod{A}(\lcomod{B}),$$
of the monoidal category $\lYD{B}$ of left Yetter--Drinfeld modules over $B$ from \cites{Bes,BD} (cf. Definition \ref{YDmodules}). A similar result holds given a right $B$-module algebra $A$, see Theorem \ref{mainthm2}.
These actions generalize the induced $A$-module structure in a non-trivial way (cf. Example \ref{trivialcoaction}). In the special case where $A=B$ with the regular coaction, the result gives an action of $\lYD{B}$ on \emph{Hopf modules} over $B$ (Example \ref{regularcoaction}) similar to \cites{Lu,Lau}. If $B=H$ is a commutative Hopf algebra, using $A=H$ with the adjoint coaction, the functor $\triangleright$ corresponds, under equivalence, to the tensor product of $\lYD{H}$ (see Example \ref{adjointcoaction}).

In Section \ref{algebrapicture}, we slightly generalize these results to produce comodule algebra over a braided version of the Drinfeld double, (the double bosonization of \cites{Maj3,Maj1}), which we denote by $\Drin_H(C,B)$. For this, let $\cB=\lmod{H}$ where $H$ is a quasitriangular Hopf algebra over a field $\Bbbk$, and let $C,B$ be bialgebras with a bialgebra pairing $\ev\colon C\otimes B\to \Bbbk$ in $\cB$. This way, there is a monoidal functor
$$\Psi\colon \lYD{B}\longrightarrow \lmod{\Drin_H(C,B)},$$
see Proposition \ref{drinfeldprop}. The setup allows for infinite-dimensional bialgebras $C,B$, for which the latter category may be larger. If the pairing $\ev$ is non-degenerate, then $\Psi$ is fully faithful.

The main results of  Section \ref{algebrapicture} are:
\begin{itemize}
\item Let $A$ be a left $B$-comodule algebra in $\cB=\lmod{H}$. Then $A\rtimes \leftexp{\cop}{C}\rtimes H$ is a left $\Drin_H(C,B)$-comodule algebra, see Corollary \ref{comodulealgebracor}.
\item Let $A$ be a right $C$-comodule algebra in $\cB=\lmod{H}$. Then $A\rtimes B\rtimes H$ is a left $\Drin_H(C,B)$-comodule algebra, see Corollary \ref{comodulealgebracor2}.
\end{itemize}
A notable example is that the braided Heisenberg double is a comodule algebra over the braided Drinfeld double $\Drin_H(C,B)$ (Example \ref{comodulealgebracor}). This case recovers a result of \cite{Lau} generalizing \cite{Lu}. 

Weakening the assumption of $H$ being quasitriangular to that of a \emph{weak quasitriangular structure} of \cite{Maj3}, cf. Definition \ref{weakquasi}, the above results can be applied to Lusztig's version of the quantum groups $U_q(\fr{g})$. Namely, for $B=U_q(\fr{n}^+)$, the nilpotent part of the quantum group, its braided Drinfeld double with $U_q(\fr{n}^-)$ is isomorphic as a Hopf algebra to $U_q(\fr{g})$ by \cite{Maj3}. The results of this paper add that \emph{any} $U_q(\fr{n}^+)$-comodule algebra can be used to produce a $U_q(\fr{g})$-comodule algebra. In particular, the regular comodule algebra gives an algebra $D_q(\fr{g})$ containing the quantum Weyl algebra  of \cite{Jos}, cf. Corollary \ref{quantumweyl}.

In Section \ref{2cocycles}, we investigate the construction of bialgebra $2$-cocycles over the (braided) Drinfeld double from $2$-cocycles over weakly dually paired Hopf algebras. In the most general form (Theorem \ref{general2cocycles}), there is a map 
$$H^2_H(B,\Bbbk)\times H^2_H(\leftexp{\cop}{C},\Bbbk)\longrightarrow H^2(\Drin_{H}(C,B),\Bbbk),\qquad (\sigma,\tau) \longmapsto \sigma\circ\tau,$$
where $C, B$ are dually paired Hopf algebra objects in the braided monoidal category $\lmod{H}$ for a quasitriangular Hopf algebra $H$.
The subscripts $H$ here indicate that these $2$-cocycles live in the braided monoidal category $\lmod{H}$ (see Section \ref{braidedcycles} for basic facts of a theory of such braided $2$-cocycles). Such $2$-cocycles in $\cB$ were defined in \cite{BD1} and relate to $2$-cocycles of a monoidal category as discussed in \cites{PSO,BB} in general.

The $\Drin_H(C,B)$-comodule algebras constructed in Section \ref{algebrapicture} relate to the map $$\op{Ind}_B\colon H^2_H(B,\Bbbk)\longrightarrow H^2(\Drin_{H}(C,B),\Bbbk), \qquad \sigma \longmapsto \sigma\circ \triv_C$$
by means of constructing one-sided twists, or \emph{cleft objects} (see \cites{DT,Sch4, Mas}).
If $B_\sigma$ denotes the one-sided twist of $B$ by $\sigma$, then $B_\sigma \rtimes \leftexp{\cop}{C}\rtimes H$ is the one-sided twist of $\Drin_H(C,B)$ by $\op{Ind}_B(\sigma)$ using the $\Drin_H(C,B)$-comodule algebra structure from Theorem \ref{maincor}. In particular, using the trivial $2$-cocycle $\sigma=\triv_B$, we recover the result of \cite{Lau} that the braided Heisenberg double $\Heis_H(C,B)$ is a one-sided $2$-cocycle twist of $\Drin_H(C,B)$ (see Example \ref{heistwist}). This is a generalization of an earlier result of \cite{Lu}.
An example is that the $q$-Weyl algebra $D_q(\fr{g})$ is a $2$-cocycle twist of $U_q(\fr{g})$.

To conclude the summary, note that the categorical actions from Theorem \ref{mainthm} and \ref{mainthm2} fit into a more general categorical picture \cite{Lau2} of constructing categorical modules over the relative monoidal center. We further note that it would be interesting to construct more explicit examples of comodule algebras over the class of braided Hopf algebras given by Nichols algebras (see e.g. \cites{AS,AS2}). In the literature of Nichols algebras over a group, examples of such comodule algebras can be given using twists by $2$-cocycles which are typically defined over the bosonization (Radford biproduct \cite{Rad}) of the braided Hopf algebra by twisting the group part, see e.g. \cite{Mas}.

\subsection{Setup}\label{setup}

In this paper, $\cB$ is a braided monoidal category with braiding $\Psi$. We assume that $\cB$ is $\Bbbk$-linear (or even abelian), over $\Bbbk$ a field. As a prototype example, we may think of $\cB$ as a category $\lmod{H}$ for $H$ a quasitriangular Hopf algebra over $\Bbbk$, or $\lcomod{H}$ for $H$ dual quasitriangular (see e.g. \cite{Maj6}*{Section 1.3}). All categories considered are assumed to be equivalent to small categories.

Let $A$ be an algebra object in $\cB$, with associative multiplication $m\colon A\otimes A\to A$ and unit $1\colon I\to A$. We can then consider the category $\lmod{A}=\lmod{A}(\cB)$ of left $A$-modules in $\cB$. Here, a left $A$-module is an object $V$ of $\cB$ together with a morphism $a\colon A\otimes V\to V$ in $\cB$ satisfying the usual module axioms:
\begin{align}
a(m\otimes \ide_V)&=a(\ide_A\otimes a), &&a(1\otimes\ide_V)=\ide_V.
\end{align}
This category is again $\Bbbk$-linear (and abelian if $\cB$ is abelian). Analogously, we denote by $\rmod{A}$, $\lcomod{A}$, and $\rcomod{A}$ the categories of \emph{right} $A$-modules, left $A$-comodules, respectively right $A$-comodules. Although these categories consists of objects in $\cB$ and all morphisms are in $\cB$, we often suppress $\cB$ in the notation of these categories.

Unless otherwise specified, $B$ denotes a bialgebra object in $\cB$. That is, $B$ has an algebra structure $m\colon B\otimes B\to B$, $1\colon I \to B$, and a coalgebra structure $\Delta\colon B\to B\otimes B$, $\varepsilon\colon B\to I$, such that the following axioms are satisfied:
\begin{gather}
m(m\otimes \ide_B)=m(\ide_B\otimes m),\\
m(1\otimes \ide_B)=m(\ide_B\otimes 1)=\ide_B,\\
(\Delta\otimes \ide_B)\Delta=(\ide_B\otimes \Delta)\Delta,\\
(\varepsilon\otimes \ide_B)\Delta=(\ide_B\otimes \varepsilon)\Delta=\ide_B,\\
\Delta m=m\otimes m(\ide_B\otimes \Psi_{B,B}\otimes \ide_B)(\Delta\otimes \Delta),\label{bialgebra}\\
\Delta 1=1\otimes 1,\\
\varepsilon m=\varepsilon\otimes \varepsilon.
\end{gather}
Here, the last three axioms say that $\Delta, \varepsilon$ are morphisms of algebras (equivalently, $m$ and $1$ are morphisms of coalgebras). Morphisms of bialgebras are morphisms in $\cB$ which commute with all these structures. If, in addition, there exists a morphism $S\colon B\to B$ in $\cB$ which is a two-sided convolution inverse to $\ide_B$, i.e.
\begin{align}
m(\ide_B\otimes S)\Delta&=m(S\otimes \ide_B)\Delta=1\varepsilon,
\end{align}
then we say $B$ is a \emph{Hopf algebra} in $\cB$. We assume that all Hopf algebras have an invertible antipode. Morphisms of bialgebras between Hopf algebras commute with the antipodes.

Further, we say that $B$ is \emph{commutative} in $\cB$ if $m=m\Psi$, which is equivalent to $m=m\Psi^{-1}$.

Graphical calculus, similar to that used in \cite{Maj6}, is helpful for computations in the monoidal category $\lmod{B}(\cB)$. For example, the $B$-module structure $\triangleright_{V\otimes W}$ on the tensor product $V\otimes W$ of left $B$-modules $V$, $W$ is depicted as
\begin{equation}\label{tensorproductaction}\triangleright_{V\otimes W}= (\triangleright_V\otimes \triangleright_W)(\ide_B\otimes \Psi_{B,V}\otimes \ide_W)(\Delta\otimes \ide_{V\otimes W})=\vcenter{\hbox{
\begingroup%
  \makeatletter%
  \providecommand\color[2][]{%
    \errmessage{(Inkscape) Color is used for the text in Inkscape, but the package 'color.sty' is not loaded}%
    \renewcommand\color[2][]{}%
  }%
  \providecommand\transparent[1]{%
    \errmessage{(Inkscape) Transparency is used (non-zero) for the text in Inkscape, but the package 'transparent.sty' is not loaded}%
    \renewcommand\transparent[1]{}%
  }%
  \providecommand\rotatebox[2]{#2}%
  \ifx\svgwidth\undefined%
    \setlength{\unitlength}{30.44988496bp}%
    \ifx\svgscale\undefined%
      \relax%
    \else%
      \setlength{\unitlength}{\unitlength * \real{\svgscale}}%
    \fi%
  \else%
    \setlength{\unitlength}{\svgwidth}%
  \fi%
  \global\let\svgwidth\undefined%
  \global\let\svgscale\undefined%
  \makeatother%
  \begin{picture}(1,0.88991524)%
    \put(0,0){\includegraphics[width=\unitlength,page=1]{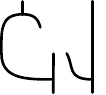}}%
    \put(0.5005321,0.47717664){\color[rgb]{0,0,0}\makebox(0,0)[lb]{\smash{$~$}}}%
    \put(0,0){\includegraphics[width=\unitlength,page=2]{tensoraction.pdf}}%
  \end{picture}%
\endgroup%
}}.\end{equation}
Here, $\Psi=\vcenter{\hbox{
\begingroup%
  \makeatletter%
  \providecommand\color[2][]{%
    \errmessage{(Inkscape) Color is used for the text in Inkscape, but the package 'color.sty' is not loaded}%
    \renewcommand\color[2][]{}%
  }%
  \providecommand\transparent[1]{%
    \errmessage{(Inkscape) Transparency is used (non-zero) for the text in Inkscape, but the package 'transparent.sty' is not loaded}%
    \renewcommand\transparent[1]{}%
  }%
  \providecommand\rotatebox[2]{#2}%
  \ifx\svgwidth\undefined%
    \setlength{\unitlength}{8.22109334bp}%
    \ifx\svgscale\undefined%
      \relax%
    \else%
      \setlength{\unitlength}{\unitlength * \real{\svgscale}}%
    \fi%
  \else%
    \setlength{\unitlength}{\svgwidth}%
  \fi%
  \global\let\svgwidth\undefined%
  \global\let\svgscale\undefined%
  \makeatother%
  \begin{picture}(1,1.0011064)%
    \put(0,0){\includegraphics[width=\unitlength,page=1]{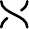}}%
  \end{picture}%
\endgroup%
}}$ denotes the braiding in $\cB$, $\Delta=\vcenter{\hbox{
\begingroup%
  \makeatletter%
  \providecommand\color[2][]{%
    \errmessage{(Inkscape) Color is used for the text in Inkscape, but the package 'color.sty' is not loaded}%
    \renewcommand\color[2][]{}%
  }%
  \providecommand\transparent[1]{%
    \errmessage{(Inkscape) Transparency is used (non-zero) for the text in Inkscape, but the package 'transparent.sty' is not loaded}%
    \renewcommand\transparent[1]{}%
  }%
  \providecommand\rotatebox[2]{#2}%
  \ifx\svgwidth\undefined%
    \setlength{\unitlength}{16.80101968bp}%
    \ifx\svgscale\undefined%
      \relax%
    \else%
      \setlength{\unitlength}{\unitlength * \real{\svgscale}}%
    \fi%
  \else%
    \setlength{\unitlength}{\svgwidth}%
  \fi%
  \global\let\svgwidth\undefined%
  \global\let\svgscale\undefined%
  \makeatother%
  \begin{picture}(1,0.76387957)%
    \put(0,0){\includegraphics[width=\unitlength]{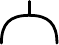}}%
  \end{picture}%
\endgroup%
}}$ denotes the coproduct of $B$, and $\triangleright=\vcenter{\hbox{
\begingroup%
  \makeatletter%
  \providecommand\color[2][]{%
    \errmessage{(Inkscape) Color is used for the text in Inkscape, but the package 'color.sty' is not loaded}%
    \renewcommand\color[2][]{}%
  }%
  \providecommand\transparent[1]{%
    \errmessage{(Inkscape) Transparency is used (non-zero) for the text in Inkscape, but the package 'transparent.sty' is not loaded}%
    \renewcommand\transparent[1]{}%
  }%
  \providecommand\rotatebox[2]{#2}%
  \ifx\svgwidth\undefined%
    \setlength{\unitlength}{9.12051093bp}%
    \ifx\svgscale\undefined%
      \relax%
    \else%
      \setlength{\unitlength}{\unitlength * \real{\svgscale}}%
    \fi%
  \else%
    \setlength{\unitlength}{\svgwidth}%
  \fi%
  \global\let\svgwidth\undefined%
  \global\let\svgscale\undefined%
  \makeatother%
  \begin{picture}(1,1.40715315)%
    \put(0,0){\includegraphics[width=\unitlength]{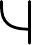}}%
    \put(0.13167899,0.48417802){\color[rgb]{0,0,0}\makebox(0,0)[lb]{\smash{$~$
}}}%
  \end{picture}%
\endgroup%
}}\colon B\otimes V\to V$ denotes a left action of $B$ on $V$.
To give another example of such graphical calculus, the bialgebra condition (\ref{bialgebra}) can be displayed as:
\begin{equation}\vcenter{\hbox{
\begingroup%
  \makeatletter%
  \providecommand\color[2][]{%
    \errmessage{(Inkscape) Color is used for the text in Inkscape, but the package 'color.sty' is not loaded}%
    \renewcommand\color[2][]{}%
  }%
  \providecommand\transparent[1]{%
    \errmessage{(Inkscape) Transparency is used (non-zero) for the text in Inkscape, but the package 'transparent.sty' is not loaded}%
    \renewcommand\transparent[1]{}%
  }%
  \providecommand\rotatebox[2]{#2}%
  \ifx\svgwidth\undefined%
    \setlength{\unitlength}{75.57609773bp}%
    \ifx\svgscale\undefined%
      \relax%
    \else%
      \setlength{\unitlength}{\unitlength * \real{\svgscale}}%
    \fi%
  \else%
    \setlength{\unitlength}{\svgwidth}%
  \fi%
  \global\let\svgwidth\undefined%
  \global\let\svgscale\undefined%
  \makeatother%
  \begin{picture}(1,0.40732931)%
    \put(0.31497865,0.20559355){\color[rgb]{0,0,0}\makebox(0,0)[lb]{\smash{ }}}%
    \put(0,0){\includegraphics[width=\unitlength,page=1]{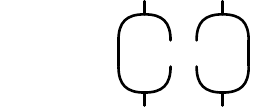}}%
    \put(0.25222581,0.17800578){\color[rgb]{0,0,0}\makebox(0,0)[lb]{\smash{$=$}}}%
    \put(0,0){\includegraphics[width=\unitlength,page=2]{bialgebra.pdf}}%
    \put(0.79876256,0.15420241){\color[rgb]{0,0,0}\makebox(0,0)[lb]{\smash{$~$}}}%
    \put(0,0){\includegraphics[width=\unitlength,page=3]{bialgebra.pdf}}%
  \end{picture}%
\endgroup%
}}.\end{equation}

\begin{lemma}\label{weirdcross}
For $B$ a commutative (or cocommutative) bialgebra in $\cB$, $B$ is also a bialgebra in $\overline{\cB}=(\cB,\Psi^{-1})$, the braided monoidal category with \emph{inverse} braiding.
\end{lemma}
\begin{proof}We have to verify the bialgebra condition using the inverse braiding $\Psi^{-1}$ instead of the braiding $\Psi$. For $B$ commutative in $B$, this amounts to the following computation:
\begin{align*}
\Delta m&=\Delta m\Psi^{-1}=(m\otimes m)\Delta_{B\otimes B}\Psi^{-1}\\
&=(m\Psi^{-1}\otimes m\Psi^{-1})(\ide_B\otimes \Psi^{-1}\otimes \ide_B)(\Delta\otimes \Delta)\\
&=(m\otimes m)(\ide_B\otimes \Psi^{-1}\otimes \ide_B)(\Delta\otimes \Delta).
\end{align*}
The proof for $B$ cocommutative is similar; it follows dually.
\end{proof}
In fact, it was shown in \cite{Sch} that for a cocommutative Hopf algebra in $\cB$, it follows that $\Psi_{H,H}^2=\ide_{H\otimes H}$, i.e. the braiding on $H$ is \emph{symmetric}.

For $A$ an algebra and $C$ a coalgebra in $\cB$, there exists a \emph{convolution product} $\ast$ on $\Hom_{\cB}(C,A)$, where
\begin{align}
\phi\ast \psi &=m_A(\phi\otimes \psi)\Delta_C, &&\forall \phi,\psi\colon C\to A.
\end{align}
Note that $1_A\varepsilon_C$ is the neutral element with respect to $\ast$. If a two-sided inverse for $\phi$ exists, then it is unique, and we denote it by $\phi^{-\ast}$.

Further, recall that a bialgebra $B$ over $\Bbbk$ is \emph{quasitriangular} if there exists a \emph{universal R-matrix} $R\in B\otimes B$ such that
\begin{align}\label{Rmatrix1}
{R^{(1)}}_{(1)}\otimes {R^{(1)}}_{(2)}\otimes {R^{(2)}}&=R_1^{(1)}\otimes R_2^{(1)}\otimes R_1^{(2)}R_2^{(2)},\\\label{Rmatrix2}
R^{(1)}\otimes {R^{(2)}}_{(1)}\otimes {R^{(2)}}_{(2)}&=R_1^{(1)}R_2^{(1)}\otimes R_2^{(2)}\otimes R_1^{(2)},\\\label{Rmatrix3}
R^{(1)}b_{(1)}\otimes R^{(2)}b_{(2)}&=b_{(2)}R^{(1)}\otimes b_{(1)}R^{(2)}, &\forall b\in B,
\end{align}
cf. \cite{Maj1}*{Definition 2.1.1}.
We require $R$ to be convolution invertible.
Given $B$ quasitriangular, the category $\lmod{B}$ is braided with braiding (cf. \cite{Maj1}*{Theorem 9.2.4})
\begin{align}
\Psi_{V,W}(v\otimes w)= (R^{(2)}\triangleright w)\otimes (R^{(1)}\triangleright v), && \forall v\in V, w\in W.
\end{align}
Here, the notation $b\triangleright v$ is used to denote a left $B$-action.

Dually, $B$ is \emph{dual quasitriangular} if there exists a \emph{dual R-matrix} $R\colon B\otimes B\to \Bbbk$ such that
\begin{align}\label{dualR1}
R(ab,c)&=R(a,c_{(1)})R(b,c_{(2)}), \\ \label{dualR2} R(a,bc)&=R(a_{(1)},c)R(a_{(2)},b), \\ \label{dualR3} 
a_{(2)}b_{(2)}R(a_{(1)},b_{(1)})&=b_{(1)}a_{(1)}R(a_{(2)},b_{(2)}),
\end{align}
for all $a,b,c\in B$; cf. \cite{Maj1}*{Definition 2.2.1}.
In this case, the category of left $B$-comodules is braided monoidal with braiding given by
\begin{equation}
\Psi_{V,W}(v\otimes w)= w^{(0)}\otimes v^{(0)}R(w^{(-1)},v^{(-1)}),
\end{equation}
where $\delta(v)=v^{(-1)}\otimes v^{(0)}\in B\otimes V$ is adapted Sweedler's notation for the left $B$-coaction on $V$ (sums are implicit).


\section{A Tensor Product Action by Yetter--Drinfeld Modules}\label{ydtensoractions}

In this section, $B$ is a bialgebra in a braided monoidal category $\cB$ (cf. Section \ref{setup}). Let $A$ be an algebra object in the category $\lcomod{B}$. That is, $A$ has a $B$-comodule structure $\delta\colon A\to B\otimes A$, a multiplication $m_A\colon A\otimes A \to A$, and a unit $1_A\colon I\to A$, such that
\begin{gather}
\delta m_A=(m_B\otimes m_A)(\ide_B\otimes \Psi_{A,B}\otimes \ide_A)(\delta\otimes \delta),\\
\delta 1_A=1_B\otimes 1_A.
\end{gather}
These conditions can be depicted as
\begin{equation}\label{comodalg}\vcenter{\hbox{
\begingroup%
  \makeatletter%
  \providecommand\color[2][]{%
    \errmessage{(Inkscape) Color is used for the text in Inkscape, but the package 'color.sty' is not loaded}%
    \renewcommand\color[2][]{}%
  }%
  \providecommand\transparent[1]{%
    \errmessage{(Inkscape) Transparency is used (non-zero) for the text in Inkscape, but the package 'transparent.sty' is not loaded}%
    \renewcommand\transparent[1]{}%
  }%
  \providecommand\rotatebox[2]{#2}%
  \ifx\svgwidth\undefined%
    \setlength{\unitlength}{141.37503922bp}%
    \ifx\svgscale\undefined%
      \relax%
    \else%
      \setlength{\unitlength}{\unitlength * \real{\svgscale}}%
    \fi%
  \else%
    \setlength{\unitlength}{\svgwidth}%
  \fi%
  \global\let\svgwidth\undefined%
  \global\let\svgscale\undefined%
  \makeatother%
  \begin{picture}(1,0.21843502)%
    \put(0.16843524,0.11068584){\color[rgb]{0,0,0}\makebox(0,0)[lb]{\smash{ }}}%
    \put(0,0){\includegraphics[width=\unitlength,page=1]{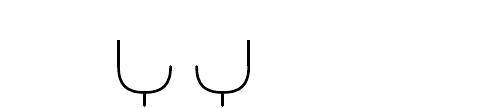}}%
    \put(0.13488891,0.09593801){\color[rgb]{0,0,0}\makebox(0,0)[lb]{\smash{$=$}}}%
    \put(0,0){\includegraphics[width=\unitlength,page=2]{comodalg.pdf}}%
    \put(0.55958177,0.10129679){\color[rgb]{0,0,0}\makebox(0,0)[lb]{\smash{$,$}}}%
    \put(0,0){\includegraphics[width=\unitlength,page=3]{comodalg.pdf}}%
    \put(0.7714935,0.09631749){\color[rgb]{0,0,0}\makebox(0,0)[lb]{\smash{$=$}}}%
    \put(0,0){\includegraphics[width=\unitlength,page=4]{comodalg.pdf}}%
  \end{picture}%
\endgroup%
}},\end{equation}
where $1=\vcenter{\hbox{
\begingroup%
  \makeatletter%
  \providecommand\color[2][]{%
    \errmessage{(Inkscape) Color is used for the text in Inkscape, but the package 'color.sty' is not loaded}%
    \renewcommand\color[2][]{}%
  }%
  \providecommand\transparent[1]{%
    \errmessage{(Inkscape) Transparency is used (non-zero) for the text in Inkscape, but the package 'transparent.sty' is not loaded}%
    \renewcommand\transparent[1]{}%
  }%
  \providecommand\rotatebox[2]{#2}%
  \ifx\svgwidth\undefined%
    \setlength{\unitlength}{4.8bp}%
    \ifx\svgscale\undefined%
      \relax%
    \else%
      \setlength{\unitlength}{\unitlength * \real{\svgscale}}%
    \fi%
  \else%
    \setlength{\unitlength}{\svgwidth}%
  \fi%
  \global\let\svgwidth\undefined%
  \global\let\svgscale\undefined%
  \makeatother%
  \begin{picture}(1,2.56889521)%
    \put(0,0){\includegraphics[width=\unitlength]{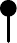}}%
  \end{picture}%
\endgroup%
}}$ denotes the units (of $B$ and $A$). Given $A$, the category $\lmod{A}(\lcomod{B})$ can be defined, consisting of left $A$-modules in $\cB$ such that the $A$-action is a morphism of left $B$-comodules, where $B$ itself has the regular $B$-comodule structure given by $\Delta$.

\subsection{Yetter--Drinfeld Modules}

We recall some facts about the category of Yetter--Drinfeld modules over a bialgebra in $\cB$. Such modules were introduced in \cite{Bes}*{Section 3.3} and \cite{BD} and called \emph{crossed modules} therein.

\begin{definition}\label{YDmodules}
The category $\lYD{B}=\lYD{B}(\cB)$ of left \emph{Yetter--Drinfeld modules} over $B$ in $\cB$ consists of objects $V$ in $\cB$ which are equipped with a left $B$-action $a$ and a left $B$-coaction $\delta$ in $\cB$, such that the following compatibility holds:
\begin{equation}\label{YDcond}
\begin{split}
&(m_B\otimes a)(\ide_B\otimes \Psi_{B,B}\otimes \ide_V)(\Delta\otimes \delta)\\&=(m_B\otimes \ide_V)(\ide_B\otimes \Psi_{V,B})(\delta a\otimes \ide_V)(\ide_B\otimes \Psi_{B,V})(\Delta\otimes \ide_V).
\end{split}
\end{equation}
That is,
\begin{equation}\label{ydpic}\vcenter{\hbox{
\begingroup%
  \makeatletter%
  \providecommand\color[2][]{%
    \errmessage{(Inkscape) Color is used for the text in Inkscape, but the package 'color.sty' is not loaded}%
    \renewcommand\color[2][]{}%
  }%
  \providecommand\transparent[1]{%
    \errmessage{(Inkscape) Transparency is used (non-zero) for the text in Inkscape, but the package 'transparent.sty' is not loaded}%
    \renewcommand\transparent[1]{}%
  }%
  \providecommand\rotatebox[2]{#2}%
  \ifx\svgwidth\undefined%
    \setlength{\unitlength}{83.4452175bp}%
    \ifx\svgscale\undefined%
      \relax%
    \else%
      \setlength{\unitlength}{\unitlength * \real{\svgscale}}%
    \fi%
  \else%
    \setlength{\unitlength}{\svgwidth}%
  \fi%
  \global\let\svgwidth\undefined%
  \global\let\svgscale\undefined%
  \makeatother%
  \begin{picture}(1,0.54781815)%
    \put(0,0.27579781){\color[rgb]{0,0,0}\makebox(0,0)[lb]{\smash{ }}}%
    \put(0,0){\includegraphics[width=\unitlength,page=1]{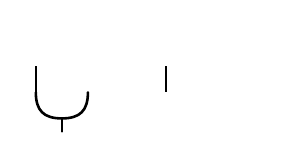}}%
    \put(0.64422368,0.25081165){\color[rgb]{0,0,0}\makebox(0,0)[lb]{\smash{$=$}}}%
    \put(0,0){\includegraphics[width=\unitlength,page=2]{ydcond.pdf}}%
    \put(0.44325035,0.26670016){\color[rgb]{0,0,0}\makebox(0,0)[lb]{\smash{$~$\\  }}}%
    \put(0,0){\includegraphics[width=\unitlength,page=3]{ydcond.pdf}}%
  \end{picture}%
\endgroup%
}}~.\end{equation}
Morphisms in  $\lYD{B}$ are required to commute with both the $B$-action and coaction.
\end{definition}

It was shown in \cite{Bes}*{Theorem 3.4.3} that the category $\lYD{B}$ is a braided monoidal category, where the monoidal structure is given by the tensor product action of Eq. (\ref{tensorproductaction}) and the tensor product coaction given by
\begin{align}
\delta_{V\otimes W}=(m\otimes \ide_{V\otimes W})(\ide_B\otimes \Psi_{V,B}\otimes \ide_W)(\delta_V\otimes \delta_W)=\vcenter{\hbox{
\begingroup%
  \makeatletter%
  \providecommand\color[2][]{%
    \errmessage{(Inkscape) Color is used for the text in Inkscape, but the package 'color.sty' is not loaded}%
    \renewcommand\color[2][]{}%
  }%
  \providecommand\transparent[1]{%
    \errmessage{(Inkscape) Transparency is used (non-zero) for the text in Inkscape, but the package 'transparent.sty' is not loaded}%
    \renewcommand\transparent[1]{}%
  }%
  \providecommand\rotatebox[2]{#2}%
  \ifx\svgwidth\undefined%
    \setlength{\unitlength}{30.44988496bp}%
    \ifx\svgscale\undefined%
      \relax%
    \else%
      \setlength{\unitlength}{\unitlength * \real{\svgscale}}%
    \fi%
  \else%
    \setlength{\unitlength}{\svgwidth}%
  \fi%
  \global\let\svgwidth\undefined%
  \global\let\svgscale\undefined%
  \makeatother%
  \begin{picture}(1,0.88991524)%
    \put(0,0){\includegraphics[width=\unitlength,page=1]{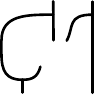}}%
    \put(0.5005321,0.41273921){\color[rgb]{0,0,0}\makebox(0,0)[lb]{\smash{$~$}}}%
    \put(0,0){\includegraphics[width=\unitlength,page=2]{tensorcoaction.pdf}}%
  \end{picture}%
\endgroup%
}}~.
\end{align}
The braiding is given by
\begin{equation}
\Psi^{\mathbf{YD}}_{V,W}=(a\otimes \ide_V)(\ide_B\otimes \Psi_{V,W})(\delta_V\otimes\ide)=\vcenter{\hbox{
\begingroup%
  \makeatletter%
  \providecommand\color[2][]{%
    \errmessage{(Inkscape) Color is used for the text in Inkscape, but the package 'color.sty' is not loaded}%
    \renewcommand\color[2][]{}%
  }%
  \providecommand\transparent[1]{%
    \errmessage{(Inkscape) Transparency is used (non-zero) for the text in Inkscape, but the package 'transparent.sty' is not loaded}%
    \renewcommand\transparent[1]{}%
  }%
  \providecommand\rotatebox[2]{#2}%
  \ifx\svgwidth\undefined%
    \setlength{\unitlength}{32.00101968bp}%
    \ifx\svgscale\undefined%
      \relax%
    \else%
      \setlength{\unitlength}{\unitlength * \real{\svgscale}}%
    \fi%
  \else%
    \setlength{\unitlength}{\svgwidth}%
  \fi%
  \global\let\svgwidth\undefined%
  \global\let\svgscale\undefined%
  \makeatother%
  \begin{picture}(1,0.77737724)%
    \put(0,0){\includegraphics[width=\unitlength]{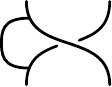}}%
  \end{picture}%
\endgroup%
}}.
\end{equation}

\begin{lemma}
If $H$ is a Hopf algebra in $\cB$, Equation (\ref{ydpic}) is equivalent either one of the following equations 
\begin{gather}\begin{split}
&(m\otimes \ide_V)(\ide_H\otimes \delta)(S\otimes a)(\Delta\otimes \ide_V)\\&=(m\otimes \ide_V)(\ide\otimes \Psi_{V,H})(\ide\otimes a\otimes\ide)(\Psi_{H,H}\otimes \ide_V\otimes S)(\ide_H\otimes (\delta\otimes \ide_H)\Psi_{H,V})(\Delta\otimes \ide_V)
\end{split}\label{altYD1}\\
\delta a=(m\otimes \ide)(\ide\otimes\Psi_{V,H})(m\otimes a\otimes \ide_H)(\ide_H\otimes \Psi_{H,H}\otimes \ide)(\Delta\otimes (\delta\otimes S)\Psi_{H,V})(\Delta\otimes \ide_V).\label{altYD2}
\end{gather}
\end{lemma}
\begin{proof}
These equations are depicted as 
\begin{align}\label{ydpic2}
\vcenter{\hbox{
\begingroup%
  \makeatletter%
  \providecommand\color[2][]{%
    \errmessage{(Inkscape) Color is used for the text in Inkscape, but the package 'color.sty' is not loaded}%
    \renewcommand\color[2][]{}%
  }%
  \providecommand\transparent[1]{%
    \errmessage{(Inkscape) Transparency is used (non-zero) for the text in Inkscape, but the package 'transparent.sty' is not loaded}%
    \renewcommand\transparent[1]{}%
  }%
  \providecommand\rotatebox[2]{#2}%
  \ifx\svgwidth\undefined%
    \setlength{\unitlength}{75.4507477bp}%
    \ifx\svgscale\undefined%
      \relax%
    \else%
      \setlength{\unitlength}{\unitlength * \real{\svgscale}}%
    \fi%
  \else%
    \setlength{\unitlength}{\svgwidth}%
  \fi%
  \global\let\svgwidth\undefined%
  \global\let\svgscale\undefined%
  \makeatother%
  \begin{picture}(1,0.60795871)%
    \put(0,0){\includegraphics[width=\unitlength,page=1]{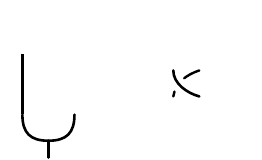}}%
    \put(0.45309773,0.25750977){\color[rgb]{0,0,0}\makebox(0,0)[lb]{\smash{$=$}}}%
    \put(0,0){\includegraphics[width=\unitlength,page=2]{ydcond3.pdf}}%
    \put(0.0333389,0.25954102){\color[rgb]{0,0,0}\makebox(0,0)[lb]{\smash{$S$}}}%
    \put(0,0){\includegraphics[width=\unitlength,page=3]{ydcond3.pdf}}%
    \put(0.81861939,0.25954102){\color[rgb]{0,0,0}\makebox(0,0)[lb]{\smash{$S$}}}%
  \end{picture}%
\endgroup%
}}~,&&\text{ and }&&
\vcenter{\hbox{
\begingroup%
  \makeatletter%
  \providecommand\color[2][]{%
    \errmessage{(Inkscape) Color is used for the text in Inkscape, but the package 'color.sty' is not loaded}%
    \renewcommand\color[2][]{}%
  }%
  \providecommand\transparent[1]{%
    \errmessage{(Inkscape) Transparency is used (non-zero) for the text in Inkscape, but the package 'transparent.sty' is not loaded}%
    \renewcommand\transparent[1]{}%
  }%
  \providecommand\rotatebox[2]{#2}%
  \ifx\svgwidth\undefined%
    \setlength{\unitlength}{66.33410942bp}%
    \ifx\svgscale\undefined%
      \relax%
    \else%
      \setlength{\unitlength}{\unitlength * \real{\svgscale}}%
    \fi%
  \else%
    \setlength{\unitlength}{\svgwidth}%
  \fi%
  \global\let\svgwidth\undefined%
  \global\let\svgscale\undefined%
  \makeatother%
  \begin{picture}(1,0.57808271)%
    \put(0.2459148,0.2896886){\color[rgb]{0,0,0}\makebox(0,0)[lb]{\smash{ }}}%
    \put(0,0){\includegraphics[width=\unitlength,page=1]{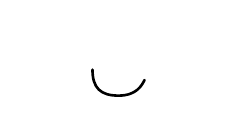}}%
    \put(0.17441893,0.25825717){\color[rgb]{0,0,0}\makebox(0,0)[lb]{\smash{$=$}}}%
    \put(0,0){\includegraphics[width=\unitlength,page=2]{ydcond2.pdf}}%
    \put(0.79369132,0.24926118){\color[rgb]{0,0,0}\makebox(0,0)[lb]{\smash{$S$}}}%
  \end{picture}%
\endgroup%
}}~.
\end{align}
using graphical calculus, and follow easily under use of the antipode axioms.
\end{proof}

The category of Yetter--Drinfeld module is a realization of a version of the center of the monoidal category $\lmod{B}$, relative to $\cB$, the \emph{relative monoidal center} (cf. \cites{Lau,Lau2}).

\subsection{The Tensor Product Action}

The following first main theorem of this paper states that we can tensor (using the tensor product of $\cB$) objects of $\lmod{A}(\lcomod{B})$ with objects of $\lYD{B}$ on the left and obtain modules in $\lmod{A}(\lcomod{B})$.

\begin{theorem}\label{mainthm}
Let $A$ be an algebra object in $\lcomod{B}$. Given objects $(V,a_V,\delta_V)$ in $\lYD{B}$, and an object $(W,a_W,\delta_W)$ in $\lmod{A}(\lcomod{B})$, their tensor product $V\otimes W$ becomes an object $V\triangleright W$ of $\lmod{A}(\lcomod{B})$ with $A$-action $a_{V\triangleright W}$ and $B$-coaction $\delta_{V\triangleright W}$ given by
\begin{align}\label{Aactioneq}
a_{V\triangleright W}&:=(a_V\otimes a_W)(\ide_B\otimes \Psi_{A,V}\otimes \ide_W)(\delta_A\otimes \ide_{V\otimes W})=\vcenter{\hbox{
\begingroup%
  \makeatletter%
  \providecommand\color[2][]{%
    \errmessage{(Inkscape) Color is used for the text in Inkscape, but the package 'color.sty' is not loaded}%
    \renewcommand\color[2][]{}%
  }%
  \providecommand\transparent[1]{%
    \errmessage{(Inkscape) Transparency is used (non-zero) for the text in Inkscape, but the package 'transparent.sty' is not loaded}%
    \renewcommand\transparent[1]{}%
  }%
  \providecommand\rotatebox[2]{#2}%
  \ifx\svgwidth\undefined%
    \setlength{\unitlength}{27.19538815bp}%
    \ifx\svgscale\undefined%
      \relax%
    \else%
      \setlength{\unitlength}{\unitlength * \real{\svgscale}}%
    \fi%
  \else%
    \setlength{\unitlength}{\svgwidth}%
  \fi%
  \global\let\svgwidth\undefined%
  \global\let\svgscale\undefined%
  \makeatother%
  \begin{picture}(1,0.99123827)%
    \put(0,0){\includegraphics[width=\unitlength,page=1]{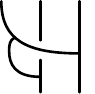}}%
    \put(0.44076032,0.50448358){\color[rgb]{0,0,0}\makebox(0,0)[lb]{\smash{$~$}}}%
  \end{picture}%
\endgroup%
}},\\
\delta_{V\triangleright W}&:=(m_B\otimes \ide_{V\otimes W})(\ide_B\otimes \Psi_{V,B}\otimes \ide_W)(\delta_V\otimes \delta_W)=\vcenter{\hbox{}}.
\end{align}
\end{theorem}
\begin{proof} It follows directly from $A$ being a $B$-comodule algebra that $a_{V\triangleright W}$ defines an $A$-action. We need to check that it is a morphism of $B$-comodules.
A proof of this, using graphical calculus, is given in Fig. \ref{Fig1}.
\begin{figure}[bt]
\[\vcenter{\hbox{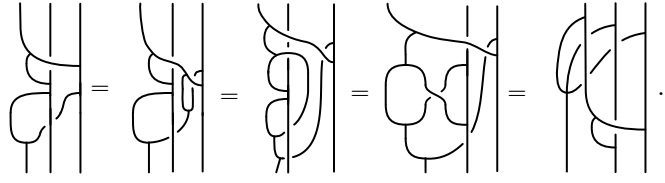}}.\]
\caption{Proof of Theorem \ref{mainthm}}
\label{Fig1}
\end{figure}
The first step uses that $A$ is a $B$-comodule algebra, see Eq. (\ref{comodalg}). The second equality uses coassociativity of the coproduct of $B$ an that $\delta$ is a coaction. Next, we apply the Yetter--Drinfeld compatibility condition depicted in Eq. (\ref{ydpic}), and finally, coassociativity and the comodule axioms are used again to conclude this proof that $V\triangleright W$ is an object in $\lmod{A}(\lcomod{B})$.

Alternatively, the proof can be given using the notation of composition and tensor products of maps:
\begin{align*}
\delta_{V\triangleright W} a_{V\triangleright W}=&(m_B\otimes\ide)(\ide_B\otimes \Psi_{V,B}\otimes a_W)(\delta_V a_V\otimes\delta_{A\otimes W})(\ide_B\otimes\Psi_{A,V}\otimes \ide_W)(\delta_A\otimes \ide_{V\otimes W})\\
=&(m_B\otimes\ide)(m_B\otimes \Psi_{V,B}\otimes \ide_W)(\ide_B\otimes \Psi_{V,B}\otimes \ide_{B\otimes W})(\delta a_V\otimes \ide_{B\otimes B}\otimes a_W)\\&(\ide_B\otimes \Psi_{B,V}\otimes \Psi_{A,B}\otimes \ide_W)(\Delta\otimes \Psi_{A,V}\otimes \delta_W)(\delta_A\otimes \ide_{V\otimes W})
\\
=&(m_B\otimes \ide)(\ide_B\otimes \Psi_{V,B}\otimes \ide_W)(m_B\otimes a_W\otimes \ide_{B\otimes W})(\ide_B\otimes \Psi_{B,B}\otimes \ide_{V\otimes B}\otimes a_W)\\&(\Delta\otimes \delta_V\otimes \Psi_{A,B}\otimes \ide_W)(\ide_B\otimes \Psi_{A,V}\otimes \delta_W)(\delta_A\otimes \ide_{V\otimes W})\\
=&(\ide_B\otimes a_{V\triangleright W})\delta_{A\otimes V\otimes W}.\qedhere
\end{align*}
\end{proof}

The following theorem gives a version of Theorem \ref{mainthm} for $B$-module algebras which is proved in an analogous way. A right $B$-module algebra in $\cB$ is an algebra object in the category $\rmod{B}$. That is, $A$ has a right $B$-module structure $a_A\colon A\otimes B\to A$, a multiplication $m_A\colon A\otimes A \to A$, and a unit $1_A\colon I\to A$, such that
\begin{gather}
a_A(m_A\otimes \ide_B)=m_A a_{A\otimes A}, \\
a_A(1_A\otimes \ide_B)=1_A\varepsilon.
\end{gather}
These conditions can be displayed as
\begin{equation}\label{modalg}\vcenter{\hbox{
\begingroup%
  \makeatletter%
  \providecommand\color[2][]{%
    \errmessage{(Inkscape) Color is used for the text in Inkscape, but the package 'color.sty' is not loaded}%
    \renewcommand\color[2][]{}%
  }%
  \providecommand\transparent[1]{%
    \errmessage{(Inkscape) Transparency is used (non-zero) for the text in Inkscape, but the package 'transparent.sty' is not loaded}%
    \renewcommand\transparent[1]{}%
  }%
  \providecommand\rotatebox[2]{#2}%
  \ifx\svgwidth\undefined%
    \setlength{\unitlength}{130.8750375bp}%
    \ifx\svgscale\undefined%
      \relax%
    \else%
      \setlength{\unitlength}{\unitlength * \real{\svgscale}}%
    \fi%
  \else%
    \setlength{\unitlength}{\svgwidth}%
  \fi%
  \global\let\svgwidth\undefined%
  \global\let\svgscale\undefined%
  \makeatother%
  \begin{picture}(1,0.29252606)%
    \put(0.18194866,0.17603019){\color[rgb]{0,0,0}\makebox(0,0)[lb]{\smash{ }}}%
    \put(0.2030175,0.11425389){\color[rgb]{0,0,0}\makebox(0,0)[lb]{\smash{$=$}}}%
    \put(0,0){\includegraphics[width=\unitlength,page=1]{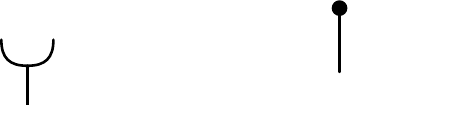}}%
    \put(0.60447656,0.10858129){\color[rgb]{0,0,0}\makebox(0,0)[lb]{\smash{$,$}}}%
    \put(0,0){\includegraphics[width=\unitlength,page=2]{modalg.pdf}}%
    \put(0.83338981,0.11466382){\color[rgb]{0,0,0}\makebox(0,0)[lb]{\smash{$=$}}}%
    \put(0,0){\includegraphics[width=\unitlength,page=3]{modalg.pdf}}%
  \end{picture}%
\endgroup%
}}.\end{equation}
where $\varepsilon=\vcenter{\hbox{
\begingroup%
  \makeatletter%
  \providecommand\color[2][]{%
    \errmessage{(Inkscape) Color is used for the text in Inkscape, but the package 'color.sty' is not loaded}%
    \renewcommand\color[2][]{}%
  }%
  \providecommand\transparent[1]{%
    \errmessage{(Inkscape) Transparency is used (non-zero) for the text in Inkscape, but the package 'transparent.sty' is not loaded}%
    \renewcommand\transparent[1]{}%
  }%
  \providecommand\rotatebox[2]{#2}%
  \ifx\svgwidth\undefined%
    \setlength{\unitlength}{4.8bp}%
    \ifx\svgscale\undefined%
      \relax%
    \else%
      \setlength{\unitlength}{\unitlength * \real{\svgscale}}%
    \fi%
  \else%
    \setlength{\unitlength}{\svgwidth}%
  \fi%
  \global\let\svgwidth\undefined%
  \global\let\svgscale\undefined%
  \makeatother%
  \begin{picture}(1,2.56889521)%
    \put(0,0){\includegraphics[width=\unitlength]{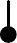}}%
  \end{picture}%
\endgroup%
}}$ denotes the counit.

\begin{theorem}\label{mainthm2}
Let $A$ be a right $B$-module algebra. Given objects $(V,a_V,\delta_V)$ in $\rYD{B}$, and $(W,a_W,b_W)$ in $\lmod{A}(\rmod{B})$, where $a_W\colon A\otimes W\to W$ and $b_W\colon W\otimes B\to W$ are $A$-, respectively $B$-actions, their tensor product $V\otimes W$ becomes an object $V\triangleright W$ of $\lmod{A}(\rmod{B})$ with $A$-action $a_{V\triangleright W}$ and $B$-action $b_{V\triangleright W}$ given by
\begin{align}
a_{V\triangleright W}&:=(\ide_{V}\otimes a_W)(\ide_{V}\otimes b_A\otimes \ide_W)(\Psi_{A,V}\otimes \ide_{B\otimes W})(\ide_A\otimes \delta_V\otimes \ide_W)=\vcenter{\hbox{
\begingroup%
  \makeatletter%
  \providecommand\color[2][]{%
    \errmessage{(Inkscape) Color is used for the text in Inkscape, but the package 'color.sty' is not loaded}%
    \renewcommand\color[2][]{}%
  }%
  \providecommand\transparent[1]{%
    \errmessage{(Inkscape) Transparency is used (non-zero) for the text in Inkscape, but the package 'transparent.sty' is not loaded}%
    \renewcommand\transparent[1]{}%
  }%
  \providecommand\rotatebox[2]{#2}%
  \ifx\svgwidth\undefined%
    \setlength{\unitlength}{27.19538815bp}%
    \ifx\svgscale\undefined%
      \relax%
    \else%
      \setlength{\unitlength}{\unitlength * \real{\svgscale}}%
    \fi%
  \else%
    \setlength{\unitlength}{\svgwidth}%
  \fi%
  \global\let\svgwidth\undefined%
  \global\let\svgscale\undefined%
  \makeatother%
  \begin{picture}(1,0.99123827)%
    \put(0,0){\includegraphics[width=\unitlength,page=1]{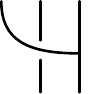}}%
    \put(0.44076032,0.50448358){\color[rgb]{0,0,0}\makebox(0,0)[lb]{\smash{$~$}}}%
    \put(0,0){\includegraphics[width=\unitlength,page=2]{Aaction2.pdf}}%
  \end{picture}%
\endgroup%
}},\\
b_{V\triangleright W}&:=(b_V\otimes b_W)(\ide_V\otimes \Psi_{W,B}\otimes \ide_B)(\ide_{V\otimes W}\otimes \Delta_B)=\vcenter{\hbox{
\begingroup%
  \makeatletter%
  \providecommand\color[2][]{%
    \errmessage{(Inkscape) Color is used for the text in Inkscape, but the package 'color.sty' is not loaded}%
    \renewcommand\color[2][]{}%
  }%
  \providecommand\transparent[1]{%
    \errmessage{(Inkscape) Transparency is used (non-zero) for the text in Inkscape, but the package 'transparent.sty' is not loaded}%
    \renewcommand\transparent[1]{}%
  }%
  \providecommand\rotatebox[2]{#2}%
  \ifx\svgwidth\undefined%
    \setlength{\unitlength}{30.44989059bp}%
    \ifx\svgscale\undefined%
      \relax%
    \else%
      \setlength{\unitlength}{\unitlength * \real{\svgscale}}%
    \fi%
  \else%
    \setlength{\unitlength}{\svgwidth}%
  \fi%
  \global\let\svgwidth\undefined%
  \global\let\svgscale\undefined%
  \makeatother%
  \begin{picture}(1,0.88991508)%
    \put(0,0){\includegraphics[width=\unitlength,page=1]{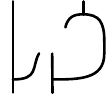}}%
    \put(0.49946781,0.47717655){\color[rgb]{0,0,0}\rotatebox{-180}{\makebox(0,0)[lb]{\smash{$~$}}}}%
    \put(0,0){\includegraphics[width=\unitlength,page=2]{tensoractionr.pdf}}%
  \end{picture}%
\endgroup%
}}.
\end{align}
\end{theorem}

Theorems \ref{mainthm} and \ref{mainthm2} in fact provide categorical modules in the sense of \cite{EGNO}*{Definition 7.1.1}:

\begin{corollary}\label{catmodules}
Let $A$ be an algebra object in $\lcomod{B}$ and $D$ an algebra object in $\rmod{B}$.
\begin{enumerate}
\item[(i)] The category $\lmod{A}(\lcomod{B})$ is a left module over the monoidal category $\lYD{B}$. That is, 
\begin{align*}
\triangleright\colon \lYD{B}\times \lmod{A}(\lcomod{B})&\longrightarrow\lmod{A}(\lcomod{B}), &(V,W)\longmapsto V\triangleright W
\end{align*}
extends to a $\Bbbk$-bilinear functor such that $1\triangleright V\cong V$ and there is a natural isomorphism 
$$(V_1\otimes V_2)\triangleright W\cong V_1\triangleright(V_2\triangleright W),$$
for two objects $V_1,V_2$ of $\lYD{B}$.
\item[(ii)] Similarly, $\lmod{D}(\rmod{B})$ is a left categorical module over $\rYD{B}$ via the $\Bbbk$-bilinear functor
\begin{align*}
\triangleright\colon \rYD{B}\times \lmod{D}(\rmod{B})&\longrightarrow\lmod{D}(\rmod{B}), &(V,W)\longmapsto V\triangleright W.
\end{align*}
\end{enumerate}
\end{corollary}

\begin{remark}
Note that the categorical actions in Corollary \ref{catmodules} are instances of a categorical construction of modules over the relative monoidal center. This point of view is explored further in \cite{Lau2}.
\end{remark}

For a version of Theorem \ref{mainthm2} valid for a right $B$-module algebra, we need the following lemma about right-left translation of Yetter--Drinfeld modules. We then obtain a version of Theorem \ref{mainthm2} for \emph{left} $B$-module algebras.

For this, denote by $\leftexp{\cop}{B}$ the bialgebra in $\overline{\cB}$ with the same product  $m_B$ as $B$ but opposite coproduct $\Psi^{-1}_{B,B}\Delta_B$. If $B$ is a Hopf algebra, then $\leftexp{\cop}{B}$ is a Hopf algebra with antipode is $S^{-1}$.

\begin{lemma}\label{rightleft}Let $B$ be a Hopf algebra in $\cB$.
There is an equivalence of braided monoidal categories
\begin{align}
\rYD{B}(\cB) &\longrightarrow \lYD{\leftexp{\cop}{B}}(\overline{\cB}), &(V,a,\delta)\longmapsto (V, a', \delta'),
\end{align}
where 
\begin{align*}
a'&=a\Psi^{-1}_{B,V}(S^{-1}\otimes \ide_V),&
\delta'&=\Psi^{-1}_{V,B}\delta.
\end{align*}
\end{lemma}
\begin{proof}
We first verify that given a right $B$-module $V$, with action $a$, $(V,a')$ as above is a left $\leftexp{\cop}{B}$-module, and the tensor product is given using the opposite coproduct $(\Psi^{-1}_{B,B})\Delta_B$ (in the category $\overline{\cB}$). This provides a monoidal functor $\rmod{B}(\cB)\to \lmod{\leftexp{\cop}{B}}(\overline{\cB})$. Similarly, $(V,\delta)\mapsto (V,\delta')$ gives a monoidal functor $\rcomod{B}(\cB)\to \lcomod{\leftexp{\cop}{B}}(\overline{\cB})$. Finally, given a right YD-module over $B$, the left $\leftexp{\cop}{B}$-module and comodule structures are again YD-compatible, but using the braiding of $\ov{\cB}$. To show this, it helps to use Equation (\ref{altYD1}) for the Hopf algebra $\leftexp{\cop}{B}$ in $\ov{\cB}$.
\end{proof}

\begin{corollary}\label{maincor}Assume $B$ is a Hopf algebra in $\cB$.
Let $A$ be an algebra object in $\lmod{B}$. Given two objects $(V,a_V,\delta_V)$ in $\lYD{B}$, and $(W,a_W,b_W)$ in $\lmod{A}(\lmod{B})$, where $a_W\colon A\otimes W\to W$ and $b_W\colon B\otimes W\to W$ are $A$-, respectively $B$-actions, their tensor product $V\otimes W$ becomes an object $V\triangleright W$ of $\lmod{A}(\lmod{B})$ with $A$-action $a_{V\triangleright W}$ and $B$-action $b_{V\triangleright W}$ given by
\begin{align}
\begin{split}
a_{V\triangleright W}&:=(\ide\otimes a_W)(\Psi^{-1}_{V,A}\otimes\ide_W)(a_V\Psi^{-1}_{B,A}\otimes \ide_{B\otimes W})(\ide_A\otimes (S^{-1}\otimes \ide_V)\delta_V\otimes \ide_W)\\&=\vcenter{\hbox{
\begingroup%
  \makeatletter%
  \providecommand\color[2][]{%
    \errmessage{(Inkscape) Color is used for the text in Inkscape, but the package 'color.sty' is not loaded}%
    \renewcommand\color[2][]{}%
  }%
  \providecommand\transparent[1]{%
    \errmessage{(Inkscape) Transparency is used (non-zero) for the text in Inkscape, but the package 'transparent.sty' is not loaded}%
    \renewcommand\transparent[1]{}%
  }%
  \providecommand\rotatebox[2]{#2}%
  \ifx\svgwidth\undefined%
    \setlength{\unitlength}{75.26778031bp}%
    \ifx\svgscale\undefined%
      \relax%
    \else%
      \setlength{\unitlength}{\unitlength * \real{\svgscale}}%
    \fi%
  \else%
    \setlength{\unitlength}{\svgwidth}%
  \fi%
  \global\let\svgwidth\undefined%
  \global\let\svgscale\undefined%
  \makeatother%
  \begin{picture}(1,0.51829191)%
    \put(0,0){\includegraphics[width=\unitlength,page=1]{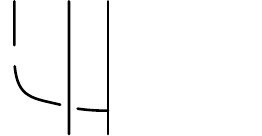}}%
    \put(0.26886153,0.28263354){\color[rgb]{0,0,0}\makebox(0,0)[lb]{\smash{$~$}}}%
    \put(0,0){\includegraphics[width=\unitlength,page=2]{Aaction3.pdf}}%
    \put(0.0833371,0.35031365){\color[rgb]{0,0,0}\makebox(0,0)[lb]{\smash{\tiny{$S^{\text{-}1}$}}}}%
    \put(0.46379542,0.25217375){\color[rgb]{0,0,0}\makebox(0,0)[lb]{\smash{$,$}}}%
  \end{picture}%
\endgroup%
}}\end{split}\\
b_{V\triangleright W}&:=(b_V\otimes b_W)(\ide_B\otimes \Psi_{B,V}\otimes \ide_W)(\Delta_B\otimes \ide_{V\otimes W})=\vcenter{\hbox{}}.
\end{align}
\end{corollary}
\begin{proof}
This follows from Theorem \ref{mainthm2} under Lemma \ref{rightleft}.
\end{proof}

Theorem \ref{mainthm} and \ref{mainthm2} can further be dualized to give a version for (co)module coalgebras. Graphically, dualizing corresponds to rotating a picture by 180 degrees in the horizontal middle axis. From this, formulas for the tensor product can be derived.

\begin{corollary}\label{coalgebracor}Let $B$ be a bialgebra in $\cB$.
\begin{itemize}
\item[(i)]
Let $C$ be an coalgebra object in $\lmod{B}$, with $B$-action $a_C\colon B\otimes C\to C$. Given objects $(V,a_V,\delta_V)$ in $\lYD{B}$, and $(W,\delta_W,a_W)$ in $\lcomod{C}(\lmod{B})$, then $V\otimes W$ becomes an object $V\triangleright W$ of $\lcomod{C}(\lmod{B})$ with $C$-coaction $\delta_{V\triangleright W}$ and $B$-action $a_{V\triangleright W}$ given by
\begin{align}
\delta_{V\triangleright W}&:=(a_C\otimes \ide_{V\otimes W})(\ide_B\otimes \Psi_{V,C}\otimes \ide_W)(\delta_V\otimes \delta_W)=\vcenter{\hbox{
\begingroup%
  \makeatletter%
  \providecommand\color[2][]{%
    \errmessage{(Inkscape) Color is used for the text in Inkscape, but the package 'color.sty' is not loaded}%
    \renewcommand\color[2][]{}%
  }%
  \providecommand\transparent[1]{%
    \errmessage{(Inkscape) Transparency is used (non-zero) for the text in Inkscape, but the package 'transparent.sty' is not loaded}%
    \renewcommand\transparent[1]{}%
  }%
  \providecommand\rotatebox[2]{#2}%
  \ifx\svgwidth\undefined%
    \setlength{\unitlength}{26.83653035bp}%
    \ifx\svgscale\undefined%
      \relax%
    \else%
      \setlength{\unitlength}{\unitlength * \real{\svgscale}}%
    \fi%
  \else%
    \setlength{\unitlength}{\svgwidth}%
  \fi%
  \global\let\svgwidth\undefined%
  \global\let\svgscale\undefined%
  \makeatother%
  \begin{picture}(1,1.00449309)%
    \put(0,0){\includegraphics[width=\unitlength,page=1]{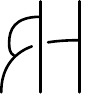}}%
    \put(0.43328217,0.26369901){\color[rgb]{0,0,0}\makebox(0,0)[lb]{\smash{$~$}}}%
  \end{picture}%
\endgroup%
}},\\
a_{V\triangleright W}&:=(a_V\otimes a_W)(\ide_B\otimes \Psi_{B,V}\otimes \ide_W)(\Delta_B\otimes \ide_{V\otimes W})=\vcenter{\hbox{}}~.
\end{align}
\item[(ii)]
Let $C$ be an coalgebra object in $\rcomod{B}$, with $B$-coaction $\delta_C\colon C\to C\otimes B$. Given objects $(V,a_V,\delta_V)$ in $\rYD{B}$, and and object $(W,\gamma_W,\delta_W)$ in $\lcomod{C}(\rcomod{B})$, where $\gamma_W\colon W\to C\otimes W $ and $\delta_W\colon W\to W\otimes B$ are $C$, and $B$-coactions, then $V\otimes W$ is an object $V\triangleright W$ of $\lcomod{C}(\rcomod{B})$ with $C$-coaction $\gamma_{V\triangleright W}$ and $B$-coaction $\delta_{V\triangleright W}$ given by
\begin{align}
\gamma_{V\triangleright W}&:=(\ide_{C}\otimes a_V\otimes \ide_W)(\Psi_{V,C}\otimes \ide_{B\otimes W})(\ide_V\otimes \delta_{C}\otimes \ide_{W})(\ide_V\otimes \gamma_W)=\vcenter{\hbox{
\begingroup%
  \makeatletter%
  \providecommand\color[2][]{%
    \errmessage{(Inkscape) Color is used for the text in Inkscape, but the package 'color.sty' is not loaded}%
    \renewcommand\color[2][]{}%
  }%
  \providecommand\transparent[1]{%
    \errmessage{(Inkscape) Transparency is used (non-zero) for the text in Inkscape, but the package 'transparent.sty' is not loaded}%
    \renewcommand\transparent[1]{}%
  }%
  \providecommand\rotatebox[2]{#2}%
  \ifx\svgwidth\undefined%
    \setlength{\unitlength}{27.21565341bp}%
    \ifx\svgscale\undefined%
      \relax%
    \else%
      \setlength{\unitlength}{\unitlength * \real{\svgscale}}%
    \fi%
  \else%
    \setlength{\unitlength}{\svgwidth}%
  \fi%
  \global\let\svgwidth\undefined%
  \global\let\svgscale\undefined%
  \makeatother%
  \begin{picture}(1,0.99050017)%
    \put(0,0){\includegraphics[width=\unitlength,page=1]{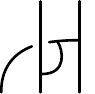}}%
    \put(0.44117674,0.29397986){\color[rgb]{0,0,0}\makebox(0,0)[lb]{\smash{$~$}}}%
  \end{picture}%
\endgroup%
}},\\
\delta_{V\triangleright W}&:=(\ide_{V\otimes W}\otimes m_B)(\ide_V\otimes \Psi_{B,W}\otimes \ide_B)(\delta_V\otimes \delta_W)=\vcenter{\hbox{
\begingroup%
  \makeatletter%
  \providecommand\color[2][]{%
    \errmessage{(Inkscape) Color is used for the text in Inkscape, but the package 'color.sty' is not loaded}%
    \renewcommand\color[2][]{}%
  }%
  \providecommand\transparent[1]{%
    \errmessage{(Inkscape) Transparency is used (non-zero) for the text in Inkscape, but the package 'transparent.sty' is not loaded}%
    \renewcommand\transparent[1]{}%
  }%
  \providecommand\rotatebox[2]{#2}%
  \ifx\svgwidth\undefined%
    \setlength{\unitlength}{30.44989059bp}%
    \ifx\svgscale\undefined%
      \relax%
    \else%
      \setlength{\unitlength}{\unitlength * \real{\svgscale}}%
    \fi%
  \else%
    \setlength{\unitlength}{\svgwidth}%
  \fi%
  \global\let\svgwidth\undefined%
  \global\let\svgscale\undefined%
  \makeatother%
  \begin{picture}(1,0.88991508)%
    \put(0,0){\includegraphics[width=\unitlength,page=1]{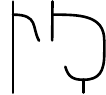}}%
    \put(0.49946781,0.41273914){\color[rgb]{0,0,0}\rotatebox{-180}{\makebox(0,0)[lb]{\smash{$~$}}}}%
    \put(0,0){\includegraphics[width=\unitlength,page=2]{tensorcoactionr.pdf}}%
  \end{picture}%
\endgroup%
}}.
\end{align}
\end{itemize}
\end{corollary}

Hence, similarly to before, there are categorical actions
\begin{align*}
\triangleright\colon &\lYD{B}\times \lcomod{C}(\lmod{B})\longrightarrow\lcomod{C}(\lmod{B}),\\
\triangleright\colon &\rYD{B}\times \lcomod{C}(\rcomod{B})\longrightarrow\lcomod{C}(\rcomod{B}).
\end{align*}

\subsection{Examples}\label{examplessect}

The rest of this section is devoted to the inclusion of various general examples, based on natural (co)module algebra structures from the literature.

\begin{example}[Regular coaction]\label{regularcoaction}
Any bialgebra $B$ is a $B$-comodule algebra with respect to the \emph{left regular coaction} via the coproduct $\Delta$, see e.g. \cite{Maj1}*{Example 1.6.17}. We denote $B$ viewed as a $B$-comodule algebra in this way by $B^{\op{reg}}$. The category $\lmod{B^{\op{reg}}}(\lcomod{B})$ is equivalent to the category of \emph{Hopf modules} over $B$ in $\cB$ as treated (in such a generality) in \cites{Bes,BD}. A Hopf module $V$ is an object of $\cB$ with a $B$-coaction $\delta$ and a $B$-action $a$, such that
\begin{align}
\delta a&= (m_B\otimes a)(\ide_B\otimes \Psi_{B,B}\otimes \ide_V)(\Delta\otimes \delta) &&\Longleftrightarrow&&
\vcenter{\hbox{
\begingroup%
  \makeatletter%
  \providecommand\color[2][]{%
    \errmessage{(Inkscape) Color is used for the text in Inkscape, but the package 'color.sty' is not loaded}%
    \renewcommand\color[2][]{}%
  }%
  \providecommand\transparent[1]{%
    \errmessage{(Inkscape) Transparency is used (non-zero) for the text in Inkscape, but the package 'transparent.sty' is not loaded}%
    \renewcommand\transparent[1]{}%
  }%
  \providecommand\rotatebox[2]{#2}%
  \ifx\svgwidth\undefined%
    \setlength{\unitlength}{59.04203141bp}%
    \ifx\svgscale\undefined%
      \relax%
    \else%
      \setlength{\unitlength}{\unitlength * \real{\svgscale}}%
    \fi%
  \else%
    \setlength{\unitlength}{\svgwidth}%
  \fi%
  \global\let\svgwidth\undefined%
  \global\let\svgscale\undefined%
  \makeatother%
  \begin{picture}(1,0.52112539)%
    \put(0.27628689,0.26289578){\color[rgb]{0,0,0}\makebox(0,0)[lb]{\smash{ }}}%
    \put(0,0){\includegraphics[width=\unitlength,page=1]{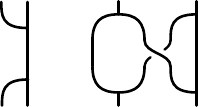}}%
    \put(0.1968943,0.22893528){\color[rgb]{0,0,0}\makebox(0,0)[lb]{\smash{$=$}}}%
    \put(0.7939266,0.19711312){\color[rgb]{0,0,0}\makebox(0,0)[lb]{\smash{$~$}}}%
  \end{picture}%
\endgroup%
}}.
\end{align}
In this case, Theorem \ref{mainthm} recovers the result of \cite{Lau} (proved there for quasi-Hopf algebras) that the tensor product $V\otimes W$ of a Yetter--Drinfeld module $V$ with a Hopf module $W$ is again a Hopf module over $B$. We will remark in Examples \ref{Heisexample} and \ref{heistwist} how this tensor product generalizes the main result of \cite{Lu}.
\end{example}

\begin{example}[Trivial coaction]\label{trivialcoaction} A more basic example is given by viewing $B$ as a comodule algebra $B^{\triv}$ with respect to the trivial coaction (via the counit $\varepsilon$) on itself. In this case, the category $\lmod{B^{\triv}}(\lcomod{B})$ consist of simultaneous left $B$-modules and comodules such that the structures commute, i.e. an object $(V,a,\delta)$ with action $a$ and coaction $\delta$ satisfying
\begin{align}
\delta a&=(\ide_B\otimes a)(\Psi_{B,B}\otimes \ide_V)(\ide_B\otimes \delta).
\end{align}
In this case, the action of $\lYD{B}$ factors via the forgetful functor $\lYD{B}\to \lcomod{B}$ through the regular categorical action of $\lcomod{B}$ on $\lmod{B^{\triv}}(\lcomod{B})$, which acts by tensoring the $B$-comodule structures, and the induced $B$-module structure.
\end{example}

\begin{example}[Adjoint action]
Another natural action for a Hopf algebra object $H$ in $\cB$ on itself is given by the \emph{right adjoint action}, defined by
\begin{align}
a^{\op{ad}} := m_H(S\otimes m_H)(\Psi_{H,H}\otimes \ide_H)(\ide_H\otimes \Delta)=\vcenter{\hbox{
\begingroup%
  \makeatletter%
  \providecommand\color[2][]{%
    \errmessage{(Inkscape) Color is used for the text in Inkscape, but the package 'color.sty' is not loaded}%
    \renewcommand\color[2][]{}%
  }%
  \providecommand\transparent[1]{%
    \errmessage{(Inkscape) Transparency is used (non-zero) for the text in Inkscape, but the package 'transparent.sty' is not loaded}%
    \renewcommand\transparent[1]{}%
  }%
  \providecommand\rotatebox[2]{#2}%
  \ifx\svgwidth\undefined%
    \setlength{\unitlength}{35.32045251bp}%
    \ifx\svgscale\undefined%
      \relax%
    \else%
      \setlength{\unitlength}{\unitlength * \real{\svgscale}}%
    \fi%
  \else%
    \setlength{\unitlength}{\svgwidth}%
  \fi%
  \global\let\svgwidth\undefined%
  \global\let\svgscale\undefined%
  \makeatother%
  \begin{picture}(1,1.16079628)%
    \put(0,0){\includegraphics[width=\unitlength]{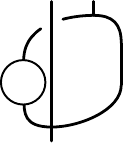}}%
    \put(0.07582427,0.38995321){\color[rgb]{0,0,0}\makebox(0,0)[lb]{\smash{$S$}}}%
  \end{picture}%
\endgroup%
}}~~,
\end{align}
cf. \cite{Maj1}*{Example 1.6.9}. This makes $H$ a right $H$-module algebra denoted by $H^{\op{ad}}$. Thus, by Theorem \ref{mainthm2}, we we have a left tensor action $$\rYD{H}\times \lmod{H^{\op{ad}}}(\rmod{H})\to \lmod{H^{\op{ad}}}(\rmod{H}).$$
\end{example}

\begin{example}
A more concrete example can be given by considering the case $H=\Bbbk G$ of a group algebra. In this case, the category $\lmod{H^{\op{ad}}}(\rmod{H})$ consists of $G\rtimes^{\op{ad}} G$-modules, for the group algebra of the semidirect product of $G$ by itself, in which  
$(1,g)(h,1) = (ghg^{-1},g)$ holds. Indeed, an object $W$ in $\lmod{H^{\op{ad}}}(\rmod{H})$ has a right $H$-action $w\cdot g$ and a left $H$-action $g\cdot w$, for $w\in W$ and $g\in G$. It becomes a left $G\rtimes^{\op{ad}} G$-module by linearly extending an action defined for pairs $g,h\in G$ by
\begin{align}
(g,h)\triangleright w=g\cdot (w\cdot h^{-1}).
\end{align}
Theorem \ref{mainthm2} states that we can tensor a module over $G\rtimes^{\op{ad}} G$ by right $G$-crossed module, i.e a $G$-graded right $G$-module $V=\bigoplus_{h\in G}V_g$ such that $ V_h\cdot g=V_{g^{-1}hg}$. Then the resulting $G\rtimes^{\op{ad}}G$-module has action given, for $v\in V_{|v|}$, $w\in W$, by
\begin{align*}
(h,1)\triangleright v\otimes w&=v\otimes |v|^{-1}h|v|\cdot w,\\
(1,g)\triangleright v\otimes w&=v\cdot g^{-1}\otimes w\cdot g^{-1}.
\end{align*}
\end{example}

\begin{example}[Adjoint coaction coalgebra]
We can also consider the left \emph{adjoint coaction} for a Hopf algebra $H$ in $\cB$ on itself which is defined by 
\begin{align}
\delta^{\op{ad}} := (m\otimes \ide)(\ide\otimes \Psi(\ide \otimes S)\Delta)\Delta=\vcenter{\hbox{
\begingroup%
  \makeatletter%
  \providecommand\color[2][]{%
    \errmessage{(Inkscape) Color is used for the text in Inkscape, but the package 'color.sty' is not loaded}%
    \renewcommand\color[2][]{}%
  }%
  \providecommand\transparent[1]{%
    \errmessage{(Inkscape) Transparency is used (non-zero) for the text in Inkscape, but the package 'transparent.sty' is not loaded}%
    \renewcommand\transparent[1]{}%
  }%
  \providecommand\rotatebox[2]{#2}%
  \ifx\svgwidth\undefined%
    \setlength{\unitlength}{37.0695443bp}%
    \ifx\svgscale\undefined%
      \relax%
    \else%
      \setlength{\unitlength}{\unitlength * \real{\svgscale}}%
    \fi%
  \else%
    \setlength{\unitlength}{\svgwidth}%
  \fi%
  \global\let\svgwidth\undefined%
  \global\let\svgscale\undefined%
  \makeatother%
  \begin{picture}(1,1.10634114)%
    \put(0,0){\includegraphics[width=\unitlength]{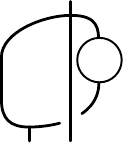}}%
    \put(0.62580889,0.52059878){\color[rgb]{0,0,0}\makebox(0,0)[lb]{\smash{$S$}}}%
  \end{picture}%
\endgroup%
}}~~.
\end{align}
This makes $H$ a coalgebra object in $\lcomod{H}$.  The left comodule coalgebra version of Corollary \ref{coalgebracor} implies a right tensor action
\[
\triangleright\colon \rcomod{H^{\op{ad}}}(\lcomod{H})\times \lYD{H}\longrightarrow \rcomod{H^{\op{ad}}}(\lcomod{H}).
\]
Given an object $(W,\delta^l,\delta^r)$ in $\rcomod{H^{\op{ad}}}(\lcomod{H})$, where $\delta^l$ is the left and $\delta^r$ the right $H$-coaction, and an object $(V,a,\delta)$ of $\lYD{H}$, the left tensor product $H$-comodule $W\triangleleft V$ obtains a right $H$-comodule structure given by
\begin{align}
\delta^r_{W\triangleright V}=(\ide_W\otimes a\otimes \ide_H)(\ide_W\otimes \Psi_{H,V})(\ide_W\otimes \delta^{\op{ad}}\otimes \ide_V)(\delta_r\otimes \ide_V)=\vcenter{\hbox{
\begingroup%
  \makeatletter%
  \providecommand\color[2][]{%
    \errmessage{(Inkscape) Color is used for the text in Inkscape, but the package 'color.sty' is not loaded}%
    \renewcommand\color[2][]{}%
  }%
  \providecommand\transparent[1]{%
    \errmessage{(Inkscape) Transparency is used (non-zero) for the text in Inkscape, but the package 'transparent.sty' is not loaded}%
    \renewcommand\transparent[1]{}%
  }%
  \providecommand\rotatebox[2]{#2}%
  \newcommand*\fsize{\dimexpr\f@size pt\relax}%
  \newcommand*\lineheight[1]{\fontsize{\fsize}{#1\fsize}\selectfont}%
  \ifx\svgwidth\undefined%
    \setlength{\unitlength}{53.3223775bp}%
    \ifx\svgscale\undefined%
      \relax%
    \else%
      \setlength{\unitlength}{\unitlength * \real{\svgscale}}%
    \fi%
  \else%
    \setlength{\unitlength}{\svgwidth}%
  \fi%
  \global\let\svgwidth\undefined%
  \global\let\svgscale\undefined%
  \makeatother%
  \begin{picture}(1,0.99884297)%
    \lineheight{1}%
    \setlength\tabcolsep{0pt}%
    \put(0,0){\includegraphics[width=\unitlength,page=1]{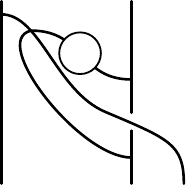}}%
    \put(0.36206025,0.64956169){\color[rgb]{0,0,0}\makebox(0,0)[lt]{\lineheight{0}\smash{\begin{tabular}[t]{l}$S$\end{tabular}}}}%
  \end{picture}%
\endgroup%
}}.
\end{align}
\end{example}

\begin{example}[Adjoint coaction algebra]\label{adjointcoaction}
Let $H$ be a commutative Hopf algebra in $\cB$, cf. Section \ref{setup}. In this case, $H^{\op{ad}}$ is also an \emph{algebra} in the category $\lcomod{H}$. 
To prove this, graphical calculus can be used (Fig \ref{Fig2}).
\begin{figure}[bt]
\[
\vcenter{\hbox{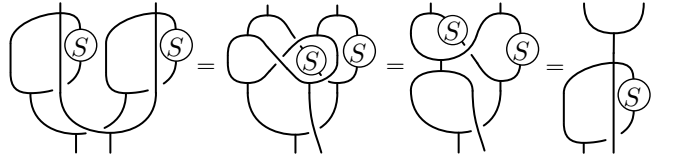}}~~.\]
\caption{Proof that $H^{\op{ad}}$ is a left $H$-comodule algebra}\label{Fig2}
\end{figure}
The third equality uses Lemma \ref{weirdcross}. In fact, combining this observation with the previous example, $H^{\op{ad}}$ becomes a Hopf algebra in the braided monoidal category $\lcomod{H}$.
Again using commutativity of $H$, the category $\lmod{H^{\op{ad}}}(\lcomod{H})$ is equivalent to $\lYD{H}$ (compare to condition \ref{altYD2}), and the tensor product of Theorem \ref{mainthm} recovers the monoidal structure on $\lYD{H}$ under this equivalence.
\end{example}

\begin{example}[Transmutation]
Let now let $\cB=\Vect$. An interesting generalization of the previous example was constructed by Majid (see \cite{Maj8}*{Theorem 4.1}, and also \cite{Maj1}*{9.4.10}) using a process called \emph{transmutation}. Let $H$ be dual quasitriangular. Then  $H$ is replaced by the \emph{covariantized} version $\un{H}$. As coalgebras, $\un{H}=H$, but the product is changed so that $\un{H}$ is an algebra object in $\lcomod{H}$ with respect to the adjoint coaction (in \cites{Maj8,Maj1}, right comodules are used, but a left comodule version can be constructed). There exists an equivalence of monoidal categories $\lmod{\un{H}}(\lcomod{H})\simeq \lYD{H}$ \cite{Maj1}*{Theorem 7.4.5}. Under this equivalence, the tensor product action of Theorem \ref{mainthm} corresponds to the monoidal structure of $\lYD{H}$.
\end{example}


\section{Constructing Comodule Algebras over the Drinfeld Double}\label{algebrapicture}

In this section, we relate the results from Section \ref{ydtensoractions} to the construction of comodule algebras over braided Drinfeld doubles.

\subsection{Dually Paired Braided Bialgebras}

We now describes the setup adapted throughout this section. We let $B$ be a bialgebra object in $\cB$ as before. Assume $C$ is another bialgebra object in $\cB$ which is a \emph{left weak dual} to $B$.

\begin{definition}\label{catdualdef}$~$
\begin{enumerate}
\item[(i)]
Two bialgebras $B,C$ in $\cB$ are \emph{weakly dual} if there exists a \emph{weak bialgebra pairing} $\ev\colon C\otimes B\to I$, denoted by $\ev=\vcenter{\hbox{
\begingroup%
  \makeatletter%
  \providecommand\color[2][]{%
    \errmessage{(Inkscape) Color is used for the text in Inkscape, but the package 'color.sty' is not loaded}%
    \renewcommand\color[2][]{}%
  }%
  \providecommand\transparent[1]{%
    \errmessage{(Inkscape) Transparency is used (non-zero) for the text in Inkscape, but the package 'transparent.sty' is not loaded}%
    \renewcommand\transparent[1]{}%
  }%
  \providecommand\rotatebox[2]{#2}%
  \ifx\svgwidth\undefined%
    \setlength{\unitlength}{16.80101968bp}%
    \ifx\svgscale\undefined%
      \relax%
    \else%
      \setlength{\unitlength}{\unitlength * \real{\svgscale}}%
    \fi%
  \else%
    \setlength{\unitlength}{\svgwidth}%
  \fi%
  \global\let\svgwidth\undefined%
  \global\let\svgscale\undefined%
  \makeatother%
  \begin{picture}(1,0.52545969)%
    \put(0,0){\includegraphics[width=\unitlength]{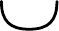}}%
  \end{picture}%
\endgroup%
}}$ such that
\begin{align}
\ev(m_C\otimes \ide_B)&=\ev^{\otimes 2}(\ide_{C\otimes C}\otimes \Delta_B), &\ev(\ide_C\otimes m_B)&=\ev^{\otimes 2}(\Delta_C\otimes \ide_{B\otimes B}),\\
\ev(1_C\otimes \ide_B)&=\varepsilon_B, & \ev(\ide_C\otimes 1_B)&=\varepsilon_C,
\end{align}
where $\ev^{\otimes 2}=\ev(\ide_C\otimes \ev\otimes \ide_B)$. 
\item[(ii)]
If $B, C$ are Hopf algebras, then we require, in addition, that
\begin{equation}
\ev(S_C\otimes \ide_B)=\ev(\ide_C\otimes S_B).
\end{equation}
\item[(iii)] We say that the weakly dual bialgebras (or Hopf algebras) $B$, $C$ are \emph{strongly dual} if there exists a \emph{coevaluation} $\coev\colon I\to B\otimes C$ such that
\begin{align}
(\ev\otimes \ide_C)(\ide_C\otimes \coev)&=\ide_C, & (\ide_B\otimes \coev)(\coev\otimes \ide_B)&=\ide_B.
\end{align}
\end{enumerate}
\end{definition}

\begin{remark}
Note that a weak bialgebra pairing $\ev$ is not required to be non-degenerate in general. Any pairing of strongly dual bialgebras is non-degenerate, but in the infinite-dimensional case, the converse does not hold as the coevaluation map only exists in a completion of the tensor product. Hence, in this case, the duality is not strong.

Note further that the convention on $\ev^{\otimes 2}$ is different from the usual induced pairing on $V^{\otimes 2}$ for vector spaces, but more suited to working in a braided monoidal category since it minimizes the occurrence of the braiding. Duality pairings following this convention are sometimes called \emph{categorical}, cf. \cite{Maj6}*{Section~2.2}.
\end{remark}

The conditions of Definition \ref{catdualdef} are express via graphical calculus as
\begin{equation}\label{dualhopfcond}
\vcenter{\hbox{
\begingroup%
  \makeatletter%
  \providecommand\color[2][]{%
    \errmessage{(Inkscape) Color is used for the text in Inkscape, but the package 'color.sty' is not loaded}%
    \renewcommand\color[2][]{}%
  }%
  \providecommand\transparent[1]{%
    \errmessage{(Inkscape) Transparency is used (non-zero) for the text in Inkscape, but the package 'transparent.sty' is not loaded}%
    \renewcommand\transparent[1]{}%
  }%
  \providecommand\rotatebox[2]{#2}%
  \ifx\svgwidth\undefined%
    \setlength{\unitlength}{322.28894175bp}%
    \ifx\svgscale\undefined%
      \relax%
    \else%
      \setlength{\unitlength}{\unitlength * \real{\svgscale}}%
    \fi%
  \else%
    \setlength{\unitlength}{\svgwidth}%
  \fi%
  \global\let\svgwidth\undefined%
  \global\let\svgscale\undefined%
  \makeatother%
  \begin{picture}(1,0.07812286)%
    \put(0,0){\includegraphics[width=\unitlength]{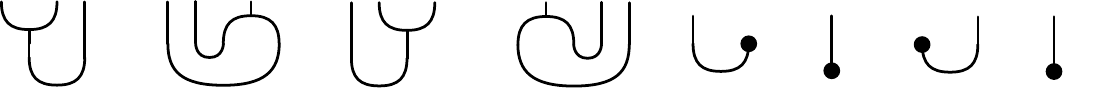}}%
    \put(0.69609229,0.03862872){\color[rgb]{0,0,0}\makebox(0,0)[lb]{\smash{$=$}}}%
    \put(0.09831456,0.03978099){\color[rgb]{0,0,0}\makebox(0,0)[lb]{\smash{$=$}}}%
    \put(0.26169948,0.03641221){\color[rgb]{0,0,0}\makebox(0,0)[lb]{\smash{$,$}}}%
    \put(0.41107738,0.03907179){\color[rgb]{0,0,0}\makebox(0,0)[lb]{\smash{$=$}}}%
    \put(0.56949781,0.03570302){\color[rgb]{0,0,0}\makebox(0,0)[lb]{\smash{$,$}}}%
    \put(0.76807739,0.03845123){\color[rgb]{0,0,0}\makebox(0,0)[lb]{\smash{$,$}}}%
    \put(0.89467186,0.03791938){\color[rgb]{0,0,0}\makebox(0,0)[lb]{\smash{$=$}}}%
    \put(0.96665696,0.03774189){\color[rgb]{0,0,0}\makebox(0,0)[lb]{\smash{$.$}}}%
  \end{picture}%
\endgroup%
}}
\end{equation}
If $\cB=\Vect$ then $C$ can be taken to be the finite dual $B^\circ$ \cite{Mon}*{Chapter~9}.

\begin{example}[Nichols algebras] Let $V$ be a Yetter--Drinfeld module over a Hopf algebra $H$ over $\Bbbk$. Then $T(V)$ has a unique structure of a braided Hopf algebra in $\lYD{H}$, generated by primitive elements $v\in V$.
The prototype example of a duality pairing as in Definition \ref{catdualdef} of braided Hopf algebras is given by the self-duality of Nichols algebras (due to \cite{Lus}*{Proposition~1.2.3} in the case of an abelian group). The pairing is obtained by taking quotients by the left and right radical in the induced weak Hopf algebra pairing $T(V)\otimes T(V)\to \Bbbk$ of the tensor algebras. Then there is a non-degenerate pairing $\cB(V)\otimes\cB(V)\to \Bbbk$ for the corresponding Nichols algebra $\cB(V)$. This non-degeneracy characterizes the Nichols algebra $\cB(V)$ as a quotient of $T(V)$, see e.g. \cite{AS} for details.

More generally, one can take any ideal $I$ in $T(V)$ which is homogeneous and generated in degree larger or equal to two and a Yetter--Drinfeld submodule. Then  $T(V)/I$ is a braided Hopf algebra, sometimes called a \emph{pre-Nichols algebra}, cf. \cite{Mas}. As $I$ is contained in the radical of the pairing $T(V)\otimes T(V)\to \Bbbk$, the pairing factors to a weak self-duality of the braided Hopf algebra $T(V)/{I}$.
\end{example}

\begin{lemma}\label{comodmodduality}
Let $C, B$ be dually paired Hopf algebras. Then there are functors of monoidal categories
\begin{align*}
&\Phi_B^{\leftexp{\cop}{C}}\colon \lcomod{B}(\cB)\longrightarrow\lmod{\leftexp{\cop}{C}}(\overline{\cB}), &(V,\delta)\longmapsto (V,a),\\
&a:=(\ev\otimes \ide_V)(\ide_C\otimes \delta)=\vcenter{\hbox{
\begingroup%
  \makeatletter%
  \providecommand\color[2][]{%
    \errmessage{(Inkscape) Color is used for the text in Inkscape, but the package 'color.sty' is not loaded}%
    \renewcommand\color[2][]{}%
  }%
  \providecommand\transparent[1]{%
    \errmessage{(Inkscape) Transparency is used (non-zero) for the text in Inkscape, but the package 'transparent.sty' is not loaded}%
    \renewcommand\transparent[1]{}%
  }%
  \providecommand\rotatebox[2]{#2}%
  \ifx\svgwidth\undefined%
    \setlength{\unitlength}{26.7150203bp}%
    \ifx\svgscale\undefined%
      \relax%
    \else%
      \setlength{\unitlength}{\unitlength * \real{\svgscale}}%
    \fi%
  \else%
    \setlength{\unitlength}{\svgwidth}%
  \fi%
  \global\let\svgwidth\undefined%
  \global\let\svgscale\undefined%
  \makeatother%
  \begin{picture}(1,0.7861199)%
    \put(0,0){\includegraphics[width=\unitlength]{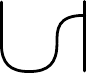}}%
    \put(0.50592518,0.39026647){\color[rgb]{0,0,0}\makebox(0,0)[lb]{\smash{$~$}}}%
  \end{picture}%
\endgroup%
}};\\
&\Phi_{\leftexp{\cop}{C}}^B\colon \lcomod{\leftexp{\cop}{C}}(\overline{\cB})\longrightarrow\lmod{B}(\cB), &(V,\delta)\longmapsto (V,a'),\\
&a':=(\ev\otimes \ide_V)(\Psi_{B,C}\otimes \ide_V)(S\otimes \delta)=\vcenter{\hbox{
\begingroup%
  \makeatletter%
  \providecommand\color[2][]{%
    \errmessage{(Inkscape) Color is used for the text in Inkscape, but the package 'color.sty' is not loaded}%
    \renewcommand\color[2][]{}%
  }%
  \providecommand\transparent[1]{%
    \errmessage{(Inkscape) Transparency is used (non-zero) for the text in Inkscape, but the package 'transparent.sty' is not loaded}%
    \renewcommand\transparent[1]{}%
  }%
  \providecommand\rotatebox[2]{#2}%
  \ifx\svgwidth\undefined%
    \setlength{\unitlength}{26.49642346bp}%
    \ifx\svgscale\undefined%
      \relax%
    \else%
      \setlength{\unitlength}{\unitlength * \real{\svgscale}}%
    \fi%
  \else%
    \setlength{\unitlength}{\svgwidth}%
  \fi%
  \global\let\svgwidth\undefined%
  \global\let\svgscale\undefined%
  \makeatother%
  \begin{picture}(1,1.14264308)%
    \put(0,0){\includegraphics[width=\unitlength,page=1]{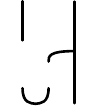}}%
    \put(0.42600781,0.36930284){\color[rgb]{0,0,0}\makebox(0,0)[lb]{\smash{$~$}}}%
    \put(0,0){\includegraphics[width=\unitlength,page=2]{dualaction2.pdf}}%
    \put(0.09493526,0.5759929){\color[rgb]{0,0,0}\makebox(0,0)[lb]{\smash{$S$}}}%
    \put(0,0){\includegraphics[width=\unitlength,page=3]{dualaction2.pdf}}%
  \end{picture}%
\endgroup%
}},
\end{align*}
where $\overline{\cB}$ denotes the monoidal category $\cB$ with the \emph{inverse} braiding $\Psi^{-1}$, and $\leftexp{\cop}{C}$ is $C$ with the same product, and (inverse) opposite coproduct $\Psi^{-1}\Delta_C$. The functor $\Phi_B^{\leftexp{\cop}{C}}$ does not require the existence of antipodes for $B$ and $C$.
\end{lemma}

We call the pairing $\ev$ \emph{non-degenerate} if the above functor $\Phi_{B}^{\leftexp{\cop}{C}}$ is fully faithful. If $\cB=\Vect$ then this is equivalent to the usual definition of a non-degenerate pairing, i.e. that the left and right radicals are zero.


\subsection{Braided Drinfeld Doubles}

We briefly recall the construction of braided Drinfeld doubles (called \emph{double bosonizations} in \cite{Maj3}) for conformity with the conventions of this paper. We now assume that $\cB=\lmod{H}$ for a quasitriangular Hopf algebra $H$ with universal R-matrix $R$ (cf. Section \ref{setup}). Let $C,B$ be (weakly) dual bialgebras in $\cB=\lmod{H}$. We denote the $H$-action on $B$ and $C$ by $\triangleright$ in order to be able to distinguish it from the product in the Drinfeld double defined below.

\begin{definition}\label{braideddrin1}
We define $\Drin_H(C,B)$ to be the $\Bbbk$-bialgebra generated by $C,H,B$ as subalgebras such that
\begin{align}\label{braideddrineq1}
(R^{-(1)}\triangleright b_{(2)})(R^{-(2)}\triangleright c_{(1)})\ev(c_{(2)}, b_{(1)})&=R^{-(1)}c_{(2)}b_{(1)}R^{(2)}\ev(R^{-(2)}\triangleright c_{(1)}, R^{(1)}\triangleright b_{(2)}),
\end{align}
for any $b\in B$, $c\in C$. Further,
\begin{align}
hb&=(h_{(1)}\triangleright b)h_{(2)},& hc&=(h_{(1)}\triangleright c)h_{(2)}, &\forall h\in H.\label{braideddrineq2}
\end{align}
The coproduct is given on generators by
\begin{align}\label{drincop1}
\Delta(h)&=h_{(1)}\otimes h_{(2)}, \\\label{drincop2}
\Delta(b)&=b_{(1)}R^{(2)}\otimes (R^{(1)}\triangleright b_{(2)}),\\\label{drincop3}
\Delta(c)&=R^{-(1)}c_{(2)}\otimes (R^{-(2)}\triangleright c_{(1)}).
\end{align}
The counit is defined on generators using the underlying counits in $C,H$, and $B$, and extended multiplicatively.

If $B$, $C$ are weakly dual Hopf algebras, then $\Drin_H(C,B)$ is $C\otimes H\otimes B$ as a $\Bbbk$-vector space. In this case, the antipode $S_{\Drin}$ is given by
\begin{align}\label{Drinantipode}
S_{\Drin}(h)&=S(h), & S_{\Drin}(b)&=S(R^{(2)})(R^{(1)}\triangleright Sb), &S_{\Drin}(c)&=S(R^{-(1)})(R^{-(2)}\triangleright S^{-1}c).
\end{align}
Further, in this case, Equation (\ref{braideddrineq1}) is equivalent to
\begin{align}\label{crossrel}
cb&=R_1^{-(1)}(R_2^{-(1)}\triangleright b_{(2)})(R_2^{-(2)}\triangleright c_{(2)})R^{(2)}\ev(R_1^{-(2)}\triangleright c_{(1)},R^{(1)}\triangleright Sb_{(3)})\ev(c_{(3)},b_{(1)}).
\end{align}
\end{definition}

Here, the notation $\Delta_B(b)=b_{(1)}\otimes b_{(2)}$ and $\Delta_C(c)=c_{(1)}\otimes c_{(2)}$ is used.
Equivalent to Equation (\ref{braideddrineq2}) we have 
$xh=h_{(2)}(S^{-1}h_{(1)}\triangleright x)$ for $x$ in $B$ or $C$.

If $C,B$ are strongly dual in the sense of Definition \ref{catdualdef}(iii) and we denote the coevaluation morphism $\coev\colon I\to B\otimes C$, by $\coev=e_\alpha \otimes f^\alpha$ (which is a sum of tensors), then $\Drin_H(C,B)$ is quasitriangular with universal R-matrix
\begin{equation}
R_{\Drin}=R^{(1)} f^\alpha \otimes e_\alpha R^{(2)}.
\end{equation}

\begin{proposition}\label{drinfeldprop}
If $C,B$ are weakly dual bialgebras (or Hopf algebras), then $\Drin_H(C,B)$ is a bialgebra (respectively, Hopf algebra), and there is a monoidal functor 
$$\Phi\colon \lYD{B}\longrightarrow \lmod{\Drin_H(C,B)}.$$
If the pairing $\ev$ of $C$ and $B$ is non-degenerate, then this functor is fully faithful. 

If $B,C$ are strongly dual, then $\Phi$ is an equivalence of braided monoidal categories. 
\end{proposition}
\begin{proof}[Proof (sketch).]
Given an object $V$ in $\lYD{B}$, define $\Phi(V)=V$ with the same left $H$-action and $B$-action. The $C$-coaction is given by $c\triangleright v=\ev(c\otimes v^{(-1)})v^{(0)}$, for any $c\in C$, $v\in V$. This assignment becomes a functor by setting $\Phi(f)=f$ for any morphism $f\colon V\to W$ in $\lYD{B}$.

The result is due to \cite{Maj3}, with slightly different relations (as there a different category of Yetter--Drinfeld modules is used, see \cite{Maj3}*{Appendix~B}, and $B$ is regarded as a \emph{right} $H$-module).

Consider the monoidal functor $\Phi_{B}^{\leftexp{\cop}{C}}$ from Lemma \ref{comodmodduality} to produce a monoidal functor from $\lYD{B}$ into a category $\leftexp{C,B}{\mathbf{YD}}$ of simultaneous left $B$-modules and $\leftexp{\cop}{C}$-modules $V$ which are compatible via the condition
\begin{align}\label{drinfeldcomp}
&a_B(\ide_B\otimes a_C)(\Psi_{B,C}^{-1}\otimes\ide)(\ide_C\otimes \ev\otimes \ide_{B\otimes V})(\Delta_C\otimes \Delta_B\otimes \ide_V)\\
&=(\ev\otimes \ide_V)(\ide\otimes \Psi_{B,V})(\ide_C\otimes a_C\otimes \ide_B)(\ide\otimes a_B\otimes \ide_B)(\ide\otimes \Psi_{B,V})(\Delta_C\otimes \Delta_B\otimes \ide_V),
\end{align}
where $a_B$ denotes the $B$-action, and $a_C$ the $\leftexp{\cop}{C}$-action. It is not hard to check that this category is monoidal and braided monoidal if $B$ and $C$ are strongly dual. If $B,C$ are Hopf algebras, the subcategory on finite-dimensional objects is rigid. Since $\cB=\lmod{H}$ for a quasitriangular Hopf algebra $H$, we can now apply a type of Tannaka--Krein duality (see e.g. \cite{Maj1}*{Chapter~9}) to the monoidal category $\leftexp{C,B}{\mathbf{YD}}$, using the forgetful functor  $$\rF\colon \leftexp{C,B}{\mathbf{YD}}\longrightarrow \Vect.$$
Note that the representability condition \cite{Maj1}*{Equation (9.40)} holds for $\Drin_H(C,B)$. That is, there are natural bijections
\begin{align*}\Hom_{\Vect}(V, \Drin_H(C,B))\cong \Nat(V\otimes \rF^n,\rF^n), &&\forall n\geq 0.\end{align*}
Hence \cite{Maj1}*{Theorem 9.4.6} gives a monoidal functor
\begin{align*}
\Phi'\colon \leftexp{C,B}{\mathbf{YD}}\longrightarrow \lmod{\Drin_H(C,B)}
\end{align*}
which is in fact an equivalence of categories. This follows as it can be explicitly described as follows: Given an object $V$ of $\leftexp{C,B}{\mathbf{YD}}$, the actions of $H$, $B$, and $C$ on $\Phi'(V)$ are given by the original actions on $V$, and hence uniquely determine $V$. This shows that $\Phi'$ is essentially surjective and full. It is faithful since $\rF$ is faithful. The relations for $\Drin_H(C,B)$ are now obtained via Tannaka--Krein reconstruction, and the general theory ensures that $\Drin_H(C,B)$ is a bialgebra (or a Hopf algebra, provided that $C,B$ are), cf. \cite{Maj1}*{Section 9.4}.

To find the formulas for the antipode in the Hopf algebra case, the identities
\begin{align*}
R_2^{(1)}R_1^{(1)}\otimes R_1^{(2)}S(R_2^{(2)})=1\otimes 1, &&R_1^{-(1)}S(R_2^{-(1)})\otimes R_2^{-(2)}R_1^{-(2)}=1\otimes 1,
\end{align*}
are used. Note that if the duality is strong, then the functor $\Phi_B^{\leftexp{\cop}{C}}$ from Lemma \ref{comodmodduality} gives an equivalence $\lYD{B}\simeq \leftexp{C,B}{\mathbf{YD}}$.
\end{proof}

\begin{remark}
If the pairing $\ev$ is convolution invertible, then $\Drin_H(C,B)$ is isomorphic to $C\otimes H\otimes B$ as a $\Bbbk$-vector space even if $C,B$ are no Hopf algebras. This follows using the approach of \cites{Maj10,KS}. 
\end{remark}


\begin{example}
For $H=\Bbbk$, the trivial Hopf algebra, and $B$ a finite-dimensional Hopf algebra, we define $$\Drin_{\Bbbk}(B):=\Drin_\Bbbk( \leftexp{\cop}{B},\leftexp{\cop}{B^*}).$$
Here, $B^*=\Hom_\Bbbk(B,\Bbbk)$ is the dual of $B$, which becomes a bialgebra by defining 
\begin{align}
(\Delta f )(h\otimes g)&=f(hg), & (e\cdot f)(h)=e(h_{(1)})f(h_{(2)}), &&\forall h,g\in B, e,f\in B^*.
\end{align}
There is strong duality pairing $$\ev\colon  \leftexp{\cop}{B}\otimes \leftexp{\cop}{B^*} \longrightarrow\Bbbk, \qquad h\otimes f\longmapsto f(h)$$
in the sense of Definition \ref{catdualdef}(iii).

Equivalently, this pairing gives a pairing 
$$\langle~,~\rangle\colon  B\otimes B^* \longrightarrow\Bbbk, \qquad h\otimes f\longmapsto f(h),$$
which satisfies the conditions 
\begin{align}
\langle hg,f \rangle=\langle h,f_{(1)}\rangle\langle g,f_{(2)}\rangle, &&\langle h,ef \rangle=\langle h_{(1)},e\rangle\langle h_{(2)},f\rangle.
\end{align}
The Hopf algebra $\Drin_\Bbbk(B)$ defined this way recovers the usual \emph{Drinfeld double} (or \emph{quantum double}) in the form containing $B^{*\cop}$ as found, for example, in \cite{Maj1}*{Theorem 7.1.2}.

For completeness, we give an explicit presentation for $\Drin_{\Bbbk}(B)$. The algebra $\Drin_\Bbbk(B)$ is defined on $B\otimes B^*$ subject to the relation
\begin{align*}
h_{(1)}f_{(2)}\ev(h_{(2)},f_{(1)})=f_{(1)}h_{(2)}\ev(h_{(1)},f_{(2)}),&&\forall h\in B, f\in B^*.
\end{align*}
The coproduct and counit are the respective structures on $(B^*)^{\cop}\otimes B$.
\end{example}

\begin{example}
Now assume that $C,B$ are primitively generated. That is, as algebras, the sets $P(C)$ and $P(B)$ generate $C$, respectively $B$. Here, $b\in P(B)$ and $c\in P(C)$ if and only if
\begin{align*}
\Delta_B(b)=b\otimes 1+1\otimes b, &&\Delta_C(c)=c\otimes 1+1\otimes c.
\end{align*}
Denote the braided commutator by
\begin{align*}
[c,b]_{\Psi}&=cb-m\Psi(c\otimes b)=cb-(R^{(2)}\triangleright b)(R^{(1)}\triangleright c).
\end{align*}
The algebra $\Drin_H(C,B)$ is then generated by $b\in P(B)$, $c\in P(C)$ and $h\in H$ such that
\begin{align}\label{comm1}
[c,b]_{\Psi^{-1}}&=\ev(c,b)-R^{-(1)}R^{(2)}\ev(R^{-(2)}\triangleright c,R^{(1)}\triangleright b).
\end{align}
Or, equivalently,
\begin{align}\label{comm2}
[b,c]_{\Psi}&={R^{(2)}}_{(1)}R^{-(1)}\ev(R^{-(2)}{R^{(2)}}_{(2)}\triangleright c,R^{(1)}\triangleright b)-\ev(R^{(2)}\triangleright c,R^{(1)}\triangleright b),
\end{align}
together with the relations from Equation (\ref{braideddrineq2}) and those from the products of $B,C$, and $H$.

The equivalence of the two commutator relations Eqs. (\ref{comm1}) and (\ref{comm2}) is seen by precomposing with $\Psi_{C,B}$ in Equation (\ref{comm1}) to produce Equation (\ref{comm2}), under use of the axioms for the universal R-matrix, or more conveniently, using graphical calculus.

For primitive elements, the coproduct formulas from (\ref{drincop2})--(\ref{drincop3}) become
\begin{align*}
\Delta_{\Drin}(b)&=b\otimes 1+R^{(2)}\otimes (R^{(1)}\triangleright b),\\
\Delta_{\Drin}(c)&=c\otimes 1+R^{-(1)}\otimes (R^{-(2)}\triangleright c),
\end{align*}
with counit
\begin{equation*}
\varepsilon_{\Drin}(b)=\varepsilon_{\Drin}(c)=0,\qquad \varepsilon_{\Drin}(h)=\varepsilon(h).
\end{equation*}
The antipode is given by the same formulas as in Equation (\ref{Drinantipode2}).
\end{example}


\subsection{Braided Crossed Products}

Given an algebra object $A$ in $\lmod{B}=\lmod{B}(\cB)$. We recall the construction of the algebra $A\rtimes B$ from \cite{Maj6}*{Proposition~2.6}. The defining feature of $A\rtimes B$ is that $\lmod{A\rtimes B}$ is equivalent to $\lmod{A}(\lmod{B})$. 

The algebra $A\rtimes B$ is given on the tensor product $A\otimes B$ in $\cB$ with product
\begin{equation}\label{crossproduct}
m_{A\otimes B}= (m_A\otimes m_B)(\ide_A\otimes a_A\otimes \ide_{B\otimes B})(\ide_{A\otimes B}\otimes \Psi_{B,A}\otimes \ide_{B})(\ide_A\otimes \Delta_B\otimes \ide_{A\otimes B}).
\end{equation}
Using graphical calculus, the multiplication becomes
\begin{equation}
m_{A\otimes B}=\vcenter{\hbox{
\begingroup%
  \makeatletter%
  \providecommand\color[2][]{%
    \errmessage{(Inkscape) Color is used for the text in Inkscape, but the package 'color.sty' is not loaded}%
    \renewcommand\color[2][]{}%
  }%
  \providecommand\transparent[1]{%
    \errmessage{(Inkscape) Transparency is used (non-zero) for the text in Inkscape, but the package 'transparent.sty' is not loaded}%
    \renewcommand\transparent[1]{}%
  }%
  \providecommand\rotatebox[2]{#2}%
  \ifx\svgwidth\undefined%
    \setlength{\unitlength}{45.43224335bp}%
    \ifx\svgscale\undefined%
      \relax%
    \else%
      \setlength{\unitlength}{\unitlength * \real{\svgscale}}%
    \fi%
  \else%
    \setlength{\unitlength}{\svgwidth}%
  \fi%
  \global\let\svgwidth\undefined%
  \global\let\svgscale\undefined%
  \makeatother%
  \begin{picture}(1,0.9181643)%
    \put(0,0){\includegraphics[width=\unitlength,page=1]{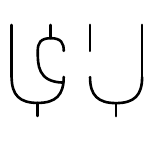}}%
    \put(-0.00822181,0.8414365){\color[rgb]{0,0,0}\makebox(0,0)[lb]{\smash{$A$}}}%
    \put(0.23939969,0.8414365){\color[rgb]{0,0,0}\makebox(0,0)[lb]{\smash{$B$}}}%
    \put(0.48702118,0.8414365){\color[rgb]{0,0,0}\makebox(0,0)[lb]{\smash{$A$}}}%
    \put(0.81423552,0.84291058){\color[rgb]{0,0,0}\makebox(0,0)[lb]{\smash{$B$}}}%
    \put(0.15538511,0.01455744){\color[rgb]{0,0,0}\makebox(0,0)[lb]{\smash{$A$}}}%
    \put(0.65504984,0.01897968){\color[rgb]{0,0,0}\makebox(0,0)[lb]{\smash{$B$}}}%
    \put(0,0){\includegraphics[width=\unitlength,page=2]{crossprod.pdf}}%
  \end{picture}%
\endgroup%
}}~.
\end{equation}

Given weakly left dual bialgebras $C$ for $B$, recall the monoidal functor $\Psi_{B}^{\leftexp{\cop}{C}}$ from Lemma \ref{comodmodduality}. Given a $B$-comodule algebra $A$, $A$ becomes a $\leftexp{\cop}{C}$-module algebra (in $\overline{\cB}$) with action given by 
$c\triangleright a=a^{(0)}\ev(c,a^{(-1)})$, for all  $a\in A, c\in C$,
where $\delta(a)=a^{(-1)}\otimes a^{(0)}\in B\otimes A$.
Hence, we can consider the crossed product $A\rtimes \leftexp{\cop}{C}$. Now again assume $\cB=\lmod{H}$ for a quasitriangular Hopf algebra $H$ over $\Bbbk.$ We give a presentation for the algebra $A\rtimes \leftexp{\cop}{C}$. The product from Equation (\ref{crossproduct}) gives that $A\rtimes \leftexp{\cop}{C}=A\otimes C$ with $A$, $C$ as subalgebras and the additional relations
\begin{align}\label{crossproduct2}
ca&=(R^{-(1)}\triangleright a^{(0)})(R^{-(2)}\triangleright c_{(1)})\ev(c_{(2)},a^{(-1)}), &&\forall a\in A, c\in C.
\end{align}
We also need the iterated cross product algebra $A\rtimes \leftexp{\cop}{C}\rtimes H$. This algebra is defined on $A\otimes C\otimes H$, with $A$, $C$, and $H$ as subalgebras, and relations Equation (\ref{crossproduct2}) as well as
\begin{align}
ha&=(h_{(1)}\triangleright a)h_{(2)}, &hc&=(h_{(1)}\triangleright c)h_{(2)}, &\forall h\in H, a\in A, c\in C.
\end{align}

\begin{example}[Braided Heisenberg double]\label{braidedheis}
In Example \ref{regularcoaction} we saw that $B$ itself is naturally a left $B$-comodule algebra in $\cB$, denoted by $B^{\op{reg}}$. Applying the construction above, we obtain the braided cross product $$\Heis_H(C,B)=B^{\op{reg}}\rtimes \leftexp{\cop}{C}\rtimes H$$ which is referred to as the \emph{braided Heisenberg double} (cf. \cite{Lau}*{Section~2.3}). It is generated as an algebra by the subalgebras $C,B$, and $H$ subject to the relations Equation (\ref{braideddrineq2}) and
\begin{align}
cb=(R^{-(1)}\triangleright b_{(2)})(R^{-(2)}\triangleright c_{(1)})\ev(c_{(2)},b_{(1)}),&&\forall c\in C, b\in B.
\end{align}
\end{example}
If, for example, $B$, $C$ are primitively generated, this relation becomes
\begin{align}
[c,b]_{\Psi^{-1}}&=\ev(c,b).
\end{align}

\begin{example}[Twisted tensor product algebra]\label{twistedtensor}
Recall from Example \ref{trivialcoaction} that $B$ is also a left $B$-comodule algebra, denoted by $B^{\triv}$, with respect to the trivial coaction. In this case, we denote $B^{\triv}\rtimes\leftexp{\cop}{C}$ by $B\otimes_{\Psi^{-1}}C$, which indices that this is a braided tensor product algebra, with product given by
\begin{equation}
m_{B\otimes_{\Psi^{-1}} C}=(m_B\otimes m_C)(\ide_B\otimes \Psi^{-1}_{C,B}\otimes \ide_C).
\end{equation}
Hence, in $B\otimes_{\Psi^{-1}}C\rtimes H$, the multiplication is given by
\begin{equation}
bhc\cdot b'h'c'= b(h_{(1)}R^{-(1)}\triangleright b')h_{(2)}h'_{(2)}(S^{-1}(h'_{(1)})R^{-(2)}\triangleright c)c'.
\end{equation}
\end{example}

\subsection{Comodule Algebras over Braided Drinfeld Doubles}\label{comodulealgebras}

We now generalize the results of Theorem \ref{mainthm} and \ref{mainthm2} to $\Drin_H(C,B)$-comodule algebra structures on certain cross product algebras, for $\cB=\lmod{H}$.

\begin{corollary}\label{comodulealgebracor}Let $C, B$ be weakly dual bialgebras in $\cB$.
Let $A$ be a left $B$-comodule algebra in $\cB=\lmod{H}$. Then $A\rtimes \leftexp{\cop}{C}\rtimes H$ is a left $\Drin_H(C,B)$-comodule algebra (with $A\rtimes \leftexp{\cop}{C}$ as a comodule subalgebra). The coaction $\delta_{\Drin}$ is given by
\begin{align}
\delta_{\Drin}(a)&=a^{(-1)}R^{(2)}\otimes (R^{(1)}\triangleright a^{(0)}), \\
\delta_{\Drin}(c)&=R^{-(1)}c_{(2)}\otimes (R^{-(2)}\triangleright c_{(1)}), \\
 \delta_{\Drin}(h)&=h_{(1)}\otimes h_{(2)},
\end{align}
for $a\in A$, $c\in C$, $h\in H$, and $\delta(a)=a^{(-1)}\otimes a^{(0)}$ denoting the $B$-coaction on $A$.
\end{corollary}
\begin{proof}
First, the monoidal functor $\Phi_B^{\leftexp{\cop}{C}}$ from Lemma \ref{comodmodduality} applied to the $B$-coaction gives a functor 
$$\Phi_A\colon \lmod{A}(\lcomod{B})\longrightarrow \lmod{A}(\lmod{\leftexp{\cop}{C}}).$$
Recall that by Proposition \ref{drinfeldprop}, there is a monoidal functor $\Psi\colon \lYD{B}\to \leftexp{C,B}{\mathbf{YD}}.$

Let $V\in \lYD{B}$ and $W\in \lmod{A}(\lcomod{B})$. By Theorem \ref{mainthm}, we have that $V\triangleright W$ is an object in $\lmod{A}(\lcomod{B})$. As a left $\leftexp{\cop}{C}$-module, $\Phi_A(V\triangleright W)$ is isomorphic to $\Psi(V)\otimes \Phi_A(W)$. If we define the $A$-action on $\Psi(V)\otimes \Phi_A(W)$ by the same formula as in Eq. (\ref{Aactioneq}), then these objects become isomorphic in $\lmod{A}(\lmod{\leftexp{\cop}{C}})$. More generally, for any object $V$ of $\leftexp{C,B}{\mathbf{YD}}$ and $\lmod{A}(\lmod{\leftexp{\cop}{C}})$, the tensor product $\leftexp{\cop}{C}$-module $V\otimes W$ becomes an object in $\lmod{A}(\lmod{\leftexp{\cop}{C}})$ again with the same $A$-action as in Eq. (\ref{Aactioneq}).

The result now follows under the equivalences $\lmod{A}(\lmod{\leftexp{\cop}{C}})\simeq \lmod{A\rtimes \leftexp{\cop}{C}\rtimes H}$ and $\leftexp{C,B}{\mathbf{YD}}\simeq \lmod{\Drin_H(C,B)},$ using a reconstruction statement as for example found in \cite{Lau}*{Proposition 3.8.4}.
\end{proof}

\begin{example}[Braided Heisenberg double]\label{Heisexample}
The algebra $\Heis_H(C,B)$ from Example \ref{braidedheis} is a $\Drin_H(C,B)$-comodule algebra with coaction given by
\begin{align}
\delta(b)=b_{(1)}R^{(2)}\otimes (R^{(1)}\triangleright b_{(2)}),&&
\delta(c)=R^{(-1)}c_{(2)}\otimes (R^{-(2)}\triangleright c_{(1)}), &&
 \delta(h)=h_{(1)}\otimes h_{(2)},
\end{align}
This recovers the result from \cite{Lau}*{Corollary 3.5.8}, generalizing \cite{Lu}*{Corollary 6.4}.
\end{example}

\begin{example}
For the twisted tensor product algebra $B\otimes_{\Psi^{-1}}C\rtimes H$ from Example \ref{twistedtensor}, we find that the coaction is determined by
\begin{align}
\delta(b)=R^{(2)}\otimes (R^{(1)}\triangleright b),
\end{align}
as well as the same formulas for $\delta(c)$, $\delta(h)$.
\end{example}

We can produce a version of Corollary \ref{comodulealgebracor} valid for right $C$-comodule coalgebras. For this, we shall assume from now on that $C,B$ are weakly dual \emph{Hopf algebras} in $\cB$. We need the following technical lemma:

\begin{lemma}Let $C,B$ are weakly dual Hopf algebras in $\cB$. 
There is a monoidal functor 
$$\Psi\colon \rYD{C}\longrightarrow \lmod{\Drin_H(C,B)}.$$
If the pairing $\ev$ of $C$ and $B$ is non-degenerate, then this functor is fully faithful. 

If $C,B$ are strongly dual, then $\Psi$ is part of an equivalence of braided monoidal categories.
\end{lemma}
\begin{proof}
Using Lemma \ref{rightleft}, there is an equivalence of categories $\rYD{C}(\cB)\simeq\lYD{\leftexp{\cop}{C}}(\overline{\cB})$. Combine this with the functor $\Phi^B_{\leftexp{\cop}{C}}$ from Lemma \ref{comodmodduality}, and we obtain from the data of a right YD-module over $C$, a left $B$-action together with a left $\leftexp{\cop}{C}$-action, such that tensor products can be formed. It remains to verify that these induced actions satisfy the a compatibility condition equivalent to Equation (\ref{drinfeldcomp}). Indeed, translating the compatibility condition of Equations (\ref{YDcond}), (\ref{ydpic})
under $\Phi^B_{\leftexp{\cop}{C}}$ gives the equation
\begin{align}\begin{split}
&a_C(\ide_C\otimes\ev\otimes a_B)(\Delta_C\otimes (S\otimes \ide_B)\Delta_B\otimes \ide_V)=(\ev\otimes \ide_V)(\ide_C\otimes \Psi_{V,B})\\&(\ide_C\otimes a_B\otimes\ide_B)(\ide_{C\otimes B}\otimes a_C\otimes \ide_B)(\ide_C\otimes \Psi^{-1}_{B,C}\otimes\Psi_{B,V})(\Delta_C\otimes(\ide\otimes S)\Delta_B\otimes\ide_V).\end{split}
\end{align}
Applying the antipode axioms for $S$ twice, we see that this equations is equivalent to Equation (\ref{drinfeldcomp}).
Hence, there is a monoidal functor $\rYD{C}\to \lmod{\Drin_H(C,B)}$ via a Tannaka--Krein reconstruction argument as in Proposition \ref{drinfeldprop}.
\end{proof}

Given a right $C$-comodule algebra $A$, the above functors make $A$ a right $B$-module algebra. We provide a presentation for $A\rtimes B$ using this induced $B$-action 
\begin{equation}\label{inducedBaction}
a':=(\ide_A\otimes \ev)(\delta_A\otimes \ide_B)\Psi_{B,A}(S\otimes \ide_A)=\vcenter{\hbox{
\begingroup%
  \makeatletter%
  \providecommand\color[2][]{%
    \errmessage{(Inkscape) Color is used for the text in Inkscape, but the package 'color.sty' is not loaded}%
    \renewcommand\color[2][]{}%
  }%
  \providecommand\transparent[1]{%
    \errmessage{(Inkscape) Transparency is used (non-zero) for the text in Inkscape, but the package 'transparent.sty' is not loaded}%
    \renewcommand\transparent[1]{}%
  }%
  \providecommand\rotatebox[2]{#2}%
  \ifx\svgwidth\undefined%
    \setlength{\unitlength}{22.63010775bp}%
    \ifx\svgscale\undefined%
      \relax%
    \else%
      \setlength{\unitlength}{\unitlength * \real{\svgscale}}%
    \fi%
  \else%
    \setlength{\unitlength}{\svgwidth}%
  \fi%
  \global\let\svgwidth\undefined%
  \global\let\svgscale\undefined%
  \makeatother%
  \begin{picture}(1,1.39242297)%
    \put(0,0){\includegraphics[width=\unitlength,page=1]{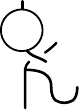}}%
    \put(0.11115479,0.86152832){\color[rgb]{0,0,0}\makebox(0,0)[lb]{\smash{$S$}}}%
    \put(0,0){\includegraphics[width=\unitlength,page=2]{dualaction3.pdf}}%
  \end{picture}%
\endgroup%
}}.\end{equation}
The algebra $A\rtimes B$ is defined on $A\rtimes B$, with subalgebras $A$, $B$ and the relation
\begin{align}\label{crossrelation2}
\begin{split}ba&=(R^{(2)}\triangleright a)^{(0)}(R^{(1)}\triangleright b)_{(2)}\ev((R^{(2)}\triangleright a)^{(-1)},S((R^{(1)}\triangleright b)_{(1)}))\\
&=({R^{(2)}}_{(1)}\triangleright a^{(0)})({R^{(1)}}_{(2)}\triangleright b_{(2)})\ev({R^{(2)}}_{(2)}\triangleright a^{(-1)},{R^{(1)}}_{(1)}\triangleright S(b_{(1)})).\end{split}
\end{align}
Here, we denote $\delta_A(a)=a^{(0)}\otimes a^{(-1)}$. 

We can now provide a version of Theorem \ref{mainthm} working with right $C$-module coalgebras.

\begin{corollary}\label{maincor2}
Let $C,B$ are weakly dual Hopf algebras in $\cB$, and let $A$ be an algebra object in $\rcomod{C}$. Given objects $(V,a_V,\delta_V)$ in $\rYD{C}$, and $(W,a_W,\delta_W)$ in $\lmod{A}(\rcomod{C})$, their tensor product $V\otimes W$ becomes an object $V\triangleright W$ of $\lmod{A}(\rcomod{C})$ with $A$-action $a_{V\triangleright W}$ and $C$-coaction $\delta_{V\triangleright W}$ given by
\begin{align}
\begin{split}a_{V\triangleright W}&:=(\ide_W\otimes a_W)(\Psi^{-1}_{A,V}\otimes \ide_W)(\ide_A\otimes a_V\Psi^{-1}_{C,V}\otimes \ide_W)((\ide_A\otimes S^{-1})\delta_A\otimes \ide_{V\otimes W})\\&=\vcenter{\hbox{
\begingroup%
  \makeatletter%
  \providecommand\color[2][]{%
    \errmessage{(Inkscape) Color is used for the text in Inkscape, but the package 'color.sty' is not loaded}%
    \renewcommand\color[2][]{}%
  }%
  \providecommand\transparent[1]{%
    \errmessage{(Inkscape) Transparency is used (non-zero) for the text in Inkscape, but the package 'transparent.sty' is not loaded}%
    \renewcommand\transparent[1]{}%
  }%
  \providecommand\rotatebox[2]{#2}%
  \ifx\svgwidth\undefined%
    \setlength{\unitlength}{57.17458272bp}%
    \ifx\svgscale\undefined%
      \relax%
    \else%
      \setlength{\unitlength}{\unitlength * \real{\svgscale}}%
    \fi%
  \else%
    \setlength{\unitlength}{\svgwidth}%
  \fi%
  \global\let\svgwidth\undefined%
  \global\let\svgscale\undefined%
  \makeatother%
  \begin{picture}(1,0.72818993)%
    \put(0,0){\includegraphics[width=\unitlength]{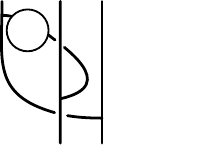}}%
    \put(0.3107874,0.39705876){\color[rgb]{0,0,0}\makebox(0,0)[lb]{\smash{$~$}}}%
    \put(0.05076752,0.51833444){\color[rgb]{0,0,0}\makebox(0,0)[lb]{\smash{\tiny{$S^{\text{-}1}$}}}}%
    \put(0.58451727,0.35428654){\color[rgb]{0,0,0}\makebox(0,0)[lb]{\smash{$,$}}}%
  \end{picture}%
\endgroup%
}}\end{split}\\
\delta_{V\triangleright W}&:=(\ide_{V\otimes W}\otimes m_B)(\ide_V\otimes \Psi_{B,W}\otimes \ide_B)(\delta_V\otimes \delta_W)=\vcenter{\hbox{}}.
\end{align}
\end{corollary}
\begin{proof}
This follows under use of Lemma \ref{rightleft} from Theorem \ref{mainthm}.
\end{proof}

\begin{corollary}\label{comodulealgebracor2}
Let $C,B$ are weakly dual Hopf algebras in $\cB$, and let $A$ be a right $C$-comodule algebra in $\cB=\lmod{H}$. Then $A\rtimes B\rtimes H$ is a left $\Drin_H(C,B)$-comodule algebra (with $A\rtimes B$ as a comodule subalgebra), where the left $B$-module structure is obtained from the right $C$-comodule structure via Equation (\ref{inducedBaction}). The coaction $\delta_{\Drin}$ is given by
\begin{align}
\delta_{\Drin}(a)&=R^{-(1)}a^{(-1)}\otimes (R^{-(2)}\triangleright a^{(0)}), \\
\delta_{\Drin}(b)&=b_{(1)}R^{(2)}\otimes (R^{(1)}\triangleright b_{(2)}), \\
 \delta_{\Drin}(h)&=h_{(1)}\otimes h_{(2)},
\end{align}
for $a\in A$, $c\in C$, $h\in H$, and $\delta(a)=a^{(0)}\otimes a^{(-1)}$ denoting the $C$-coaction on $A$.
\end{corollary}
\begin{proof}
The proof follows the same strategy as the proof of Corollary \ref{comodulealgebracor}. Note that the $A$-action of Corollary \ref{maincor2} can be rewritten, using the passage from a right $C$-action to a left $\leftexp{\cop}{C}$-action from Lemma \ref{rightleft}, as 
\begin{align}
a_{V\triangleright W}&=(\ide_V\otimes a_W)(\Psi^{-1}_{V,A}\otimes \ide_W)(\ide_A\otimes b_V\otimes \ide_W)(\delta_A\otimes \ide_{V\otimes W}).
\end{align}
Thus, we find the formula for $\delta_{\Drin}(a)$ above. Further combine Lemma \ref{rightleft} and Lemma \ref{comodmodduality} to induce a left $B$-action $b$ from the right $C$-coaction $\delta$ by the formula
$$b=(\ide_V\otimes \ev)(\delta\otimes S)\Psi_{B,V},$$
showing that $\delta_{\Drin}(b)$ is just giving by the coproduct of $B\rtimes H$.
The $H$-actions is simply given by $\Delta_H$.
\end{proof}

\begin{example}
For a finite group $G$, let $A$ be a left $G$-comodule algebra. Then $A\rtimes \Bbbk [G]$ becomes a left $\Drin(G)$-coalgebra. Here, $\Bbbk[G]=\Hom_{\Bbbk}(G,\Bbbk)$ with coproduct
$$\Delta(\delta_g)=\sum_{a,b: ab=g}\delta_a\otimes \delta_b,$$
for $\delta_g(h)=\delta_{g,h}$. Further $\Drin(G)=\Bbbk[G]\otimes \Bbbk G$ with relations 
\begin{align*}
g\delta_h=\delta_{ghg^{-1}} g, && \forall g,h\in G,
\end{align*}
and $$\Delta_{\Drin}(g\delta_h)=\sum_{a,b: ab=g}{g\delta_a \otimes g\delta_b}.$$
The $\Drin(G)$-coaction on $A\rtimes \Bbbk [G]$ is given by 
\begin{align*}
\delta_{\Drin}(a)=|a|\otimes a, && \delta_{\Drin}(\delta_h)=\Delta(\delta_h),
\end{align*}
where $a$ is homogeneous of $G$-degree $|a|$.

Dual to this, assume that $B$ is a right $G$-module algebra. Then $B$ becomes a right $\Bbbk[G]$-comodule algebra via dualizing, i.e.
$$\delta(b)=\sum_{g\in G} (b\triangleleft g)\otimes \delta_g.$$
Now, $B\rtimes \Bbbk G$ is a left $\Drin(G)$-comodule algebra, with coaction given by
\begin{align*}
\delta_{\Drin}(b)&=\sum_{g\in G} \delta_g\otimes (b\triangleleft g), & \delta_{\Drin}(g)&=g\otimes g.
\end{align*}
\end{example}

\subsection{Weak Quasitriangular Structures}\label{weakquasisect}

In order to work with quantum groups, we present versions of the above results for $\cB$ having a braiding via a \emph{weak quasitriangular pair} of braided Hopf algebras following \cite{Maj3}. The basic idea is to generalize Definition \ref{braideddrin1} to a weaker context than $H$ being quasitriangular suitable for infinite-dimensional $H$.

Assume given a pair of weakly dual Hopf algebras $A$, $H$ over $\Bbbk$ with pairing $\langle~,~\rangle \colon H\otimes A\to \Bbbk$ satisfying
\begin{gather}
\langle \Delta(h),a\otimes b\rangle=\langle h_{(1)},b\rangle\langle h_{(2)},a\rangle, \qquad\langle h\otimes k,\Delta(a)\rangle=\langle h,a_{(2)}\rangle\langle k,a_{(1)}\rangle,\\ \langle 1, a\rangle=\varepsilon(a),\qquad \langle h,1\rangle=\varepsilon(h),\\
\langle Sh,a\rangle =\langle h,Sa\rangle.
\end{gather}
Note that this is a different pairing than in \cite{Maj3} (to make the connection, replace $H$ by $\leftexp{\cop}{H}^{\oop}$, i.e. using the opposite product and coproduct).

\begin{definition}\label{weakquasi}
A \emph{weak quasitriangular pair} of Hopf algebras is the data of Hopf algebras $A$, $H$ with a weak pairing $\langle~,~\rangle \colon H\otimes A\to \Bbbk$, and convolution-invertible morphisms of Hopf algebra $R, \overline{R}\colon A\to H^{\oop}$ (that is, these morphisms are anti-algebra maps, but coalgebra maps) satisfying the following axioms:
\begin{align}
\langle \overline{R}(a),b \rangle & = \langle R^{-\ast}(b),a \rangle,\\
R(a_{(2)})h_{(1)}\ev(h_{(2)},a_{(1)})&=h_{(2)}R(a_{(1)})\ev(h_{(1)},a_{(2)}),\\
\overline{R}(a_{(2)})h_{(1)}\ev(h_{(2)},a_{(1)})&=h_{(2)}\overline{R}(a_{(1)})\ev(h_{(1)},a_{(2)}).
\end{align}
\end{definition}

It was shown in \cite{Maj3} that the braided Drinfeld double (i.e. the double bosonization) can still be defined as a Hopf algebra given that $H, A$ are a weak quasitriangular pair of Hopf algebras. This versions is important to include Lusztig's version of the quantum groups $U_q(\mathfrak{g})$ as example.

\begin{definition}\label{braideddrin3}
Assume given a weakly quasitriangular pair $A,H$, and weakly dual bialgebras $B,C$ in $\lcomod{A}$.
 We define $\Drin_H(C,B)$ to be the Hopf algebra generated by $C,H,B$ as subalgebras such that
\begin{align}\label{braideddrineq1b}
\begin{split}
(R^{-\ast}({c_{(1)}}^{(-1)})\triangleright b_{(2)}){c_{(1)}}^{(0)}\ev(c_{(2)}, b_{(1)})&=R^{-\ast}({c_{(1)}}^{(-1)})c_{(2)}b_{(1)}\overline{R}^{-\ast}({b_{(2)}}^{(-1)})\ev({c_{(1)}}^{(0)}, {b_{(2)}}^{(0)}),\\
\Leftrightarrow\quad {b_{(2)}}^{(0)}(\overline{R}({b_{(2)}}^{(-1)})\triangleright c_{(1)})\ev(c_{(2)}, b_{(1)})&=R^{-\ast}({c_{(1)}}^{(-1)})c_{(2)}b_{(1)}\overline{R}^{-\ast}({b_{(2)}}^{(-1)})\ev({c_{(1)}}^{(0)}, {b_{(2)}}^{(0)}),
\end{split}
\end{align}
for any $b\in B$, $c\in C$. Further,
\begin{align}
hb&=(h_{(2)}\triangleright b)h_{(1)},& hc&=(h_{(2)}\triangleright c)h_{(1)}, &\forall h\in H.\label{braideddrineq2b}
\end{align}
The coproduct is given on generators by
\begin{align}
\Delta(h)&=h_{(2)}\otimes h_{(1)}, \\
\Delta(b)&=b_{(1)}\overline{R}^{-\ast}({b_{(2)}}^{(-1)})\otimes {b_{(2)}}^{(0)},\\
\Delta(c)&=R^{-\ast}({c_{(1)}}^{(-1)})c_{(2)}\otimes {c_{(1)}}^{(0)}.
\end{align}
The counit is defined on generators using the underlying counits in $C,H$, and $B$, and extended multiplicatively (as before).

If $B$, $C$ are dually paired Hopf algebras, then $\Drin_H(C,B)$ is $C\otimes H\otimes B$ as a $\Bbbk$-vector space, and the antipode $S_{\Drin}$ is given by
\begin{align}\label{Drinantipode2}
S_{\Drin}(h)&=S^{-1}(h), & S_{\Drin}(b)&=S\overline{R}^{-\ast}(b^{(-1)})S(b^{(0)}), &S_{\Drin}(c)&=SR^{-\ast}(c^{(-1)})S^{-1}(c^{(0)}).
\end{align}
\end{definition}
\begin{proof}[Proof (sketch)]
We skip the proof that the above version of the braided Drinfeld double is again a Hopf algebra and refer the reader to \cite{Maj3} where a slightly different version of this Hopf algebra is presented. Note that the relations can be translated from Definition \ref{braideddrin1} via
\begin{align}
R(a)&=\langle R^{(2)},a\rangle R^{(1)}, &R^{-\ast}(a)&=\langle R^{-(2)},a\rangle R^{-(1)},\\
\overline{R}(a)&=\langle R^{-(1)},a\rangle R^{-(2)}, &\overline{R}^{-\ast}(a)&=\langle R^{(1)},a\rangle R^{(2)}.
\end{align}
Hence, when acting by a tensor leg of a universal R-matrix in the relations of the braided Drinfeld double, we replace
\begin{align}
R^{(1)}\otimes (R^{(2)}\triangleright v)&=R(v^{(-1)})\otimes v^{(0)},&R^{-(1)}\otimes (R^{-(2)}\triangleright v)&=R^{-\ast}(v^{(-1)})\otimes v^{(0)},\\
R^{-(2)}\otimes (R^{-(1)}\triangleright v)&=\overline{R}(v^{(-1)})\otimes v^{(0)},&R^{(2)}\otimes (R^{(1)}\triangleright v)&=\overline{R}^{-\ast}(v^{(-1)})\otimes v^{(0)},
\end{align}
to obtain the formulas.
\end{proof}

\begin{remark}
Let us motivate the notion of a weak quasitriangular pair further. The idea is that given a weak quasitriangular pair $A,H$, the Hopf algebra $A$ obtains a dual quasitriangular structure $R\colon A\otimes A\to \Bbbk$. Indeed, for all $a,b\in A$,
\begin{align}
R(a\otimes b)=\langle R(a),b\rangle, &&R^{-\ast}(a\otimes b)=\langle \overline{R}(b),a \rangle. 
\end{align}
Lemma \ref{comodmodduality} produces a monoidal functor $\lcomod{A}\to \lmod{\leftexp{\cop}{H}}$. Using the weak quasitriangular structure, we can define braidings not just on $A$-comodules, but one of the objects can be an $H$-module. Indeed, let $V$ be a left $A$-comodule with coaction $\delta_V$, and $W$ a left $H^{\cop}$-module with action $a_W$, then the morphisms
$\Psi_{W,V}\colon W\otimes V\to V\otimes W$ given by 
\begin{equation}
\Psi_{W,V}=(\ide_V\otimes a_W\tau_{W,A})(\tau_{W,V}\otimes R)(\ide_W\otimes \tau_{A,V}\delta_V)
\end{equation}
and
\begin{equation}
\Psi^{-1}_{V,W}=(\ide_V\otimes a_W\tau_{W,A})(\tau_{W,V}\otimes \overline{R})(\ide_W\otimes \tau_{A,V}\delta_V)
\end{equation}
are morphisms of left $A$-modules (where $V$ is an $A$-module via the induced action from Lemma \ref{comodmodduality}). Here, $\tau_{V,W}$ denotes the swap-map $v\otimes w\mapsto w\otimes v$ making $\Vect$ symmetric monoidal. Further, the maps $\tau_{W,V}$ (and $\tau^{-1}_{W,V}$) are mutually inverse braidings provided that the $A$-action on $W$ is induced from a coaction as in Lemma \ref{comodmodduality}.
\end{remark}

Analogues of Corollary \ref{comodulealgebracor} and \ref{comodulealgebracor2} hold in this setup. We first give presentations for the braided cross products that appear in the statements.

\begin{lemma}
Assume given a weak quasitriangular pair of Hopf algebras $A$, $H$ as above, and $B,C$ weakly dual bialgebras in $\lcomod{A}$. Further let $D$ be a $B$-comodule algebra in $\cB=\lmod{H}$.  Then the relation Equation (\ref{crossproduct2}) for $D\rtimes \leftexp{\cop}{C}$ becomes
\begin{align}
\begin{split}cd&=d^{\overline{(0)}(0)}(\overline{R}(d^{\overline{(0)}(-1)})\triangleright c_{(1)})\ev(c_{(2)},d^{\overline{(-1)}})\\&=(R^{-\ast}({c_{(1)}}^{(-1)})\triangleright d^{\overline{(0)}}){c_{(1)}}^{(0)}\ev(c_{(2)},d^{\overline{(-1)}}),\end{split}
\end{align}
for all $d\in D, c\in C$.
Here, $\delta_B(d)=d^{\overline{(-1)}}\otimes d^{\overline{(0)}}$ denotes the left $B$-coaction on $D$, while $\delta_A(d)=d^{(-1)}\otimes d^{(0)}$ denotes the left $A$-coaction on $D$.

Further, for $B,C$ weakly dual Hopf algebras and $D\,'$ a right $C$-comodule algebra in $\cB$, the defining relation Equation (\ref{crossrelation2}) for $D\,'\rtimes B$ becomes
\begin{align}
ba&=d^{(0)\overline{(0)}}(R(d^{(-1)})\triangleright b)_{(2)}\ev(d^{(0)\overline{(-1)}},S(R(d^{(-1)})\triangleright b)_{(1)})), && \forall d\in D\,', b\in B.
\end{align}
Here, $\delta_C(d)=d^{\overline{(0)}}\otimes d^{\overline{(-1)}}\in D\,'\otimes C$ denotes the right $C$-coaction on $D\,'$, and $\delta_A(d)=d^{(-1)}\otimes d^{(0)}$ denotes the left $A$-coaction on $D\,'$.
\end{lemma}

\begin{corollary}\label{comodulealgebracor3}
Let $D$ be a $B$-comodule algebra in $\cB=\lmod{H}$. Then $D\rtimes \leftexp{\cop}{C}\rtimes H$ is a $\Drin_H(C,B)$-comodule algebra (with $D\rtimes \leftexp{\cop}{C}$ as a comodule subalgebra). The coaction $\delta_{\Drin}$ is given by
\begin{align}
\delta_{\Drin}(d)&=d^{\overline{(-1)}}{\overline{R}}^{-\ast}(d^{\overline{(0)}(-1)})\otimes d^{\overline{(0)}(0)}, \\
\delta_{\Drin}(c)&=R^{-\ast}({c_{(1)}}^{(-1)})c_{(2)}\otimes {c_{(1)}}^{(0)}, \\
 \delta_{\Drin}(h)&=h_{(2)}\otimes h_{(1)},
\end{align}
for $d\in D$, $c\in C$, $h\in H$, $\delta_B(d)=d^{\overline{(-1)}}\otimes d^{\overline{(0)}}$ denoting the left $B$-coaction on $D$, and $\delta_A(d)=d^{(-1)}\otimes d^{(0)}$ denoting the left $A$-coaction on $D$.
\end{corollary}

\begin{corollary}\label{comodulealgebracor4}
Assume given a weak quasitriangular pair $A,H$ and weakly dual bialgebras $B,C$ in $\lcomod{A}$. Let $D\,'$ be a right $C$-comodule algebra in $\cB=\lmod{H}$. Then $D\,'\rtimes B\rtimes H$ is a $\Drin_H(C,B)$-comodule algebra (with $D\,'\rtimes B$ as a comodule subalgebra), where the left $B$-module structure is obtained from the right $C$-comodule structure via Equation (\ref{inducedBaction}). The coaction $\delta_{\Drin}$ is given by
\begin{align}
\delta_{\Drin}(d)&=R^{-\ast}(d^{\overline{(0)}(-1)})d^{\overline{(-1)}}\otimes d^{\overline{(0)}(0)}, \\
\delta_{\Drin}(c)&=b_{(1)}{\overline{R}}^{-\ast}({b_{(2)}}^{(-1)})\otimes {b_{(2)}}^{(0)}, \\
 \delta_{\Drin}(h)&=h_{(2)}\otimes h_{(1)},
\end{align}
for $d\in D\,'$, $c\in C$, $h\in H$, where $\delta_C(d)=d^{\overline{(0)}}\otimes d^{\overline{(-1)}}$ denotes the right $C$-coaction on $D\,'$, and $\delta_A(d)=d^{(-1)}\otimes d^{(0)}$ the left $A$-coaction on $D\,'$.
\end{corollary}

\subsection{Quantum Group Examples}\label{qgroupex}

We now explore the application of the results from Corollary \ref{ydtensoractions} to Lustig's version of the quantum groups $U_q(\fr{g})$. Denote $\mF:=\Bbbk(q)$ for a generic variable $q$ over $\Bbbk$.
Further fix a \emph{Cartan datum}, i.e. an index set $I$ together with a symmetric bilinear form $\cdot$ on the free abelian group $\mZ \langle I\rangle $, such that $i\cdot i$ is even, and 
\begin{align*}a_{ij}:=2\frac{i\cdot j}{i\cdot i}\in \mZ_{\leq 0},&&\forall i\neq j.
\end{align*}
Following \cite{Maj1}*{Section 4}, a weak quasitriangular structure can be defined as follows. We let $A=\mF \mZ \langle I\rangle $ be the group algebra associated to $\mZ \langle I\rangle $ with generators denoted by $g_i$ for $i\in I$. The dually paired Hopf algebra $H$ is also a copy of the same group algebra, with $\mF$-basis denoted by $K_\mu$ for $\mu\in \mZ \langle I\rangle $. A Hopf algebra pairing of $H$ and $A$ is obtained via
\begin{align}
\langle K_i, g_j\rangle=q^{a_{ij}}.
\end{align}
A weak quasitriangular structure is now given by the maps $R,\overline{R}$ from $A$ to $H$ given by
\begin{align*}
R(g_i)&=K_i^{i\cdot i/2}, & \overline{R}(g_i)&=K_i^{-i\cdot i/2}.
\end{align*}

Note that the category of left $A$-comodules is braided via the dual R-matrix $R(g_i,g_j)=q^{i\cdot j}$. We now define $E:=\mF\langle e_i\mid i\in I\rangle$ to be the left $A$-comodule where $\delta(e_i)=g_i\otimes e_i$, and $F:=\mF \langle f_i\mid i\in I\rangle$ the left $A$-comodule left dual to $E$, i.e. $\delta(f_i)=g_{i}^{-1}\otimes f_i$, and duality pairing given by
$$\langle f_i,e_i\rangle=\frac{\delta_{i,j}}{q_i-q_i^{-1}},$$
where $q_i:=q^{i\cdot i/2}$.

Using the weak quasitriangular structure, we can induce left $A$-module structures on $E, F$ via the dual R-matrix
$R(g_i,g_j)= q^{i\cdot j}$. This way, $E$ and $F$ become dually paired Yetter--Drinfeld modules over $A$. Hence, using the general theory of Nichols algebras (see e.g. \cite{AS}*{Section~2}), we obtain dually paired braided Hopf algebras $U_q(\mathfrak{n}^+):=\cB(E)$ and $U_q(\mathfrak{n}^-):=\cB(F)$ which are primitively generated. That is
\begin{align}
\Delta (e_i)&=e_i\otimes 1+1\otimes e_i, &\Delta (f_i)&=f_i\otimes 1+1\otimes f_i.
\end{align}
The pairing $\langle~,~\rangle$ of $F$ and $E$ extends uniquely to a perfect pairing $\ev\colon U_q(\mathfrak{n}^-)\otimes U_q(\mathfrak{n}^+)\to \mF$. This result is due to Lustig, see \cite{Lus}*{Proposition 1.2.3}.

\begin{theorem}[Majid]\label{quantumgroupthm}
There exists an isomorphism of Hopf algebras between the braided Drinfeld double $\Drin_H(U_q(\mathfrak{n}^-),U_q(\mathfrak{n}^+))$ and $U_q(\fr{g})$, where $\fr{g}$ denotes the semi-simple Lie algebra obtained from the given Cartan datum. 
\end{theorem}
\begin{proof}
We reprove the result here as our presentation differs slightly from \cite{Maj1}.  The relations in Definition \ref{braideddrin3} show that $\Drin_H(U_q(\mathfrak{n}^-),U_q(\mathfrak{n}^+))$ is generated by $e_i$, $f_i$ and $K_i^{\pm 1}$ subject to the relations
\begin{align*}
f_je_i-q^{-i\cdot j}e_if_j&=\frac{\delta_{i,j}}{q_i-q_i^{-1}}(1-K_i^{i\cdot i}),&
K_ie_j&=q^{i\cdot j}e_jK_i,&
K_if_j&=q^{-i\cdot j}f_jK_i.
\end{align*}
The coproduct is given by
\begin{align*}
\Delta(K_i)&=K_i\otimes K_i, &\Delta(e_i)&=e_i\otimes 1+K_i^{i\cdot i/2}\otimes e_i, &\Delta(f_i)&=f_i\otimes 1+K_i^{i\cdot i/2}\otimes f_i.&
\end{align*}
We can define an isomorphism of bialgebras $$\varphi\colon \Drin_H(U_q(\mathfrak{n}^-),U_q(\mathfrak{n}^+))\longrightarrow U_q(\fr{g}), $$
where the latter is Lustig's version of the quantum group from \cite{Lus}*{Chapter 3}, (identifying the parameter $\nu$ with $q$) by
\begin{align*}
\varphi(K_i)&=K_i, &\varphi(e_i)&=E_i, & \varphi(K_j^{-j\cdot j/2}f_j)=F_j.
\end{align*}
To verify this, note in particular that
\begin{align*}
e_i(K_j^{-j\cdot j/2}f_j)&=q^{-i\cdot j}K_j^{-j\cdot j/2}e_if_j
=(K_j^{-j\cdot j/2}f_j)e_i+ \frac{\delta_{i,j}}{q_i-q_i^{-1}}(K_i^{i\cdot i/2}-K_i^{-i\cdot i/2}).
\end{align*}
Further, $\Delta(K_j^{-j\cdot j/2}f_j)=K_j^{-j\cdot j/2}f_j\otimes K_j^{-j\cdot j/2}+1\otimes K_j^{-j\cdot j/2}f_j$, which is the same as the coproduct formula for $F_j$ in $U_q(\fr{g})$, cf. \cite{Lus}*{Proposition 3.1.4}. It is clear that $\varphi$ is invertible by construction. Since the antipode is uniquely determined, we obtain an isomorphism of Hopf algebras.
\end{proof}
In the following, we shall denote $\Drin_H(U_q(\mathfrak{n}^-),U_q(\mathfrak{n}^+))$ by $U_q(\fr{g})$ and use the presentation with $q$-commutators whenever using lower case letter $e_i, f_j$ to denote the primitive generators.

\begin{example}
Consider the \emph{$q$-Weyl algebra}, which is defined as the braided Heisenberg double $D_q(\fr{g}):=\Heis_H(U_q(\mathfrak{n}^-),U_q(\mathfrak{n}^+))$. It is the algebra generated by $e_i,f_i, K_i$ for $i\in I$ subject to the relations
\begin{align*}
f_ie_j-q^{i\cdot j}e_jf_i&=\frac{\delta_{i,j}}{q_i-q_i^{-1}},&
K_ie_j&=q^{i\cdot j}e_jK_i,&
K_if_j&=q^{-i\cdot j}f_jK_i.
\end{align*}
Similarly to $\Drin_H(U_q(\mathfrak{n}^-),U_q(\mathfrak{n}^+))$, it also has a triangular decomposition as the $\mF$-vector space $U_q(\mathfrak{n}^-)\otimes H\otimes U_q(\mathfrak{n}^+)$. This algebra contains as a subalgebra the quantum Weyl algebra $A_C$ of \cite{Jos}*{Section 3.1}.
\end{example}

\begin{corollary}\label{quantumweyl}
The quantum Weyl algebra $D_q(\mathfrak{g})$ is a left comodule algebra over the quantum group $U_q(\fr{g})$.
\end{corollary}
\begin{proof}
This result was proved in \cite{Lau}*{Section~3} using a slightly different presentation. Corollary \ref{comodulealgebracor3} implies a coaction given the coproduct of $U_q(\fr{g})$ (under the isomorphism of Theorem \ref{quantumgroupthm}, making the latter a $U_q(\fr{g})$-comodule algebra.
\end{proof}

\begin{example}
Using the trivial $U_q(\fr{n}^+)$-coaction on itself, we obtain the algebra $U_q(\fr{n}^+)\otimes_q U_q(\fr{n}^-)$ with defining relation
\begin{align*}
f_ie_j&=q^{i\cdot j}e_jf_i.
\end{align*}
Now $U_q(\fr{n}^+)\otimes_q U_q(\fr{n}^-)$ is a $U_q(\mathfrak{g})$-comodule algebra with coaction given by
\begin{align*}
\delta(e_i)&=K_i^{i\cdot i}\otimes e_i,&
\delta(f_i)&=f_i\otimes 1+K_i^{i\cdot i/2}\otimes f_i.
\end{align*}
\end{example}

\section{A Morphism of 2-Cocycle Spaces}\label{2cocycles}

In this section, we provide an explicit map inducing bialgebra $2$-cocycles on the braided Drinfeld double $\Drin_H(C,B)$ from $2$-cocycles over $B$ and $C$ within the category $\cB=\lmod{H}$. 

\subsection{Preliminaries on 2-Cocycle Twists and Cleft Objects}

Most of the material of this subsection can be found in \cite{Mas} and reference therein. We let $B$ be a bialgebra over $\Bbbk$.

\begin{definition}
A \emph{right $2$-cocycle} of $B$ is a convolution-invertible $\Bbbk$-linear map $\sigma \colon B\otimes B\to \Bbbk$, such that for $h,k,l\in B$:
\begin{align}\label{cocyclecond1}
\sigma(h\otimes k_{(1)}l_{(1)})\sigma(k_{(2)}\otimes l_{(2)})&=\sigma(h_{(1)}k_{(1)}\otimes l) \sigma(h_{(2)}\otimes k_{(2)}),\\
\sigma(h\otimes 1)&=\sigma(1\otimes h)=\varepsilon(h).\label{cocyclecond2}
\end{align}
\end{definition}
See also \cite{Maj4}*{Section~6} for a general theory of cohomology of a bialgebra where the space $C^2(B,\Bbbk)$ of $2$-cocycles appears naturally. 

Given a convolution invertible map $\beta\colon B\to \Bbbk$ and $\sigma\in C^2(B,\Bbbk)$, a new right $2$-cocycle can be given by 
\begin{equation}
\sigma^{\beta}(h\otimes k)=\beta^{-\ast}(h_{(1)}k_{(1)})\sigma(h_{(2)},k_{(2)})\beta(h_{(3)})\beta(k_{(3)}).
\end{equation}
We say that $\sigma$ and $\tau$ in $C^2(B,\Bbbk)$ are cohomologous if there exists such $\beta$ satisfying $\tau=\sigma^\beta$, and write $\sigma\sim \tau$ (cf. \cite{KS}*{Section 1.2.3} where left $2$-cocycles are used). This defines an equivalence relation on $C^2(B,\Bbbk)$. The quotient is denoted by $H^2(B,\Bbbk)$ and referred to as the \emph{second bialgebra cohomology space}.
Right $2$-cocycles can be used to twist a bialgebra:

\begin{proposition}
Let $\sigma$ be a right $2$-cocycle of $B$. Then $B$ obtains a new algebra structure $B_{\sigma}$ with product given for $g,h\in B$:
\begin{align}
g\cdot_\sigma h:=g_{(1)}\cdot h_{(1)}\sigma(g_{(2)},h_{(2)}).
\end{align}
Moreover, the coproduct can be viewed as a morphism of algebras 
\[
\Delta\colon B_{\sigma}\longrightarrow B\otimes B_{\sigma},
\]
making $B_\sigma$ a $B$-comodule algebra.
\end{proposition}
If $\sigma\sim \tau$, then $B_\sigma$ and $B_\tau$ are isomorphic as $B$-comodule algebras via the map 
\begin{align*}\ide\ast\beta\colon  B_\sigma\longrightarrow B_\tau, && h\longmapsto h_{(1)}\beta(h_{(2)}).\end{align*} 
The notion of a $2$-cocycle twist of a bialgebra $B$ can also be reformulated in terms of \emph{cleft objects}.

\begin{definition}
Let $B$ be a bialgebra over $\Bbbk$. A \emph{left $B$-cleft object} is a left $B$-comodule  algebra $C$  together with a unit-preserving\footnote{In \cite{Mas} this is not a requirement, but $\phi$ can be chosen to preserve the unit.} isomorphism of left $B$-comodules $\phi\colon B\to C$ which is also invertible with respect to the convolution product. That is, there exists $\psi\colon B\to C$ such that
\begin{align}
m_C(\psi\otimes \phi)\Delta_B&=m_C(\phi\otimes \psi)\Delta_B=1_C\varepsilon_B.
\end{align}
\end{definition}

It is shown in \cite{DT} that given a left $B$-cleft object $(C,\phi)$, a right $2$-cocycle $\sigma$ can be obtained via
\begin{align}
\sigma(g\otimes h)&:=\phi(g_{(1)})\phi(h_{(1)})\psi(g_{(2)}h_{(2)})\in \Bbbk.
\end{align}
In particular, $\sigma$ takes values in the $B$-coinvariants of $C$, which are given by the one-dimensional space $\Bbbk$ as $C\cong B^{\op{reg}}$ as a $B$-comodules. Then $B_\sigma \cong C$ via $\phi$. In fact, every $2$-cocycle arises in this way \cite{Mas2}*{Proposition~1.4}. See also \cite{Sch4} for details on cleft and Galois objects and their relation to $2$-cocycles.

\begin{example}
Let $H$ be a dual quasitriangular Hopf algebra with dual R-matrix $R$. Then $R$ is a left $2$-cocycle and the convolution inverse $R^{-\ast}$ is a right $2$-cocycle. This can be proved in two steps \cite{Lu}. First, $R$ satisfies the dual quantum Yang-Baxter equation
\begin{equation}
R(g_{(1)},h_{(1)})R(g_{(2)},k_{(1)})R(h_{(2)},k_{(2)})= R(h_{(1)},k_{(1)})R(g_{(1)},k_{(2)})R(g_{(2)},h_{(2)}),
\end{equation}
using Equations (\ref{dualR1})--(\ref{dualR3}). From which we conclude, using the same equations, that
\begin{equation}
R(g_{(1)},h_{(1)})R(g_{(2)}h_{(2)},k)= R(h_{(1)},k_{(1)})R(g,h_{(2)}k_{(2)}),
\end{equation}
the \emph{left} $2$-cocycle condition. Similarly, $R^{-\ast}$ (and hence also $R^{\oop}$) are right $2$-cocycles.
\end{example}

\subsection{Induced 2-Cocycles over the Drinfeld Double}\label{drin2co}

From Theorem \ref{mainthm}, we first derive two morphisms of right $2$-cocycles for weakly dual bialgebras to right $2$-cocycles over their Drinfeld double. We then define, in the Hopf algebra case, more general $2$-cocycles over $\Drin_\Bbbk(C,B)$ from a pair of $2$-cocycles, one over $B$ and one over $C$.

\begin{corollary}\label{induced2cocycles} Let $C,B$ be weakly dual bialgebras (respectively, Hopf algebras) over $\Bbbk$. Then there exist two maps
\begin{align*}
\op{Ind}_B\colon &H^2(B,\Bbbk)\longrightarrow H^2(\Drin_{\Bbbk}(C,B),\Bbbk), & \sigma \longmapsto \op{Ind}_B\sigma,\\
\op{Ind}_C\colon &H^2(\leftexp{\cop}{C},\Bbbk)\longrightarrow H^2(\Drin_{\Bbbk}(C,B),\Bbbk), & \tau \longmapsto \op{Ind}_C\tau,
\end{align*}
where the $2$-cocycle $\op{Ind}_B\sigma$ is defined by requiring that
\begin{align}
\op{Ind}_B\sigma(b,b')&= \sigma(b,b'), & \op{Ind}_\Delta\sigma(c,c')&= \varepsilon(c)\varepsilon(c'), 
\end{align}
for any elements $b,b'\in B$, $c,c'\in C$. The $2$-cocycle $\op{Ind}_B\tau$ is defined by requiring that 
\begin{align}
\op{Ind}_C\tau(b,b')&=\varepsilon(b)\varepsilon(b'), & \op{Ind}_C\tau(c,c')&=\tau(c,c').
\end{align}

If $C,B$ are Hopf algebras, then $\Drin_{\Bbbk}(C,B)$ is defined on the $\Bbbk$-vector space $B\otimes C$, on which the $2$-cocycle is given by
\begin{align}\label{inducedcocycleB}
\op{Ind}_B\sigma(bc,b'c')&=\sigma(b,b'_{(2)})\ev(c,b'_{(1)})\varepsilon(c').
\end{align}
Similarly,
\begin{align}\label{inducedcocycleC}
\op{Ind}_C\tau(cb,c'b')&=\tau(c,c'_{(1)})\ev(c'_{(2)},S(b))\varepsilon(b').
\end{align}
\end{corollary}
\begin{proof}
This follows from Theorem \ref{mainthm}, via Corollary \ref{comodulealgebracor}, using \cite{Lau}*{Proposition 3.8.4}. We will provide detailed proofs of more general statements in Theorem \ref{2cocyclethm}, and Section \ref{cleftbraided}.
\end{proof}

Note that, in the Hopf algebra case, Equation (\ref{inducedcocycleC}) is equivalent to
\begin{align}\label{inducedcocycleC2}
\op{Ind}_C\tau(bc,b'c')&=\varepsilon(b)\tau(c_{(1)},c')\ev(c_{(2)},b').
\end{align}
We include a presentation of these induced $2$-cocycles using graphical calculus.
The cocycles of Equation (\ref{inducedcocycleB}) and Equation (\ref{inducedcocycleC2}) are given by
\begin{align*}
\op{Ind}_B\sigma&=\vcenter{\hbox{
\begingroup%
  \makeatletter%
  \providecommand\color[2][]{%
    \errmessage{(Inkscape) Color is used for the text in Inkscape, but the package 'color.sty' is not loaded}%
    \renewcommand\color[2][]{}%
  }%
  \providecommand\transparent[1]{%
    \errmessage{(Inkscape) Transparency is used (non-zero) for the text in Inkscape, but the package 'transparent.sty' is not loaded}%
    \renewcommand\transparent[1]{}%
  }%
  \providecommand\rotatebox[2]{#2}%
  \ifx\svgwidth\undefined%
    \setlength{\unitlength}{64.37651367bp}%
    \ifx\svgscale\undefined%
      \relax%
    \else%
      \setlength{\unitlength}{\unitlength * \real{\svgscale}}%
    \fi%
  \else%
    \setlength{\unitlength}{\svgwidth}%
  \fi%
  \global\let\svgwidth\undefined%
  \global\let\svgscale\undefined%
  \makeatother%
  \begin{picture}(1,0.67200066)%
    \put(0,0){\includegraphics[width=\unitlength]{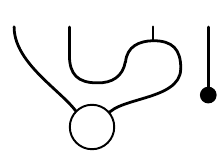}}%
    \put(-0.0026577,0.61057499){\color[rgb]{0,0,0}\makebox(0,0)[lb]{\smash{$B$}}}%
    \put(0.24735955,0.61371319){\color[rgb]{0,0,0}\makebox(0,0)[lb]{\smash{$C$}}}%
    \put(0.62076552,0.61334875){\color[rgb]{0,0,0}\makebox(0,0)[lb]{\smash{$B$}}}%
    \put(0.34961769,0.0492047){\color[rgb]{0,0,0}\makebox(0,0)[lb]{\smash{$\sigma$}}}%
    \put(0.86664149,0.61371319){\color[rgb]{0,0,0}\makebox(0,0)[lb]{\smash{$C$}}}%
  \end{picture}%
\endgroup%
}},& \op{Ind}_C\tau&=\vcenter{\hbox{
\begingroup%
  \makeatletter%
  \providecommand\color[2][]{%
    \errmessage{(Inkscape) Color is used for the text in Inkscape, but the package 'color.sty' is not loaded}%
    \renewcommand\color[2][]{}%
  }%
  \providecommand\transparent[1]{%
    \errmessage{(Inkscape) Transparency is used (non-zero) for the text in Inkscape, but the package 'transparent.sty' is not loaded}%
    \renewcommand\transparent[1]{}%
  }%
  \providecommand\rotatebox[2]{#2}%
  \ifx\svgwidth\undefined%
    \setlength{\unitlength}{64.37651367bp}%
    \ifx\svgscale\undefined%
      \relax%
    \else%
      \setlength{\unitlength}{\unitlength * \real{\svgscale}}%
    \fi%
  \else%
    \setlength{\unitlength}{\svgwidth}%
  \fi%
  \global\let\svgwidth\undefined%
  \global\let\svgscale\undefined%
  \makeatother%
  \begin{picture}(1,0.67200066)%
    \put(0,0){\includegraphics[width=\unitlength]{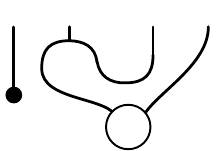}}%
    \put(-0.0026577,0.61057499){\color[rgb]{0,0,0}\makebox(0,0)[lb]{\smash{$B$}}}%
    \put(0.24735955,0.61371319){\color[rgb]{0,0,0}\makebox(0,0)[lb]{\smash{$C$}}}%
    \put(0.62076552,0.61334875){\color[rgb]{0,0,0}\makebox(0,0)[lb]{\smash{$B$}}}%
    \put(0.51166983,0.0492047){\color[rgb]{0,0,0}\makebox(0,0)[lb]{\smash{$\tau$}}}%
    \put(0.86664149,0.61371319){\color[rgb]{0,0,0}\makebox(0,0)[lb]{\smash{$C$}}}%
  \end{picture}%
\endgroup%
}}.
\end{align*}

Another way to present these results of this section is the following statement:

\begin{corollary}\label{cleftthm}$~$\begin{enumerate}
\item[(i)]
Let $A$ be a left $B$-cleft object. Then $A\rtimes \leftexp{\cop}{C}$ is a left $\Drin_{\Bbbk}(C,B)$-cleft object.
\item[(ii)]
Let $A'$ be a right $C$-cleft object. Then $A'\rtimes B$ is a left $\Drin_{\Bbbk}(C,B)$-cleft object.
\end{enumerate}
\end{corollary}

$2$-Cocycles of the Drinfeld double have been studied in the literature. A special case of the general construction in \cite{Sch3} describes the kernel of the map 
\[
H^2(Q\bowtie K^*)\longrightarrow H^2(Q)\times H^2( K^*)
\]
where $K$, $Q$ are Hopf algebras, $K$ finite-dimensional \cite{Sch3}*{Theorem 6.5.6}. Here, $Q\bowtie K^*$ is a bicross product, which as a special case recovers the Drinfeld double \cite{Maj9}. Moreover, \cite{Sch2}*{Theorem 2.1} classifies all Galois objects  over a twisted bicross product $H\bowtie_\tau L$ in terms of Galois objects of $H$ and $L$ together with a Hopf pairing of $H$ and $L$ (called a skew-pairing therein). Over a field $\Bbbk$, this result gives a characterization of a cleft objects over bicross products (including the Drinfeld double).

Later work in \cite{BC}*{Corollary 4.11} provides a description of the group of all lazy $2$-cocycles on $\Drin_{\Bbbk}(A^*,A)$ in terms of data for $A$ and $A^*$ using the approach that the Drinfeld double is a twist of the tensor product Hopf algebra. We further note that \cite{CP}*{Eq. (3.1)} gives a different way to induce \emph{lazy} $2$-cocycles on the Drinfeld double. 

Motivated by the above form of induced $2$-cocycles, a more general statement arises:

\begin{theorem}\label{2cocyclethm}
For $C,B$ weakly dual Hopf algebras over $\Bbbk$, let $\sigma\in C^2(B,\Bbbk)$ and $\tau\in C^2(\leftexp{\cop}{C},\Bbbk)$ (or, equivalently, $\tau$ is a \emph{left} $2$-cocycle over $C$). Then $\sigma \circ \tau\in C^2(\Drin_\Bbbk(C,B), \Bbbk)$, where
\begin{align}
(\sigma \circ \tau)(bc,b'c')&=\sigma(b,b'_{(2)})\tau(c_{(1)},c')\ev(c_{(2)},b'_{(1)}).
\end{align} 
Right twist by $\sigma \circ \tau$ gives an algebra defined on the $\Bbbk$-vector space $B\otimes C$ with the product
\begin{align}\label{doubletwistproduct}
m_{\sigma \circ \tau}(bc,b'c')&=m_\sigma(b,b'_{(2)})\otimes m_\tau(c_{(1)},c')\ev(c_{(2)},b'_{(1)}).
\end{align}
\end{theorem}
\begin{proof}
Condition (\ref{cocyclecond2}) easily follows from the corresponding condition for $\sigma$ and $\tau$. We verify the $2$-cocycle condition (\ref{cocyclecond1}) explicitly.
\begin{align*}
&(\sigma \circ \tau)(bc,(b'c')_{(1)}(b''c'')_{(1)})(\sigma\circ \tau)((b'c')_{(2)},(b''c'')_{(2)})\\&=
\sigma(b,b'_{(2)}b''_{(3)})\tau(c_{(1)},c'_{(4)}c''_{(2)})\sigma(b'_{(3)},b''_{(6)})\tau(c'_{(1)}, c''_{(1)})\\&~\phantom{=}~\ev(c_{(2)},b'_{(1)}b''_{(2)})\ev(c'_{(3)},S(b''_{(4)}))\ev(c'_{(5)},b''_{(1)})\ev(c'_{(2)},b''_{(5)})\\
&=
\sigma(b,b'_{(2)}b''_{(3)})\tau(c_{(1)},c'_{(2)}c''_{(2)})\sigma(b'_{(3)},b''_{(4)})\tau(c'_{(1)}, c''_{(1)})\ev(c_{(2)},b'_{(1)}b''_{(2)})\ev(c'_{(3)},b''_{(1)})\\
&=
\sigma(b_{(1)}b'_{(2)},b''_{(3)})\tau(c_{(1)},c'_{(2)}c''_{(2)})\sigma(b_{(2)},b'_{(3)})\tau(c'_{(1)}, c''_{(1)})\ev(c_{(2)},b'_{(1)}b''_{(2)})\ev(c'_{(3)},b''_{(1)})
\\
&=
\sigma(b_{(1)}b'_{(2)},b''_{(3)})\tau(c_{(1)},c'_{(1)})\sigma(b_{(2)},b'_{(3)})\tau(c_{(2)}c'_{(2)}, c'')\ev(c_{(3)},b'_{(1)}b''_{(2)})\ev(c'_{(3)},b''_{(1)})\\
&=\sigma(b_{(1)}b'_{(2)},b''_{(3)})\tau(c_{(1)},c'_{(1)})\sigma(b_{(2)},b'_{(3)})\tau(c_{(2)}c'_{(2)}, c'')\ev(c_{(4)},b'_{(1)})\ev(c_{(3)},b''_{(2)})\ev(c'_{(3)},b''_{(1)})\\
&=\sigma(b_{(1)}b'_{(2)},b''_{(2)})\tau(c_{(1)},c'_{(1)})\sigma(b_{(2)},b'_{(3)})\tau(c_{(2)}c'_{(2)}, c'')\ev(c_{(4)},b'_{(1)})\ev(c'_{(3)}c_{(3)},b''_{(1)})\\
&=\sigma(b_{(1)}b'_{(2)},b''_{(2)})\tau(c_{(1)},c'_{(1)})\sigma(b_{(2)},b'_{(3)})\tau((c_{(2)}c'_{(2)})_{(1)}, c'')\ev(c_{(3)},b'_{(1)})\ev((c_{(2)}c'_{(2)})_{(2)},b''_{(1)})\\
&=\sigma(b_{(1)}b'_{(2)},b''_{(2)})\tau(c_{(1)},c'_{(1)})\sigma(b_{(2)},b'_{(5)})\tau((c_{(4)}c'_{(2)})_{(1)}, c'')\\&~\phantom{=}~\ev(c_{(3)},b'_{(1)})\ev((c_{(4)}c'_{(2)})_{(2)},b''_{(1)})\ev(c'_{(3)},b''_{(3)})\ev(c'_{(2)},S(b''_{(4)}))\\
&=(\sigma \circ \tau)((bc)_{(1)}(b'c')_{(1)},b''c'')(\sigma\circ \tau)((bc)_{(2)},(b'c')_{(2)}).
\end{align*}
Here, we first use the definition of $\sigma\circ \tau$ and the bialgebra condition applied to $b'_{(1)}\otimes b''_{(2)}$, followed by application of the identity $\ev(c'_{(3)},b''_{(4)})\ev(c'_{(2)},S(b''_{(5)}))=\varepsilon(c'_{(2)})\otimes \varepsilon(b'_{(4)})$ together with the counit axioms. In the third equality, it is used that $\sigma$ is a $2$-cocycle over $B$, followed by $\tau$ being a $2$-cocycle over $\leftexp{\cop}{C}$ in the fourth equality. Next, we use $\ev(c_{(3)},b'_{(1)}b''_{(2)})=\ev(c_{(4)},b'_{(1)})\ev(c_{(3)},b''_{(2)})$, followed by application of $\ev(c'_{(4)},b''_{(1)})\ev(c_{(3)},b''_{(2)})=\ev(c'_{(3)}c_{(4)},b''_{(1)})$ in the fifth equality, and the bialgebra condition applied to $c_{(2)}\otimes c'_{(2)}$ in the sixth equality. Finally, the equality $\ev(c'_{(3)},b''_{(3)})\ev(c'_{(2)},S(b''_{(4)}))=\varepsilon(c'_{(2)})\otimes \varepsilon(b'_{(3)})$ together with the counit axioms is used again in the penultimate equality, followed by the definitions.

Finally, the simplified formula for the product in Equation (\ref{doubletwistproduct}) follows under use of the antipode axiom.
\end{proof}

The cocycle $\sigma\circ\tau$ can be depicted using graphical calculus as
\begin{align}
\sigma\circ \tau =\vcenter{\hbox{
\begingroup%
  \makeatletter%
  \providecommand\color[2][]{%
    \errmessage{(Inkscape) Color is used for the text in Inkscape, but the package 'color.sty' is not loaded}%
    \renewcommand\color[2][]{}%
  }%
  \providecommand\transparent[1]{%
    \errmessage{(Inkscape) Transparency is used (non-zero) for the text in Inkscape, but the package 'transparent.sty' is not loaded}%
    \renewcommand\transparent[1]{}%
  }%
  \providecommand\rotatebox[2]{#2}%
  \ifx\svgwidth\undefined%
    \setlength{\unitlength}{75.99584198bp}%
    \ifx\svgscale\undefined%
      \relax%
    \else%
      \setlength{\unitlength}{\unitlength * \real{\svgscale}}%
    \fi%
  \else%
    \setlength{\unitlength}{\svgwidth}%
  \fi%
  \global\let\svgwidth\undefined%
  \global\let\svgscale\undefined%
  \makeatother%
  \begin{picture}(1,0.69888911)%
    \put(-0.00491521,0.65140829){\color[rgb]{0,0,0}\makebox(0,0)[lb]{\smash{$B$}}}%
    \put(0.19363891,0.65390053){\color[rgb]{0,0,0}\makebox(0,0)[lb]{\smash{$C$}}}%
    \put(0.5888732,0.65361109){\color[rgb]{0,0,0}\makebox(0,0)[lb]{\smash{$B$}}}%
    \put(0.78413846,0.65390053){\color[rgb]{0,0,0}\makebox(0,0)[lb]{\smash{$C$}}}%
    \put(0,0){\includegraphics[width=\unitlength,page=1]{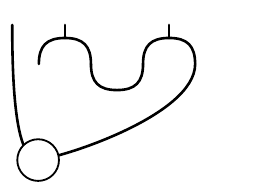}}%
    \put(0.09516855,0.05545472){\color[rgb]{0,0,0}\makebox(0,0)[lb]{\smash{$\sigma$}}}%
    \put(0,0){\includegraphics[width=\unitlength,page=2]{cocycle3.pdf}}%
    \put(0.68730621,0.04558576){\color[rgb]{0,0,0}\makebox(0,0)[lb]{\smash{$\tau$}}}%
  \end{picture}%
\endgroup%
}}.
\end{align}

Note that the more general statement of Theorem \ref{2cocyclethm} recovers $\op{Ind}_B\sigma$ as $\sigma\circ \triv_C$, where $\triv_C=\varepsilon_C\otimes \varepsilon_C$, and $\op{Ind}_C\tau$ as $\triv_B\circ \tau$. We remark that the $2$-cycles of the form $\sigma\circ \tau$ do in general not give \emph{all} $2$-cocycles over $\Drin_\Bbbk(C,B)$. For example, the trivial $2$-cocycle is not of this form.

\begin{example}
Let $G$ be a finite group. We can identify $H^2(\Bbbk G,\Bbbk)$ we the 2nd cohomology group $H^2(G, \Bbbk^\times)$ of $G$ (cf. \cite{Maj4}*{Section~6}). Moreover, $H^2(\Bbbk[G],\Bbbk)$ can be identified with the second homology $H_2(\Bbbk[G],\Bbbk)$ of $\Bbbk G$, the set of left $\Bbbk$-cycles up to boundaries. That is, elements $c=c^{(1)}\otimes c^{(2)}\in \Bbbk G\otimes\Bbbk G$ satisfying 
\begin{equation}\label{cycleq}
c_1^{(1)}{c_2^{(1)}}_{(1)}\otimes c_1^{(2)}{c_2^{(1)}}_{(2)}\otimes c_2^{(2)}=
c_1^{(1)}\otimes c_2^{(1)}{c_1^{(2)}}_{(1)}\otimes c_2^{(2)}{c_1^{(2)}}_{(2)}.
\end{equation}
Explicitly, $c=\sum_{g,h\in G}c_{g,h}g\otimes h$ and then Equation (\ref{cycleq}) becomes
\begin{align}
c_{g,h}c_{k,l}=c_{gk^{-1},hk^{-1}}c_{kh,lh}, &&\forall g,h,k,l\in G.
\end{align}
Theorem \ref{2cocyclethm} now gives a morphism
\begin{align*}
H^2(\Bbbk G, \Bbbk)\times H_2(\Bbbk G, \Bbbk)\longrightarrow H^2(\Drin_\Bbbk(G), \Bbbk),
\end{align*}
which maps $(\sigma,c)$ to the right $2$-cocycle $\sigma\circ c^*$ given by
\begin{align}
(\sigma\circ c^*)(g\delta_h\otimes k\delta_l)=\sigma(g,k)c_{l,k^{-1}h}.
\end{align}
\end{example}

\begin{example}\label{Luexample}
Let $H$ be a finite-dimensional Hopf algebra over $\Bbbk$. It was shown in \cite{Lu} that $\Heis_\Bbbk(B^*,B)$ is a $2$-cocycle twist of $\Drin_\Bbbk(B^*,B)$. The $2$-cocycle used is $\triv \circ \triv$, which coincides with the dual of the opposite universal R-matrix $R^{\oop}$.
\end{example}

\subsection{2-Cocycles in a Braided Monoidal Category}\label{braidedcycles}

We now need some facts about a generalization of $2$-cocycles to working in a braided monoidal category $\cB$. 

\begin{definition}\label{2cocycledef}
Let $B$ be a bialgebra in $\cB$. A \emph{right $2$-cocycle over $B$ in $\cB$} is a convolution-invertible map $\sigma \colon B\otimes B\to I$ in $\cB$, such that
\begin{align}\label{2cocycleinB}
&\sigma(\ide_B\otimes m\otimes \sigma)(\ide_B\otimes\Delta_{B\otimes B})
= \sigma(m\otimes \sigma\otimes \ide_B)(\Delta_{B\otimes B}\otimes\ide_B),
\end{align}
as well as
\begin{align}\label{2cocyclenorm}
\sigma(1\otimes \ide_B)&=\varepsilon, &\sigma(\ide_B\otimes 1)&=\varepsilon.
\end{align}
The set of all $2$-cocycles over $B$ in $\cB$ is denoted by $C^2_{\cB}(B,I)$. If $\cB=\lmod{H}$ for a quasitriangular Hopf algebra, we also use the notation $C^2_{H}(B,\Bbbk)$.
\end{definition}

The $2$-cocycle condition Equations (\ref{2cocycleinB}) and (\ref{2cocyclenorm}) can be visualized as
\begin{align*}
\vcenter{\hbox{
\begingroup%
  \makeatletter%
  \providecommand\color[2][]{%
    \errmessage{(Inkscape) Color is used for the text in Inkscape, but the package 'color.sty' is not loaded}%
    \renewcommand\color[2][]{}%
  }%
  \providecommand\transparent[1]{%
    \errmessage{(Inkscape) Transparency is used (non-zero) for the text in Inkscape, but the package 'transparent.sty' is not loaded}%
    \renewcommand\transparent[1]{}%
  }%
  \providecommand\rotatebox[2]{#2}%
  \ifx\svgwidth\undefined%
    \setlength{\unitlength}{121.42095279bp}%
    \ifx\svgscale\undefined%
      \relax%
    \else%
      \setlength{\unitlength}{\unitlength * \real{\svgscale}}%
    \fi%
  \else%
    \setlength{\unitlength}{\svgwidth}%
  \fi%
  \global\let\svgwidth\undefined%
  \global\let\svgscale\undefined%
  \makeatother%
  \begin{picture}(1,0.44485246)%
    \put(0,0){\includegraphics[width=\unitlength,page=1]{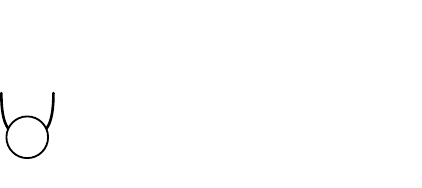}}%
    \put(-0.00926498,0.07691323){\color[rgb]{0,0,0}\makebox(0,0)[lb]{\smash{}}}%
    \put(0.04014988,0.09544381){\color[rgb]{0,0,0}\makebox(0,0)[lb]{\smash{$\sigma$}}}%
    \put(0,0){\includegraphics[width=\unitlength,page=2]{2cocycle.pdf}}%
    \put(0.28722435,0.15745914){\color[rgb]{0,0,0}\makebox(0,0)[lb]{\smash{$\sigma$}}}%
    \put(0,0){\includegraphics[width=\unitlength,page=3]{2cocycle.pdf}}%
    \put(0.76901929,0.15745914){\color[rgb]{0,0,0}\makebox(0,0)[lb]{\smash{$\sigma$}}}%
    \put(0,0){\includegraphics[width=\unitlength,page=4]{2cocycle.pdf}}%
    \put(0.77519614,0.02774512){\color[rgb]{0,0,0}\makebox(0,0)[lb]{\smash{$\sigma$}}}%
    \put(0,0){\includegraphics[width=\unitlength,page=5]{2cocycle.pdf}}%
    \put(0.41980616,0.28802953){\color[rgb]{0,0,0}\makebox(0,0)[lb]{\smash{$=$}}}%
    \put(0,0){\includegraphics[width=\unitlength,page=6]{2cocycle.pdf}}%
  \end{picture}%
\endgroup%
}}, &&\text{and }&&
 \vcenter{\hbox{
\begingroup%
  \makeatletter%
  \providecommand\color[2][]{%
    \errmessage{(Inkscape) Color is used for the text in Inkscape, but the package 'color.sty' is not loaded}%
    \renewcommand\color[2][]{}%
  }%
  \providecommand\transparent[1]{%
    \errmessage{(Inkscape) Transparency is used (non-zero) for the text in Inkscape, but the package 'transparent.sty' is not loaded}%
    \renewcommand\transparent[1]{}%
  }%
  \providecommand\rotatebox[2]{#2}%
  \ifx\svgwidth\undefined%
    \setlength{\unitlength}{100.1250375bp}%
    \ifx\svgscale\undefined%
      \relax%
    \else%
      \setlength{\unitlength}{\unitlength * \real{\svgscale}}%
    \fi%
  \else%
    \setlength{\unitlength}{\svgwidth}%
  \fi%
  \global\let\svgwidth\undefined%
  \global\let\svgscale\undefined%
  \makeatother%
  \begin{picture}(1,0.27096573)%
    \put(0,0){\includegraphics[width=\unitlength,page=1]{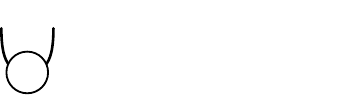}}%
    \put(-0.01123558,0.01117441){\color[rgb]{0,0,0}\makebox(0,0)[lb]{\smash{}}}%
    \put(0.04868949,0.03364631){\color[rgb]{0,0,0}\makebox(0,0)[lb]{\smash{$\sigma$}}}%
    \put(0,0){\includegraphics[width=\unitlength,page=2]{2cocyclenorm.pdf}}%
    \put(0.21910147,0.15215894){\color[rgb]{0,0,0}\makebox(0,0)[lb]{\smash{$=$}}}%
    \put(0,0){\includegraphics[width=\unitlength,page=3]{2cocyclenorm.pdf}}%
    \put(0.42322124,0.15349646){\color[rgb]{0,0,0}\makebox(0,0)[lb]{\smash{$,$}}}%
    \put(0,0){\includegraphics[width=\unitlength,page=4]{2cocyclenorm.pdf}}%
    \put(0.81835218,0.15215917){\color[rgb]{0,0,0}\makebox(0,0)[lb]{\smash{$=$}}}%
    \put(0,0){\includegraphics[width=\unitlength,page=5]{2cocyclenorm.pdf}}%
    \put(0.64794015,0.03899683){\color[rgb]{0,0,0}\makebox(0,0)[lb]{\smash{$\sigma$}}}%
  \end{picture}%
\endgroup%
}}~.
\end{align*}

\begin{remark}
The definition of bialgebra $2$-cocycles in a braided monoidal category appears in \cite{BD1}*{Definition~1.5}. The theory of $2$-cocyles over cocommutative Hopf algebras in a braided monoidal category was studied in \cite{Fem}.

The above definition of $2$-cocycles over $B$ in $\cB$ is a special case of $2$-cocycles on a monoidal category (cf. \cite{PSO}*{Section 2} or \cite{BB}*{Section 4.2}) for a monoidal category of the form $\rcomod{B}(\cB)$. A right $2$-cocycle over $B$ gives rise to a natural transformation $\mu^\sigma\colon \otimes\to \otimes$, where
$$\mu^{\sigma}_{V, W}:=(\ide_{V\otimes W\otimes \sigma})(\ide_V\otimes \Psi_{B,W}\otimes \ide_B)(\delta_V\otimes \delta_W),$$ satisfying
\begin{align}
\mu_{X,Y\otimes Z}(\ide_X\otimes \mu_{Y,Z})=\mu_{X\otimes Y, Z}(\mu_{X,Y}\otimes \ide_Z), &&\mu_{X,I}=\mu_{I,X}=\ide_X.
\end{align}

\end{remark}

\begin{lemma}\label{twistedproduct}
Let $\sigma\colon B\otimes B\to I$ be a convolution-invertible morphism in $\cB$. Then the right twisted product
\begin{equation}
m_\sigma:=(m\otimes \sigma)\Delta_{B\otimes B}
\end{equation}
makes $B$ an algebra object in $\cB$ if an only if $\sigma\in C^2_{\cB}(B,I)$. 

We denote the resulting algebra by $B_\sigma$. It is an algebra in $\lcomod{B}$, with coproduct given by $\Delta$ viewed as a map $B_{\sigma}\to B\otimes B_{\sigma}$.
\end{lemma}
\begin{proof}
It follows evidently (under use of coassociativity, the bialgebra condition of $B$, and naturality of the braiding) that given Equation (\ref{2cocycleinB}), $m_\sigma$ defines an associative product. Conversely, considering the equation \begin{align*}
\varepsilon m_\sigma(m_\sigma\otimes\ide_B)=\varepsilon m_\sigma(\ide_B\otimes m_\sigma)
\end{align*}
recovers  condition (\ref{2cocycleinB}). Equation (\ref{2cocyclenorm}) is equivalent to unitarity of the product. Similarly, it follows that $\Delta\colon B_{\sigma}\to B\otimes B_{\sigma}$ gives a coaction.
\end{proof}

The following construction adapts the idea of coboundaries  (see e.g. \cite{KS}*{Section 10.2.3}, or dually, \cite{Maj4}*{Proposition 6.2}) to this setup.

\begin{lemma}\label{coboundarylemma}
Given a convolution-invertible morphism $\beta\colon B\to I$ such that $\beta 1=\ide_I$. Then
$$\partial \beta=\beta^{-\ast} m(\ide_B\otimes \beta\otimes \ide_B\otimes \beta)(\Delta\otimes \Delta)=\vcenter{\hbox{
\begingroup%
  \makeatletter%
  \providecommand\color[2][]{%
    \errmessage{(Inkscape) Color is used for the text in Inkscape, but the package 'color.sty' is not loaded}%
    \renewcommand\color[2][]{}%
  }%
  \providecommand\transparent[1]{%
    \errmessage{(Inkscape) Transparency is used (non-zero) for the text in Inkscape, but the package 'transparent.sty' is not loaded}%
    \renewcommand\transparent[1]{}%
  }%
  \providecommand\rotatebox[2]{#2}%
  \ifx\svgwidth\undefined%
    \setlength{\unitlength}{57.72607162bp}%
    \ifx\svgscale\undefined%
      \relax%
    \else%
      \setlength{\unitlength}{\unitlength * \real{\svgscale}}%
    \fi%
  \else%
    \setlength{\unitlength}{\svgwidth}%
  \fi%
  \global\let\svgwidth\undefined%
  \global\let\svgscale\undefined%
  \makeatother%
  \begin{picture}(1,0.87125918)%
    \put(0,0){\includegraphics[width=\unitlength,page=1]{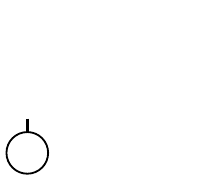}}%
    \put(-0.14915313,-0.03258773){\color[rgb]{0,0,0}\makebox(0,0)[lb]{\smash{}}}%
    \put(0.08471,0.05835905){\color[rgb]{0,0,0}\makebox(0,0)[lb]{\smash{$\beta^{-\ast}$}}}%
    \put(0,0){\includegraphics[width=\unitlength,page=2]{coboundary.pdf}}%
    \put(0.60672625,0.54748502){\color[rgb]{0,0,0}\makebox(0,0)[lb]{\smash{$\beta$}}}%
    \put(0,0){\includegraphics[width=\unitlength,page=3]{coboundary.pdf}}%
    \put(0.20164187,0.55258912){\color[rgb]{0,0,0}\makebox(0,0)[lb]{\smash{$\beta$}}}%
  \end{picture}%
\endgroup%
}}$$
gives a $2$-cocycle over $B$. 
More generally, given $\sigma\in C^2_{\cB}(B,I)$, we can define a new $2$-cocycle $\sigma^\beta$ by
\begin{equation}
\sigma^\beta=\beta^{-\ast}(m\otimes \sigma)(\ide_B\otimes \Psi_{B,B}\otimes \ide_B)(\Delta\otimes \beta\otimes\Delta\otimes \beta)(\Delta\otimes \Delta)=\vcenter{\hbox{
\begingroup%
  \makeatletter%
  \providecommand\color[2][]{%
    \errmessage{(Inkscape) Color is used for the text in Inkscape, but the package 'color.sty' is not loaded}%
    \renewcommand\color[2][]{}%
  }%
  \providecommand\transparent[1]{%
    \errmessage{(Inkscape) Transparency is used (non-zero) for the text in Inkscape, but the package 'transparent.sty' is not loaded}%
    \renewcommand\transparent[1]{}%
  }%
  \providecommand\rotatebox[2]{#2}%
  \ifx\svgwidth\undefined%
    \setlength{\unitlength}{68.97613162bp}%
    \ifx\svgscale\undefined%
      \relax%
    \else%
      \setlength{\unitlength}{\unitlength * \real{\svgscale}}%
    \fi%
  \else%
    \setlength{\unitlength}{\svgwidth}%
  \fi%
  \global\let\svgwidth\undefined%
  \global\let\svgscale\undefined%
  \makeatother%
  \begin{picture}(1,0.8378894)%
    \put(-1.15779129,-0.08163928){\color[rgb]{0,0,0}\makebox(0,0)[lb]{\smash{}}}%
    \put(0,0){\includegraphics[width=\unitlength,page=1]{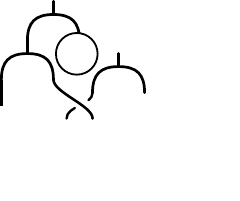}}%
    \put(0.27748806,0.57119451){\color[rgb]{0,0,0}\makebox(0,0)[lb]{\smash{$\beta$}}}%
    \put(0,0){\includegraphics[width=\unitlength,page=2]{coboundary2.pdf}}%
    \put(0.6708695,0.56692289){\color[rgb]{0,0,0}\makebox(0,0)[lb]{\smash{$\beta$}}}%
    \put(0,0){\includegraphics[width=\unitlength,page=3]{coboundary2.pdf}}%
    \put(0.45118404,0.11598363){\color[rgb]{0,0,0}\makebox(0,0)[lb]{\smash{$\sigma$}}}%
    \put(0,0){\includegraphics[width=\unitlength,page=4]{coboundary2.pdf}}%
    \put(0.04759426,0.03796731){\color[rgb]{0,0,0}\makebox(0,0)[lb]{\smash{$\beta^{-\ast}$}}}%
  \end{picture}%
\endgroup%
}}.
\end{equation}
\end{lemma}
\begin{proof}
To verify Equation (\ref{2cocycleinB}) for $\partial \beta$ is an exercise in graphical calculus. Note that both sides of the equation simplify to 
$$\beta^{-\ast}m(\ide\otimes m)(\ide_B\otimes \beta\otimes\ide_B\otimes \beta\otimes\ide_B\otimes \beta)(\Delta\otimes \Delta\otimes \Delta).$$

Similar methods are also key in checking that, more generally, $\sigma^\beta$ is a $2$-cocycle. We include this computation, which can be performed using graphical calculus:

\begin{align*}
\sigma^\beta(m\otimes \sigma^\beta\otimes \ide_B)(\Delta_{B\otimes B}\otimes \ide_B)
&=(\beta^{-\ast}m\otimes \sigma)(\ide_B\otimes \Psi_{B,B}\otimes \ide_B)(m\otimes m\otimes \sigma \otimes \Delta)\\&~\phantom{=}~(\Delta_{B\otimes B}\otimes (\ide\otimes \beta)\Delta\otimes (\ide\otimes \beta)\Delta\otimes (\ide\otimes \beta)\Delta)(\Delta_{B\otimes B}\otimes \ide_B)\\
&=\sigma(\beta^{-\ast}\otimes m\otimes \sigma\otimes \ide_B)(m \otimes \Delta_{B\otimes B}\otimes \ide_B)\\&~\phantom{=}~(m\otimes \ide\otimes (\ide\otimes \beta)\Delta\otimes(\ide\otimes \beta)\Delta\otimes (\ide\otimes \beta)\Delta)\Delta_{B\otimes B\otimes B}\\
&=\sigma(\beta^{-\ast}\otimes \ide_B\otimes m\otimes \sigma)(m \otimes \ide_B\otimes \Delta_{B\otimes B})\\&~\phantom{=}~(\ide_B\otimes m\otimes (\ide\otimes \beta)\Delta\otimes(\ide\otimes \beta)\Delta\otimes (\ide\otimes \beta)\Delta)\Delta_{B\otimes B\otimes B}\\
&=(\beta^{-\ast}m\otimes \sigma)(\ide_B\otimes m\otimes \ide_B\otimes m\otimes \sigma)\\&~\phantom{=}~(\Delta_{B\otimes B\otimes B}\otimes (\ide\otimes \beta)\Delta\otimes (\ide\otimes \beta)\Delta)((\ide\otimes \beta)\Delta\otimes \Delta_{B\otimes B})\\
&=\sigma^\beta(\ide_B\otimes m\otimes \sigma^\beta)(\ide_B\otimes \Delta_{B\otimes B}),
\end{align*}
where in the first equality the definition of $\sigma^\beta$, the bialgebra condition and (co)commutativity, as well as $(\beta^{-\ast}m)\ast (\beta m)=\varepsilon_{B^{\otimes 4}}$ are used. The second equality uses  (co)commutativity again, followed by application of the $2$-cocycle condition (\ref{2cocycleinB}) of $\sigma$ in the third equality. The fourth identity follows using (co)associativity, while the last one uses again that  $(\beta^{-\ast}m)\ast (\beta m)=\varepsilon_{B^{\otimes 4}}$, together with the counit axioms, the bialgebra condition, and the definition of $\sigma^\beta$.
\end{proof}

\begin{definition}
We define $2$-cocycles $\sigma$ and $\tau$ in $C^2_{\cB}(B,I)$ to be \emph{cohomologous}, and write $\sigma\sim \tau$, if there exists a convolution-invertible map $\beta$ as in Lemma \ref{coboundarylemma} such that $\tau=\sigma^\beta$. This gives an equivalence relation $\sim$ on $C^2_{\cB}(B,I)$. We denote the quotient modulo $\sim$ by $H^2_{\cB}(B,I)$ and refer to it as the \emph{2nd bialgebra cohomology space} of $B$.
\end{definition}

The space $H^2_{\cB}(B^*,I)$ is called \emph{non-abelian cohomology} of $B$ in \cite{Maj4}*{Section~6}. In the case where $B$ is cocommutative, the sets $C^2_{\cB}(B,I)$ and $H^2_{\cB}(B,I)$ have a group structure and generalize Sweedler's bialgebra  cohomology with trivial coefficient algebra \cite{Swe}. 

\begin{lemma}\label{cohomiso}
Let $\sigma$ and $\tau$ be cohomologous $2$-cocycles over $B$. Then \begin{align*}
\ide_B\ast \beta\colon B_{\tau}\longrightarrow B_{\sigma}
\end{align*}
is an isomorphism of $B$-comodule algebras.
\end{lemma}
\begin{proof}
Note that $\tau=\sigma^\beta$ is equivalent to the equation
\begin{equation}
\sigma((\ide\otimes \beta)\Delta\otimes(\ide\otimes \beta)\Delta)=(\beta^{-\ast}m\otimes \tau)\Delta_{B\otimes B}.
\end{equation}
For variation of methods, we include the proof that $\ide_B\ast \beta$ is a morphism of algebras using graphical calculus in Fig. \ref{Fig3}.
\begin{figure}[bt]
\[\vcenter{\hbox{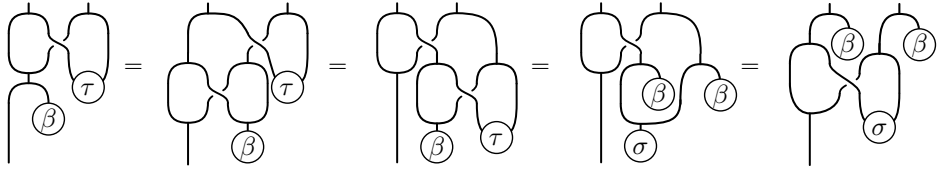}}\]
\caption{Proof that $\ide_B\ast \beta$ is a morphism of algebras}\label{Fig3}
\end{figure}
Since $\beta$ is convolution invertible, $\ide\ast\beta$ is invertible, with inverse $\ide\ast \beta^{-\ast}$. If is further clear by coassociativity that $\ide\ast\beta$ is a morphism of $B$-comodules for the regular $B$-coaction in source and target.
\end{proof}

The following proposition explains the connection of $2$-cocycles over a braided Hopf algebra $B$ with $2$-cocycles over its \emph{bosonization} (or \emph{Radford biproduct}, \cites{Rad}) $B\rtimes H$ . This recovers part of \cite{CP}*{Theorem~4.4(i)}, where $B$ is a Hopf algebra object in $\lYD{H}$.

\begin{proposition}
Let $B$ be a bialgebra in $\lmod{H}$ for $H$ a quasitriangular bialgebra over $\Bbbk$. Then there exists an injective morphism 
$$\sigma \rtimes H\colon C^2_{H}(B,\Bbbk)\longrightarrow C^2(B\rtimes H,\Bbbk),$$
mapping $\sigma\in C^2_{H}(B,\Bbbk)$ to the right $2$-cocycle given by
\begin{align}
\sigma \rtimes H(bh\otimes cg)&=\sigma(b,h\triangleright c)\varepsilon(g).
\end{align}
This map descents to one of bialgebra cohomology spaces
$$\sigma \rtimes H\colon H^2_{H}(B,\Bbbk)\longrightarrow H^2(B\rtimes H,\Bbbk).$$
\end{proposition}
\begin{proof}
We shall use the notation $\Delta(b)=b_{[1]}\otimes b_{[2]}$ to denote the coproduct of $B\rtimes H$.
Hence
\begin{align*}
b_{[1]}\otimes b_{[2]}=b_{(1)}R^{(1)}\otimes (R^{(1)}\triangleright b_{(2)}), &&\forall b\in B.
\end{align*}
We observe that for $b,c\in B$ we have
\begin{align*}
b_{[1]}c_{[1]}\sigma(b_{[2]},c_{[2]})&=b_{(1)}R_1^{(2)}c_{(1)}R_2^{(2)}\sigma(R_1^{(1)}\triangleright b_{(2)}, R_2^{(1)}\triangleright c_{(2)})\\&=b_{(1)}({R_1^{(2)}}_{(1)}\triangleright c_{(1)}){R_1^{(2)}}_{(2)}R_2^{(2)}\sigma(R_1^{(1)}\triangleright b_{(2)}, R_2^{(1)}\triangleright c_{(2)})\\
&=b_{(1)}(R_1^{(2)}\triangleright c_{(1)})R_3^{(2)}R_2^{(2)}\sigma(R_1^{(1)}R_3^{(1)}\triangleright b_{(2)}, R_2^{(1)}\triangleright c_{(2)})\\
&=b_{(1)}(R_1^{(2)}\triangleright c_{(1)})R_2^{(2)}\sigma(R_1^{(1)}{R_2^{(1)}}_{(1)}\triangleright b_{(2)}, {R_2^{(1)}}_{(2)}\triangleright c_{(2)})\\
&=b_{(1)}(R_1^{(2)}\triangleright c_{(1)})\sigma(R_1^{(1)}\triangleright b_{(2)}, c_{(2)})=m_{\sigma}(b\otimes c).
\end{align*}
where the semidirect product relations are applied, followed by the R-matrix relation Equation (\ref{Rmatrix2}) splitting $R_1$ up into $R_1$ and $R_3$, and next Equation (\ref{Rmatrix1}) summarizing $R_3$ and $R_2$ into just $R_2$; and in the penultimate equality we use that $\sigma$ is a morphism of $H$-modules and that $\varepsilon(R_2^{(1)})R_2^{(2)}=1$. The last equality is the definition of the right twisted product from Lemma \ref{twistedproduct}.

This calculation shows that the product of $B_\sigma \rtimes H$ is in fact the product given by twisting $B\rtimes H$ by the map $\sigma\rtimes H$ which takes values $\sigma\rtimes H(b\otimes c)=\sigma(b\otimes c)$ for $b,c\in B$. We can extend to $2$-cocycles defined on all of $B\rtimes H$ via
\begin{align*}\sigma\rtimes H(bh\otimes cg)= \sigma(b,h\triangleright c)\otimes \varepsilon(g).
\end{align*}
Then the product of $B_\sigma\rtimes H$ is obtained by right twist by $\sigma\rtimes H$. However, $B_\sigma\rtimes H$ is an associative product, so we conclude by Lemma \ref{twistedproduct} that $\sigma \rtimes H$ is an element of $C^2(B\rtimes H,\Bbbk)$. It is clear from construction that the mapping $\sigma \mapsto\sigma \rtimes H$ is injective. Further, it follows from a similar  computation as above that, mapping $\beta$ to $\beta\otimes \varepsilon_H$, the construction commutes with coboundary twist, and we thus obtain a map on the quotients
$H^2_H(B,\Bbbk)\to H^2(B\rtimes H,\Bbbk)$.
\end{proof}

\subsection{Induced 2-Cocycles over the Braided Drinfeld Double}\label{cleftbraided}

With the preliminary considerations from the previous sections, we can now generalize the results from the Section \ref{drin2co} to braided Drinfeld doubles.

\begin{corollary}
\label{induced2cocycles2} Let $C,B$ be weakly dual bialgebras (or Hopf algebras) in $\cB=\lmod{H}$. Then there exists two maps
\begin{align*}
\op{Ind}_B\colon &H_H^2(B,\Bbbk)\longrightarrow H^2(\Drin_{H}(C,B),\Bbbk), & \sigma \longmapsto \op{Ind}_B\sigma,\\
\op{Ind}_C\colon &H_H^2(\leftexp{\cop}{C},\Bbbk)\longrightarrow H^2(\Drin_{H}(C,B),\Bbbk), & \tau \longmapsto \op{Ind}_C\tau,
\end{align*}
where the $2$-cocycle $\op{Ind}_B\sigma$ is defined by requiring that
\begin{align}
\op{Ind}_B\sigma(b,b')&= \sigma(b,b'), & \op{Ind}_\Delta\sigma(c,c')&= \varepsilon(c)\varepsilon(c'), &\op{Ind}_\Delta\sigma(h,h')&= \varepsilon(h)\varepsilon(h'), 
\end{align}
for any elements $b,b'\in B$, $c,c'\in C$, $h,h'\in H$. The $2$-cocycle $\op{Ind}_B\tau$ is defined by requiring that 
\begin{align}
\op{Ind}_C\tau(b,b')&=\varepsilon(b)\varepsilon(b'), & \op{Ind}_C\tau(c,c')&=\tau(c,c'), &\op{Ind}_\Delta\sigma(h,h')&= \varepsilon(h)\varepsilon(h').
\end{align}

If $C,B$ are Hopf algebras, then $\Drin_{H}(C,B)$ is defined on $B\otimes H\otimes C$, on which the $2$-cocycle is given by
\begin{align}\label{inducedcocycleB2}
\op{Ind}_B\sigma(bhc,b'h'c')&=\sigma(b,h\triangleright b'_{(2)})\ev(c,b'_{(1)})\varepsilon(h')\varepsilon(c').
\end{align}
Similarly,
\begin{align}\label{inducedcocycleC3}
\op{Ind}_C\tau(chb,c'h'b')&=\tau(c,h{R^{(2)}}_{(1)}\triangleright c'_{(1)})\ev({R^{(2)}}_{(2)}\triangleright c'_{(2)},R^{(1)}\triangleright S(b))\varepsilon(h')\varepsilon(b').
\end{align}
\end{corollary}

This corollary is a consequence of Corollaries \ref{comodulealgebracor} and \ref{comodulealgebracor2} but also a special case of the following Theorem.

\begin{theorem}\label{general2cocycles}
Let $C, B$ be weakly dual Hopf algebras in $\lmod{H}$. Given $\sigma \in C^2_{H}(B,\Bbbk)$ and $\tau \in C^2_{H}(\leftexp{\cop}{C},\Bbbk)$, then $\sigma\circ \tau\in C^2(\Drin_H(C,B), \Bbbk)$, where
\begin{align*}
\sigma\circ \tau(b,b')=\sigma(b,b'),&&
\sigma\circ \tau(c,c')=\tau(c,c'),&&\sigma\circ \tau(h,h')=\varepsilon(hh').
\end{align*}
That is, for all $b,b'\in B, h,h'\in H$, and $c,c'\in C$,
\begin{align}
\sigma\circ \tau(bhc,b'h'c')=\sigma(b,hR^{-(1)}\triangleright b'_{(2)})\ev(c_{(2)}, b'_{(1)})\tau(R^{-(2)}\triangleright c_{(1)},h'\triangleright c').
\end{align}
This construction commutes with coboundaries. 
\end{theorem}
\begin{proof}
We give a more conceptual proof than in Theorem \ref{2cocyclethm}. First, define a category $\cT_{\tau}^\sigma$ consisting of objects $V$ in $\cB=\lmod{H}$ which are both a left $B_\sigma$-module with action morphism  $b_V$ and a left $\leftexp{{\cop}}{C}_{\tau}$-modules in $\lmod{H}$ with action morphism $c_V$ satisfying the compatibility condition
\begin{align}
\vcenter{\hbox{
\begingroup%
  \makeatletter%
  \providecommand\color[2][]{%
    \errmessage{(Inkscape) Color is used for the text in Inkscape, but the package 'color.sty' is not loaded}%
    \renewcommand\color[2][]{}%
  }%
  \providecommand\transparent[1]{%
    \errmessage{(Inkscape) Transparency is used (non-zero) for the text in Inkscape, but the package 'transparent.sty' is not loaded}%
    \renewcommand\transparent[1]{}%
  }%
  \providecommand\rotatebox[2]{#2}%
  \ifx\svgwidth\undefined%
    \setlength{\unitlength}{176.01276535bp}%
    \ifx\svgscale\undefined%
      \relax%
    \else%
      \setlength{\unitlength}{\unitlength * \real{\svgscale}}%
    \fi%
  \else%
    \setlength{\unitlength}{\svgwidth}%
  \fi%
  \global\let\svgwidth\undefined%
  \global\let\svgscale\undefined%
  \makeatother%
  \begin{picture}(1,0.60554072)%
    \put(0.10460754,0.4310328){\color[rgb]{0,0,0}\makebox(0,0)[lb]{\smash{ }}}%
    \put(0,0){\includegraphics[width=\unitlength,page=1]{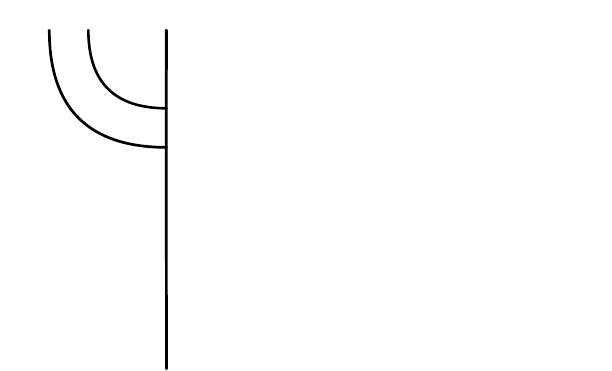}}%
    \put(0.3277319,0.30082591){\color[rgb]{0,0,0}\makebox(0,0)[lb]{\smash{$=$}}}%
    \put(0,0){\includegraphics[width=\unitlength,page=2]{Tsigmatau3.pdf}}%
    \put(0.76218938,0.19489849){\color[rgb]{0,0,0}\makebox(0,0)[lb]{\smash{$\tau$}}}%
    \put(0,0){\includegraphics[width=\unitlength,page=3]{Tsigmatau3.pdf}}%
    \put(0.59194927,0.19932105){\color[rgb]{0,0,0}\makebox(0,0)[lb]{\smash{$\sigma$}}}%
    \put(0,0){\includegraphics[width=\unitlength,page=4]{Tsigmatau3.pdf}}%
    \put(-0.00282961,0.57779393){\color[rgb]{0,0,0}\makebox(0,0)[lb]{\smash{$B$}}}%
    \put(0.06108621,0.57748913){\color[rgb]{0,0,0}\makebox(0,0)[lb]{\smash{$C$}}}%
    \put(0.12515339,0.57748913){\color[rgb]{0,0,0}\makebox(0,0)[lb]{\smash{$B$}}}%
    \put(0.18876649,0.57718432){\color[rgb]{0,0,0}\makebox(0,0)[lb]{\smash{$C$}}}%
    \put(0.46330698,0.5796415){\color[rgb]{0,0,0}\makebox(0,0)[lb]{\smash{$B$}}}%
    \put(0.59539967,0.57933669){\color[rgb]{0,0,0}\makebox(0,0)[lb]{\smash{$C$}}}%
    \put(0.76259743,0.57933669){\color[rgb]{0,0,0}\makebox(0,0)[lb]{\smash{$B$}}}%
    \put(0.8943874,0.57903189){\color[rgb]{0,0,0}\makebox(0,0)[lb]{\smash{$C$}}}%
  \end{picture}%
\endgroup%
}}~~.
\end{align}
This corresponds to the following equality of morphisms:
\begin{align}\begin{split}
&c_V(\ide_C\otimes b_V)(\ide\otimes c_V)(\ide\otimes b_V)\\&=b_V(\ide_B\otimes c_V)(m\otimes \sigma\otimes \tau\otimes m\otimes \ide_V)(\Delta_{B\otimes B}\otimes \Delta_{B\otimes C}\otimes \ide_V)\\&~\phantom{=}~(\ide\otimes \Psi^{-1}_{C,B}\otimes \ide)(\ide\otimes \ev \otimes \ide)(\ide_B\otimes \Delta_B\otimes\Delta_C\otimes \ide_{B\otimes V})
\end{split}
\end{align}
This condition is equivalent (under use of the antipode axioms) to the condition
\begin{align}\label{Tsigmatau2}
\vcenter{\hbox{
\begingroup%
  \makeatletter%
  \providecommand\color[2][]{%
    \errmessage{(Inkscape) Color is used for the text in Inkscape, but the package 'color.sty' is not loaded}%
    \renewcommand\color[2][]{}%
  }%
  \providecommand\transparent[1]{%
    \errmessage{(Inkscape) Transparency is used (non-zero) for the text in Inkscape, but the package 'transparent.sty' is not loaded}%
    \renewcommand\transparent[1]{}%
  }%
  \providecommand\rotatebox[2]{#2}%
  \ifx\svgwidth\undefined%
    \setlength{\unitlength}{160.125024bp}%
    \ifx\svgscale\undefined%
      \relax%
    \else%
      \setlength{\unitlength}{\unitlength * \real{\svgscale}}%
    \fi%
  \else%
    \setlength{\unitlength}{\svgwidth}%
  \fi%
  \global\let\svgwidth\undefined%
  \global\let\svgscale\undefined%
  \makeatother%
  \begin{picture}(1,0.61383974)%
    \put(0.11062559,0.47380024){\color[rgb]{0,0,0}\makebox(0,0)[lb]{\smash{ }}}%
    \put(0,0){\includegraphics[width=\unitlength,page=1]{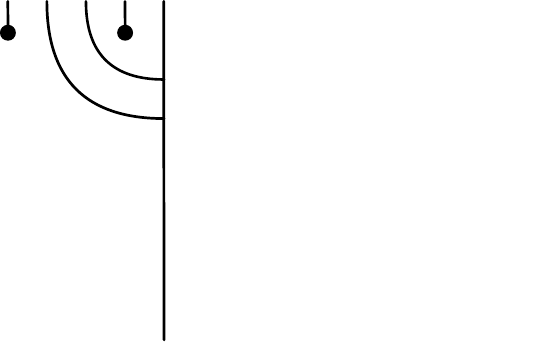}}%
    \put(0.3512047,0.32983763){\color[rgb]{0,0,0}\makebox(0,0)[lb]{\smash{$=$}}}%
    \put(0,0){\includegraphics[width=\unitlength,page=2]{Tsigmatau2.pdf}}%
    \put(0.76324369,0.21423649){\color[rgb]{0,0,0}\makebox(0,0)[lb]{\smash{$\tau$}}}%
    \put(0,0){\includegraphics[width=\unitlength,page=3]{Tsigmatau2.pdf}}%
    \put(0.57611221,0.21909786){\color[rgb]{0,0,0}\makebox(0,0)[lb]{\smash{$\sigma$}}}%
  \end{picture}%
\endgroup%
}}~~.
\end{align}
There are two things to check. First, that the product defined on $B\otimes C$ by
\begin{align}\begin{split}
m_{\sigma\circ\tau}&=(m\otimes \sigma\otimes \tau\otimes m)(\Delta_{B\otimes B}\otimes \Delta_{B\otimes C})(\ide\otimes \Psi^{-1}_{C,B}\otimes \ide)(\ide\otimes \ev \otimes \ide)\\&~\phantom{=}~(\ide_B\otimes \Delta_B\otimes\Delta_C\otimes \ide_{B})
\end{split}
\end{align}
provides an associative algebra object. It then follows that $\cT^\sigma_\tau$ is a category which is equivalent to left modules over the algebra $B\rtimes_{\sigma}^\tau C:=(B\otimes C, m_{\sigma\circ\tau})$. Second, we check that $\cT^{\sigma}_\tau$ is a left categorical module over the category $\leftexp{C,B}{\mathbf{YD}}$ (which was defined in the proof of Proposition \ref{drinfeldprop}). The action is given by the tensor product structure of tensoring such Yetter--Drinfeld modules. The theorem then follows under the equivalences $\leftexp{C,B}{\mathbf{YD}}\simeq \lmod{\Drin_H(C,B)}$ and $\cT^\sigma_\tau\simeq \lmod{(B\rtimes_{\sigma}^\tau C)\rtimes H}$ using the Tannaka-Krein reconstruction argument as in \cite{Lau}*{Proposition 3.8.4}. In particular, the right $2$-cocycle on $\Drin_H(C,B)$ is obtained as $\sigma\circ \tau=\varepsilon_{\Drin_H(C,B)}(m_{\sigma\circ\tau}\rtimes H)$.

We observe that the product of $B\rtimes_\sigma^\tau C$ can be written as  
$$m_{\sigma\circ\tau}=(m_\sigma\otimes m_\tau)(\ide\otimes \Psi^{-1}_{C,B}\otimes \ide)(\ide\otimes \ev \otimes \ide)(\ide_B\otimes \Delta_B\otimes\Delta_C\otimes \ide_{B}),$$
where $m_\tau$ is the left $2$-cocycle twist of the multiplication on $C$ by $\tau$. We denote $m_\tau=\vcenter{\hbox{
\begingroup%
  \makeatletter%
  \providecommand\color[2][]{%
    \errmessage{(Inkscape) Color is used for the text in Inkscape, but the package 'color.sty' is not loaded}%
    \renewcommand\color[2][]{}%
  }%
  \providecommand\transparent[1]{%
    \errmessage{(Inkscape) Transparency is used (non-zero) for the text in Inkscape, but the package 'transparent.sty' is not loaded}%
    \renewcommand\transparent[1]{}%
  }%
  \providecommand\rotatebox[2]{#2}%
  \ifx\svgwidth\undefined%
    \setlength{\unitlength}{32.94062858bp}%
    \ifx\svgscale\undefined%
      \relax%
    \else%
      \setlength{\unitlength}{\unitlength * \real{\svgscale}}%
    \fi%
  \else%
    \setlength{\unitlength}{\svgwidth}%
  \fi%
  \global\let\svgwidth\undefined%
  \global\let\svgscale\undefined%
  \makeatother%
  \begin{picture}(1,0.59412856)%
    \put(0,0){\includegraphics[width=\unitlength,page=1]{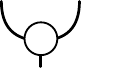}}%
    \put(0.27859823,0.15136592){\color[rgb]{0,0,0}\makebox(0,0)[lb]{\smash{$\tau$}}}%
  \end{picture}%
\endgroup%
}}$. We now check that $m_{\sigma\circ\tau}$ is associative, by comparing $m_{\sigma\circ\tau}(m_{\sigma\circ\tau}\otimes \ide_{B\otimes C})$ and $m_{\sigma\circ\tau}(\ide_{B\otimes C}\otimes m_{\sigma\circ\tau})$ in Fig. \ref{Fig4}.
\begin{figure}[tb]
\begin{align*}
\vcenter{\hbox{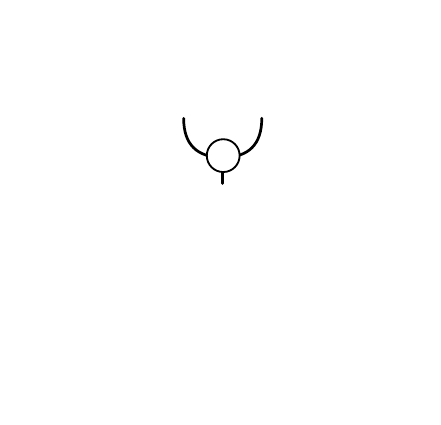}}&=\vcenter{\hbox{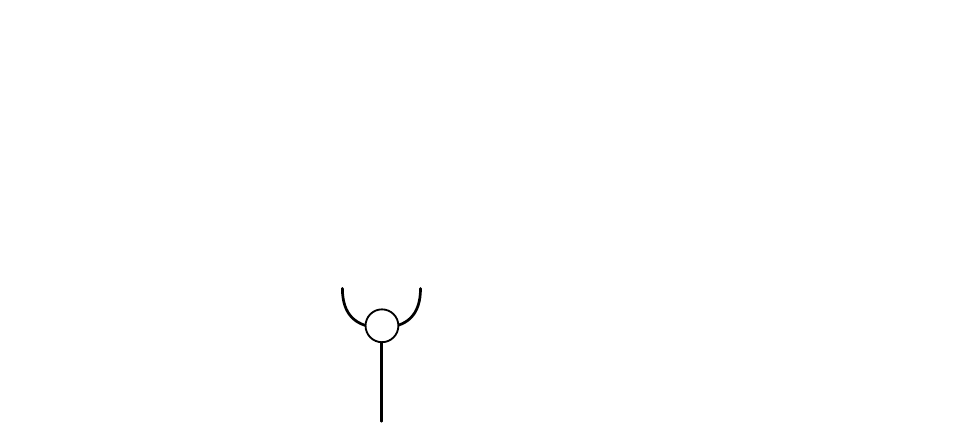}}\\
&=\vcenter{\hbox{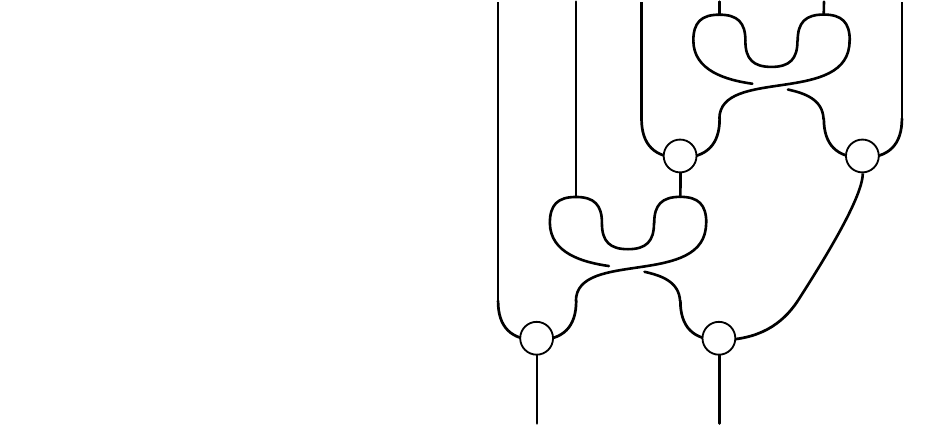}}~.
\end{align*}
\caption{Proof that $m_{\sigma\circ\tau}$ is associative}
\label{Fig4}
\end{figure}
The first picture equals $m_{\sigma\circ\tau}(m_{\sigma\circ\tau}\otimes \ide_{B\otimes C})$, and in the first equality we apply that $\Delta_C\colon C_\tau\to C_\tau\otimes C$ makes $C_\tau$ a right $C$-comodule algebra in $\cB$. In the second equality, coassociativity of $m_\tau$ and the weak Hopf algebra duality of $C$ and $B$ are applied, cf. Equation (\ref{dualhopfcond}). The third equality applies properties of the braiding and (co)associativity. Finally, the forth equality uses that $\Delta_B\colon B_\sigma\to B\otimes B_\sigma$ makes $B_\sigma$ a left $B$-comodule algebra, together with Equation (\ref{dualhopfcond}). The last picture indeed equals $m_{\sigma\circ\tau}(\ide_{B\otimes C}\otimes m_{\sigma\circ\tau})$.

Second, we verify that that we can tensor an object in $\cT^{\sigma}_\tau$ on the left with a Yetter--Drinfeld module in $\leftexp{C,B}{\mathbf{YD}}$ using graphical calculus in Fig. \ref{Fig5}.

\begin{figure}[bt]
\begin{align*}
\vcenter{\hbox{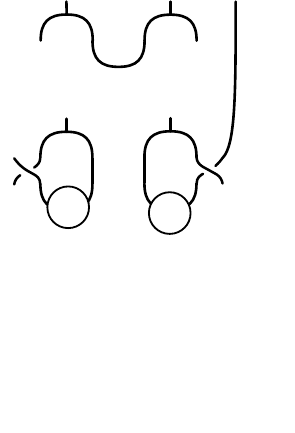}}&=\vcenter{\hbox{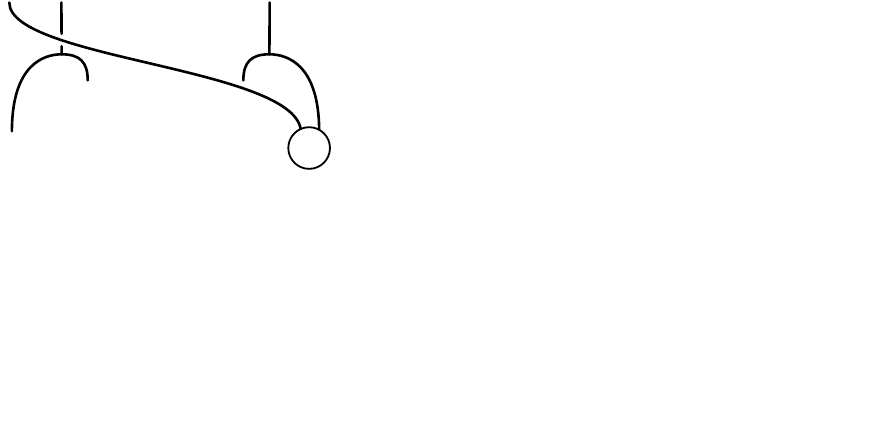}}\\
&=\vcenter{\hbox{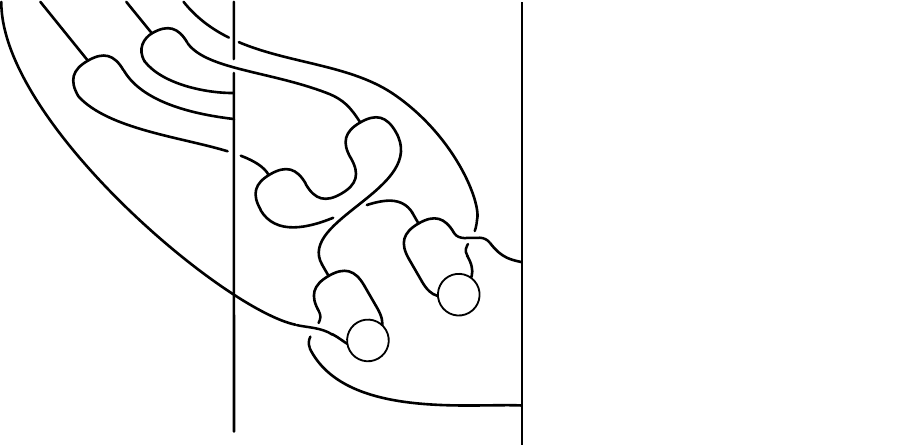}}~~.
\end{align*}
\caption{Proof of tensor product action on $\cT^{\sigma}_{\tau}$}
\label{Fig5}
\end{figure}

In this computation, the Yetter--Drinfeld condition (\ref{drinfeldcomp}) is applied in the second equality, while the defining condition (\ref{Tsigmatau2}) of $\cT_\tau^\sigma$ is applied in the last equality. The other steps use naturality of the braiding, combined with (co)associativity.

Finally, it follows, using the description of the product on $B\rtimes_\sigma^\tau C$ that if $\sigma\sim\sigma'$ and $\tau\sim\tau'$, then $B\rtimes_\sigma^\tau C\cong B\rtimes_{\sigma'}^{\tau'} C$. This implies that the map of $2$-cocycles commutes with coboundaries, and hence descents to a map
$H^2_H(B,\Bbbk)\times H^2_H(\leftexp{\cop}{C},\Bbbk)\longrightarrow H^2(\Drin_H(C,B),\Bbbk)$.
\end{proof}

\begin{example}\label{heistwist}
The braided Heisenberg double $\Heis_H(C,B)$ is obtained from $\Drin_H(C,B)$ via twist with the $2$-cocycle
\begin{align*}
\op{Ind}_B\triv(bhc,b'h'c')&=\varepsilon(b)\varepsilon(h)\ev(h_{(1)}\triangleright c,h_{(2)}\triangleright b')\varepsilon(h')\varepsilon(c'),
\end{align*}
which is induced by the trivial $2$-cocycle $\triv=\varepsilon\otimes \varepsilon$ over $B$. This results recovers \cite{Lau}*{Corollary 3.8.5}, cf. also Example \ref{Luexample}.
\end{example}

\begin{remark}
Theorem \ref{general2cocycles} gives a map
$$H^2_H(B,\Bbbk)\times H^2_H(\leftexp{\cop}{C},\Bbbk)\longrightarrow H^2(\Drin_H(C,B),\Bbbk).$$
Note that this map is in general not surjective. For example, the trivial cocycle on $\Drin_H(C,B)$ does not lay in the image unless in trivial cases, or the highly degenerate pairing $\ev=\varepsilon_C
\otimes\varepsilon_B$.

Note further that there exists a map
$$H_2^H(B,\Bbbk)\longrightarrow H^2_H(\leftexp{\cop}{C},\Bbbk),$$
where $H_2^H(B,\Bbbk)$ denotes the set of right $2$-cycles over $B$ in $\lmod{H}$ up to boundary transformations. That is, equivalence classes of elements $c\in B\otimes B$ such that 
\begin{align}
\vcenter{\hbox{
\begingroup%
  \makeatletter%
  \providecommand\color[2][]{%
    \errmessage{(Inkscape) Color is used for the text in Inkscape, but the package 'color.sty' is not loaded}%
    \renewcommand\color[2][]{}%
  }%
  \providecommand\transparent[1]{%
    \errmessage{(Inkscape) Transparency is used (non-zero) for the text in Inkscape, but the package 'transparent.sty' is not loaded}%
    \renewcommand\transparent[1]{}%
  }%
  \providecommand\rotatebox[2]{#2}%
  \ifx\svgwidth\undefined%
    \setlength{\unitlength}{113.26496862bp}%
    \ifx\svgscale\undefined%
      \relax%
    \else%
      \setlength{\unitlength}{\unitlength * \real{\svgscale}}%
    \fi%
  \else%
    \setlength{\unitlength}{\svgwidth}%
  \fi%
  \global\let\svgwidth\undefined%
  \global\let\svgscale\undefined%
  \makeatother%
  \begin{picture}(1,0.47215577)%
    \put(-0.30790599,-0.05589341){\color[rgb]{0,0,0}\makebox(0,0)[lb]{\smash{}}}%
    \put(0,0){\includegraphics[width=\unitlength,page=1]{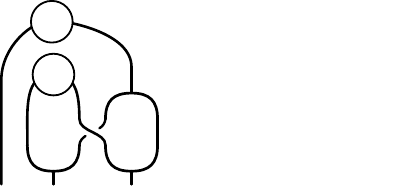}}%
    \put(0.44885301,0.2023507){\color[rgb]{0,0,0}\makebox(0,0)[lb]{\smash{$=$}}}%
    \put(0,0){\includegraphics[width=\unitlength,page=2]{2cycle2.pdf}}%
    \put(0.09672357,0.39154034){\color[rgb]{0,0,0}\makebox(0,0)[lb]{\smash{$c$}}}%
    \put(0.10027092,0.25437781){\color[rgb]{0,0,0}\makebox(0,0)[lb]{\smash{$c$}}}%
    \put(0.62645488,0.25556016){\color[rgb]{0,0,0}\makebox(0,0)[lb]{\smash{$c$}}}%
    \put(0.88895553,0.32532393){\color[rgb]{0,0,0}\makebox(0,0)[lb]{\smash{$c$}}}%
  \end{picture}%
\endgroup%
}}~.
\end{align}
The map is given by mapping $c$ to $c^*= \ev^{\otimes 2}(\ide_{C\otimes C}\otimes c)$. Hence Theorem \ref{general2cocycles}, in particular, produces a map
$$H^2_H(B,\Bbbk)\times H_2^H(B,\Bbbk)\longrightarrow H^2(\Drin_H(C,B),\Bbbk), \quad (\sigma,c)\longmapsto \sigma\circ c^*.$$
\end{remark}

We remark that \cite{Sch2}*{Section 4} classifies all cleft objects over $u_q(\fr{g})$ by viewing this Hopf algebra as a quotient of a bicross product.

\subsection*{Acknowledgments}

The author thanks Kobi Kremnizer and Shahn Majid for interesting and helpful discussions on the subject of this paper. Early parts of this research were supported by an EPSRC Doctoral Prize at the University of East Anglia. The author also thanks Florin Panaite for hints to further references and a helpful discussion, and the anonymous referee for helpful feedback and suggestions.

Graphics are created using Inkscape.

\bibliography{biblio}
\bibliographystyle{amsrefs}

\end{document}